\documentclass[11pt]{article}

\usepackage{epsfig}
\usepackage{amssymb, latexsym}
\usepackage{amscd}
\usepackage[all]{xy}
\usepackage{epsf}
\usepackage[mathscr]{eucal}
\usepackage{amsmath,amsfonts,amscd,amssymb,amsthm}
\usepackage{appendix}

\usepackage[osf]{mathpazo}
\usepackage[T1]{fontenc}
\DeclareSymbolFont{otone}{OT1}{cmr}{m}{n}
\DeclareMathSymbol{\Gamma}{\mathalpha}{otone}{0}
\DeclareMathSymbol{\Delta}{\mathalpha}{otone}{1}
\DeclareMathSymbol{\Theta}{\mathalpha}{otone}{2}
\DeclareMathSymbol{\Lambda}{\mathalpha}{otone}{3}
\DeclareMathSymbol{\Xi}{\mathalpha}{otone}{4}
\DeclareMathSymbol{\Pi}{\mathalpha}{otone}{5}
\DeclareMathSymbol{\Sigma}{\mathalpha}{otone}{6}
\DeclareMathSymbol{\Upsilon}{\mathalpha}{otone}{7}
\DeclareMathSymbol{\Phi}{\mathalpha}{otone}{8}
\DeclareMathSymbol{\Psi}{\mathalpha}{otone}{9}
\DeclareMathSymbol{\Omega}{\mathalpha}{otone}{10}

\long\def\comment#1\endcomment{}

\theoremstyle{plain}

\newtheorem{theorem}{{\sc Theorem}}[section]
\newtheorem{lemma}[theorem]{\sc Lemma}

\newtheorem{prop}[theorem]{\sc Proposition}
\newtheorem{coroll}[theorem]{\sc Corollary}
\newtheorem{klemma}[theorem]{\sc Key-lemma}
\newtheorem{sublemma}[theorem]{\sc Sub-lemma}

\newcommand{\appsection}[1]{\let\oldthesection\thesection
  \renewcommand{\thesection}{\sc{Appendix \oldthesection}}
  \section{#1}\let\thesection\oldthesection}

\newcommand{\appsubsection}[1]{\let\oldthesubsection\thesubsection
  \renewcommand{\thesubsection}{\sc\oldthesubsection}
  \subsection{#1}\let\thesubsection\oldthesubsection}

\newcommand{\apptheorem}{
\let\oldthetheorem\thetheorem
\renewcommand{\thetheorem}{\sc \oldthetheorem}
\theorem\let\thetheorem\oldthetheorem}

\renewcommand{\thesection}{\sc \arabic {section}}
\renewcommand{\thetheorem}{\thesection.{\sc \arabic {theorem}}}

\theoremstyle{plain}

\newtheorem{defn}[theorem]{\sc Definition}

\theoremstyle{exercise}
\newtheorem{remark}[theorem]{\sc Remark}

\makeatletter \@addtoreset{equation}{section} \makeatother

\def\eqref#1{\thetag{\ref{#1}}}

\let\latexref=\ref
\def\ref#1{{\normalfont{\latexref{#1}}}}

\newcommand{\udot}{{\:\raisebox{3pt}{\text{\circle*{1.5}}}}}
\newcommand{\mb}{{\:\raisebox{1pt}{\text{\circle*{1.5}}}}}

\def\dlim_#1{{\displaystyle\lim_{#1}}^\hdot}

\newcommand{\Ker}{\operatorname{{\rm Ker}}}

\newcommand{\id}{\operatorname{\rm id}}

\newcommand{\Mor}{\mathrm{Mor}}
\newcommand{\Ob}{\mathrm{Ob}}

\newcommand{\opp}{\mathrm{opp}}

\newcommand{\Hom}{\mathrm{Hom}}
\newcommand{\Bimod}{\mathscr{B}\mathrm{imod}}

\newcommand{\RHom}{\mathrm{RHom}}

\newcommand{\fin}{\mathrm{fin}}
\newcommand{\Mon}{\boldsymbol{\mathrm{Mon}}}
\newcommand{\dg}{\mathrm{dg}}
\newcommand{\Ho}{\mathrm{Ho}}

\newcommand{\colim}{\mathrm{colim}}
\newcommand{\Id}{\mathrm{Id}}

\newcommand{\mon}{\mathrm{mon}}
\newcommand{\Sets}{\mathbf{Sets}}
\newcommand{\pretr}{\mathrm{pre-tr}}

\newcommand{\nlim}{\textsc{lim}}
\newcommand{\ncolim}{\textsc{colim}}

\def\wtilde#1{\widetilde{#1}\vphantom{#1}}

\title{\sc{Monoidal cofibrant resolutions of dg algebras}}

\author{\sc{Boris Shoikhet}}
\date{}

\begin{document}\maketitle
{\footnotesize
\begin{center}{\parbox{4,5in}{{\sc Abstract.} Let $k$ be a field of any characteristic. In this paper, we construct a functorial cofibrant resolution $\mathfrak{R}(A)$ for the $\mathbb{Z}_{\le 0}$-graded dg algebras $A$ over $k$, such that the functor $A\rightsquigarrow \mathfrak{R}(A)$ is colax-monoidal with quasi-isomorphisms as the colax maps. More precisely, there are maps of bifunctors $\mathfrak{R}(A\otimes B)\to \mathfrak{R}(A)\otimes \mathfrak{R}(B)$, compatible with the projections to $A\otimes B$, and obeying the colax-monoidal axiom.

The main application of such resolution (which we consider in our next paper) is the existence of a colax-monoidal dg localization of pre-triangulated dg categories, such that the localization is a genuine dg category, whose image in the homotopy category of dg categories is isomorphic to the To\"{e}n's dg localization. }}
\end{center}
}

\bigskip
\bigskip

\section*{\sc Introduction}
We formulate our main result (Theorem \ref{theor1_intro} below) in Section \ref{section01}, discuss its applications to higher-monoidal Deligne conjecture in Section \ref{applications}, and overview our proof of Theorem \ref{theor1_intro} in Section \ref{section03}.
\subsection{\sc The Main Theorem}\label{section01}
Let $k$ be a field of any characteristic.

It is known that any associative algebra $A$ over $k$ admits a {\it free resolution}, that is, a free associative dg algebra $\mathfrak{R}(A)$ endowed with a quasi-isomorphism of algebras $p_A\colon \mathfrak{R}(A)\to A$. Here we specify this claim as follows. Having two associative algebras $A$ and $B$ over $k$, $A\otimes B$ is an associative algebra again. We want to choose a functorial free resolution $\mathfrak{R}(A)$ for all algebras $A$ over $k$, such that $\mathfrak{R}(A\otimes B)$ and $\mathfrak{R}(A)\otimes \mathfrak{R}(B)$ are related by a functorial quasi-isomorphism, defined over $A\otimes B$.

More precisely, we want to find a functor $A\rightsquigarrow\mathfrak{R}(A)$ with $\mathfrak{R}(A)$ a free associative dg algebra, such that there exists a morphism of bifunctors
$$
\beta\colon\mathfrak{R}(A\otimes B)\to\mathfrak{R}(A)\otimes\mathfrak{R}(B)
$$
such that for any two algebras $A,B$ the map $\beta(A,B)$ is a {\it quasi-isomorphism}, the diagram
\begin{equation}
\xymatrix{
\mathfrak{R}(A\otimes B)\ar[rr]^{\beta}\ar[rd]_{p_{A\otimes B}}&&\mathfrak{R}(A)\otimes \mathfrak{R}(B)\ar[ld]^{p_A\otimes p_B}\\
&A\otimes B
}
\end{equation}
is commuative, and such that for any three algebras $A,B,C$, the diagram
\begin{equation}\label{newintro_1}
\xymatrix{
\mathfrak{R}(A\otimes B\otimes C)\ar[r]\ar[d]&\mathfrak{R}(A\otimes B)\otimes \mathfrak{R}(C)\ar[d]\\
\mathfrak{R}(A)\otimes\mathfrak{R}(B\otimes C)\ar[r]&\mathfrak{R}(A)\otimes\mathfrak{R}(B)\otimes\mathfrak{R}(C)
}
\end{equation}
is commutative. A functor $F\colon\mathscr{M}_1\to\mathscr{M}_2$ between two monoidal categories, with a map of bifunctors $\beta\colon F(A\otimes B)\to F(A)\otimes F(B)$ for which the diagram \eqref{newintro_1} commutes, is called {\it colax-monoidal} (see Definition \ref{colax_intro} below for the precise definition of a colax-monoidal functor).

More generally, we construct such a resolution $A\rightsquigarrow\mathfrak{R}(A)$ on the category of $\mathbb{Z}_{\le 0}$-graded differential graded algebras. We think that such a functor hardly exists on the category of all $\mathbb{Z}$-graded dg algebras.

Throughout the paper, we use the technique of {\it closed model categories}, introduced by Quillen in [Q].

Roughly speaking, the theory of closed model categories is a non-abelian pattern of classical homological algebra. A closed model category is a category $\mathscr{C}$ with three classes of morphisms, called {\it weak equivalences, fibrations, and cofibrations}. The class of weak equivalences is the most essential. The goal is to define non-abelian analogues of derived functors, defined on the localized by weak equivalences categories. In the abelian setting, the underlying category is the category of complexes, and the weak equivalences are the quasi-isomorphisms of complexes.

In this paper, we assume the reader has some familiarity with the foundations of closed model categories. We use [Hir], [GS], and [DS] as the references.

The closed model structure on the category of associative dg algebras $\mathscr{A}lg_k$ over a field $k$ of any characteristic (and as well, for the category of dg algebras over any operad, over a field of characteristic 0) was constructed by Hinich in [Hi].
For this structure, the weak equivalences are the quasi-isomorphisms of dg algebras, the fibrations are the component-wise surjective maps of dg algebras. The cofibrations are uniquely defined from the weak equivalences and the fibrations, by the {\it left lifting property}.

Recall the description of the cofibrant objects in the Hinich's closed model category of associative dg algebras.

A dg associative algebra $A$ is called {\it standard cofibrant} if there is an increasing filtration on $A$, $A=\colim\ A_i$,
$$
A_0\subset A_1\subset A_2\subset\dots
$$
such that the underlying graded algebra $A_i$ is obtained from $A_{i-1}$ by adding the free generators graded space $V_i$,
$$
A_i=A_{i-1}\ast T(V_i)
$$
such that
$$
d(V_i)\subset A_{i-1}
$$
for $i\ge 1$, and $A_0$ has zero differential.
Here $\ast$ is the free product of algebras, which is the same as the categorical coproduct of algebras, and $T(V_i)$ is the free graded algebra on generators $V_i$.

The general description (valid as well for any operad in characteristic 0) is:
{\it any cofibrant dg associative algebra is a retract of a standard cofibrant dg algebra.}
For the category of associative dg algebras in any characteristic, the class of retracts of standard cofibrant objects coincides with the class of standard cofibrant objects itself.

As follows from this description, when restricted to $\mathbb{Z}_{\le 0}$-graded dg algebras, the concepts of free dg algebras and of cofibrant dg algebras
{\it coincide}. For general $\mathbb{Z}$-graded dg algebras, the two concepts are different: a cofibrant $\mathbb{Z}$-graded dg algebra is free, with the differential of some special ``triangular'' form.

Denote the category of $\mathbb{Z}_{\le 0}$-graded associative algebras by $\mathscr{A}lg_k^{\le 0}$.

Our main result is:
\begin{theorem}\label{theor1_intro}{\itshape
Let $k$ be a field of any characteristic.
There is a functor $\mathfrak{R}\colon \mathscr{A}lg_k^{\le 0}\to \mathscr{A}lg_k^{\le 0}$ and a morphism of functors $w\colon \mathfrak{R}\to Id$ with the following properties:
\begin{itemize}
\item[1.] $\mathfrak{R}(A)$ is cofibrant, and $w\colon \mathfrak{R}(A)\to A$ is a quasi-isomorphism, for any $A\in \mathscr{A}lg_k^{\le 0}$,
\item[2.] there is a colax-monoidal structure on the functor $\mathfrak{R}$, such that all colax-maps $\beta_{A,B}\colon \mathfrak{R}(A\otimes B)\to \mathfrak{R}(A)\otimes \mathfrak{R}(B)$ are quasi-isomorphisms of dg algebras, and such that the diagram
    $$
    \xymatrix{
    \mathfrak{R}(A\otimes B)\ar[rr]^{\beta_{A,B}}\ar[rd]&&\mathfrak{R}(A)\otimes \mathfrak{R}(B)\ar[dl]\\
    &A\otimes B
    }
    $$
    is commutative,
\item[3.] the morphism $w(k^\udot)\colon \mathfrak{R}(k^\udot)\to k^\udot$ coincides with $\alpha\colon \mathfrak{R}(k^\udot)\to k^\udot$, where $\alpha$ is a part of the colax-monoidal structure {\rm (see Definition \ref{colax_intro})}, and $k^\udot=1_{\mathscr{A}lg_k^{\le 0}}$ is the dg algebra equal to the one-dimensional $k$-algebra in degree 0, and vanishing in other degrees.
\end{itemize}
}
\end{theorem}
Note once again, that in the context of this Theorem, the words {\it cofibrant} and {\it free} are synonymous.
In the course of proof we use an ``auxiliary'' closed model category of simplicial associative algebras, where the difference between the free and the cofibrant objects is essential.
For simplicial associative algebras there is fairly analogous (and much simpler) theorem, see Theorem \ref{theor_simpl_alg} below. We formulate Theorem \ref{theor1_intro} with some closed model categories language to have a unified formulation of the both theorems.

\subsection{\sc Applications}\label{applications}
Here we explain our motivations for Theorem \ref{theor1_intro}. The results mentioned here will be proven in our sequel paper(s).

The problem which motivates our interest in Theorem \ref{theor1_intro} is the following {\it generalized Deligne conjecture for $n$-fold monoidal abelian categories}.
Let $\mathscr{M}$ be an abelian $n$-fold monoidal [BFSV] category, with some weak compatibility of the exact and of the monoidal structures, and let $e$ be its unit element. The generalized Deligne conjecture claims that $\RHom^\udot_\mathscr{M}(e,e)$ is a homotopy $(n+1)$-algebra, whose commutative product is quasi-isomorphic to the Yoneda product.

Originally, the name ``Deligne conjecture'' was given (due to a letter of P.Deligne where he announced it) to the particular case of the above claim for $n=1$, when $\mathscr{M}=\Bimod(A)$ for an associative algebra $A$, with $-\otimes_A-$ as the monoidal product. Then the claim is that the Hochschild complex $\RHom^\udot_{\Bimod(A)}(A,A)$ admits a homotopy 2-algebra structure, whose commutative product is quasi-isomorphic to the Yoneda product, and whose Lie bracket of degree -1 is quasi-isomorphic to the Gerstenhaber bracket. At the moment, there are several proofs of this claim, see e.g. [MS], [KS], and [T].

In their paper [KT], Kock and To\"{e}n suggest a very conceptual proof of (a suitable version of) the Deligne conjecture for $n$-fold monoidal categories {\it enriched over simplicial sets}. The methods of loc.cit. can be applied for the $n$-fold monoidal categories enriched over arbitrary {\it cartesian symmetric monoidal category}, that is, enriched over a symmetric monoidal category whose monoidal product of any two objects is isomorphic to their cartesian product. The categories of sets, of simplicial sets, ... give examples of cartesian-monoidal categories. The categories of abelian groups, $k$-vector spaces, complexes of $k$-vector spaces, ... are not cartesian-monoidal, as the cartesian product is the direct sum, and the monoidal product is the tensor product.

The reason why the methods of [KT] do not work in non-cartesian-monoidal context is that the authors use essentially  the concept of {\it weak Segal monoids} whose standard definition assumes that the $\Hom$-sets belong to a cartesian-monoidal category.

However, there appeared a definition of weak Segal monoids in non-cartesian monoidal categories, due to Tom Leinster [Le].
Working with the Leinster's definition, one needs to have a more subtle monoidal property of the Dwyer-Kan-type localization of categories than the one used in [KT]. Let us formulate this necessary property (see Theorem \ref{theor_appl_2} below).

Let $\mathscr{C}at^\dg_\mathbb{U}$ denotes the category of $\mathbb{U}$-small dg categories over a field $k$, let $\mathscr{CC}at^\dg_\mathbb{U}$ denotes the category of {\it colored} $\mathbb{U}$-small dg categories over $k$, and let, finally, $\mathscr{C}at^\pretr_\mathbb{U}$ (resp., $\mathscr{CC}at^\pretr_\mathbb{U}$) denotes the category of $\mathbb{U}$-small pre-triangulated dg categories (resp., the category of colored $\mathbb{U}$-small pre-triangulated dg categories) over $k$. The either of these three categories is symmetric monoidal, with $\otimes_k$ as the monoidal product.

Tabuada [Tab] constructed a closed model structure on the category of $\mathbb{U}$-small dg categories; this structure is assumed in the two following statements.

\begin{theorem}[to be proven later]\label{theor_appl_1}{\itshape
Let $k$ be a field of any characteristic.
There is a functor $\mathfrak{R}\colon \mathscr{C}at^\pretr_\mathbb{U}\to \mathscr{Cat}^\dg_\mathbb{U}$ and a morphism of functors $w\colon \mathfrak{R}\to Id$ with the following properties:
\begin{itemize}
\item[1.] $\mathfrak{R}(C)$ is cofibrant, and $w\colon \mathfrak{R}(C)\to C$ is a weak equivalence, for any $C\in \mathscr{C}at^\pretr_\mathbb{U}$,
\item[2.] there is a colax-monoidal structure on the functor $\mathfrak{R}$, such that all colax-maps $\beta_{C,D}\colon \mathfrak{R}(C\otimes D)\to \mathfrak{R}(C)\otimes \mathfrak{R}(D)$ are weak equivalences of dg categories, and such that the diagram
    $$
    \xymatrix{
    \mathfrak{R}(C\otimes D)\ar[rr]^{\beta_{C,D}}\ar[rd]&&\mathfrak{R}(C)\otimes \mathfrak{R}(D)\ar[dl]\\
    &C\otimes D
    }
    $$
    is commutative.
\end{itemize}
}
\end{theorem}
To prove Theorem \ref{theor_appl_1}, we firstly extend our Main Theorem \ref{theor1_intro} from $\mathbb{Z}_{\le 0}$-graded dg algebras to
$\mathbb{Z}_{\le 0}$-graded dg categories, this should be proven by the methods of this paper and the ones of [BM]. In the next step, we extend the result from $\mathbb{Z}_{\le 0}$-graded dg categories to pre-triangulated dg categories. It seems that the analogous statement is not true for general $\mathbb{Z}$-graded dg categories (not the pre-triangulated ones).

\begin{theorem}[to be proven later]\label{theor_appl_2}{\itshape
There is a localization functor $$\mathbb{L}\colon\mathscr{CC}at^\pretr_\mathbb{U}\to \mathscr{C}at^\dg_{[\mathbb{U}]}$$
from colored pre-triangulated $\mathbb{U}$-small dg categories to dg $\mathbb{U}$-categories, with the following properties:
\begin{itemize}
\item[(i)] $\mathbb{L}$ is colax-monoidal, with weak equivalences of dg categories as the colax maps $\beta(C,D)$,
\item[(ii)] $\beta(C,D)$ is defined over $C\otimes D$, for any two pre-triangulated dg categories $C,D$,
\item[(iii)] the image of $\mathbb{L}(C,S)$ under the projection to the homotopy category $\Ho\mathscr{C}at^\dg_{\mathbb{U}}$ coincides with the To\"{e}n dg localization [To], Sect.8.2.
\end{itemize}
}
\end{theorem}

To\"{e}n defines a localization of the dg categories with nice homotopy properties (see [To], Corollary 8.7) as the  homotopy colimit of some push-out-angle diagram. It is possible to compute this homotopy colimit  as the genuine colimit of a cofibrant replacement of the initial  diagram.
We use the cofibrant resolution $\mathfrak{R}(C)$, given in Theorem \ref{theor_appl_1}, for this replacement. (In fact, the push-out-angle diagrams and the pull-back-angle diagrams in a closed model category admit closed model structures, with component-wise weak equivalences as the weak equivalences, see [DS], Section 10).

Theorem \ref{theor_appl_2} is essentially the only new toolkit one needs to have, to extend the Kock-To\"{e}n's proof of simplicial $n$-fold monoidal Deligne conjecture to its genuine $k$-linear version. Notice that this (way of) proof of it does not use any transcendental numbers which unavoidably appear in any ``physical'' proof. When our $n$-fold monoidal category is $k$-linear, for a field $k$, the output homotopy $(n+1)$-algebra (for a homotopy equivalent definition of this concept) is defined over the field $k$ as well.

\subsection{\sc Plan of the paper}\label{section03}
We start with proving in Section 1 a direct analogue of Theorem \ref{theor1_intro} for simplicial associative algebras over $k$, instead of $\mathbb{Z}_{\le 0}$-graded dg associative algebras (see Theorem \ref{theor_simpl_alg} below). It turns out that Theorem \ref{theor1_intro} becomes much simpler in the simplicial setting. We solve it by an explicit construction, which reminiscences the construction of Dwyer-Kan in their first paper [DK1] on simplicial localization. We denote by $\mathfrak{F}(A)$, for a simplicial algebra $A$, our solution to Theorem \ref{theor_simpl_alg}.

The rough idea of the remaining Sections is to ``transfer'' the solution  of Theorem \ref{theor_simpl_alg} for simplicial algebras to a solution to the Main Theorem \ref{theor1_intro}, using the Dold-Kan correspondence and passing to the categories of monoids.
Recall (see Section 3 for detail) that the Dold-Kan correspondence
$$
N\colon\mathscr{M}od(\mathbb{Z})^\Delta\rightleftarrows \mathscr{C}^-(\mathbb{Z})\colon \Gamma
$$
is an adjoint equivalence between the category $\mathscr{M}od(\mathbb{Z})^\Delta$ of simplicial abelian groups, and the category $\mathscr{C}^-(\mathbb{Z})$ of $\mathbb{Z}_{\le 0}$-graded complexes of abelian groups.
The both categories are symmetric monoidal, but neither of the functors $N$ and $\Gamma$ respects the monoidal structure.

In Section 2, we recall some ``non-homotopy'' results of Schwede-Shipley [SchS03]. In particular, having a lax-monoidal structure $\varphi$ on the right adjoint functor $R$ in an adjunction
$$
L\colon\mathscr{C}\rightleftarrows\mathscr{D}\colon R
$$
between monoidal categories $\mathscr{C}$ and $\mathscr{D}$ one easily defines a functor
$$
R^\mon\colon\Mon\mathscr{D}\to\Mon\mathscr{C}
$$
and, bit more tricky, one defines a left adjoint $L^\mon$ to $R^\mon$, such that there is an adjunction
$$
L^\mon\colon\Mon\mathscr{C}\rightleftarrows\Mon\mathscr{D}\colon R^\mon
$$
Applying this construction to $L=N$, $R=\Gamma$, and $\varphi$ the lax-monoidal structure on $\Gamma$ adjoint to the Alexander-Whitney colax-monoidal structure on $N$, one gets an adjoint pair of functors
\begin{equation}\label{newintro1}
N^\mon\colon\Mon(\mathscr{M}od(\mathbb{Z})^\Delta)\rightleftarrows\Mon(\mathscr{C}^-(\mathbb{Z}))\colon \Gamma^\mon
\end{equation}
We prove in Section 5 that the formula
\begin{equation}\label{newintro2}
\mathfrak{R}(A)=N^\mon(\mathfrak{F}(\Gamma^\mon(A)))
\end{equation}
gives a solution to Theorem \ref{theor1_intro}, where $\mathfrak{F}(B)$ is a solution to Theorem \ref{theor_simpl_alg}.

The proof that \eqref{newintro2} solves Theorem \ref{theor1_intro} is rather sophisticated.  We use mainly the following techniques:
the ``homotopy'' results of [SchS03] (which claims, in particular, that \eqref{newintro1} is a Quillen equivalence); the bialgebra axiom for the Dold-Kan correspondence (proven in [AM] and independently but later in [Sh2]); and the theory developed in Section 4 and in Appendix B of this paper.
In general it is not true that, for an adjoint pair of functors, having a solution $\mathfrak{F}(B)$ of Th.\ref{theor_simpl_alg}-like theorem, the formulae \eqref{newintro2} gives a solution of Th.\ref{theor1_intro}-like theorem. In fact, we use many specific features of the Dold-Kan correspondence throughout the proof.

In Section 4 we find an explicit expression for the Schwede-Schipley's functor $L^\mon$ (defined a priori as a co-equalizer \eqref{quotient_def}), in the particular case of the Dold-Kan correspondence, with $L=N$, the normalized chain complex functor. It is relatively easy to do, however, we need to know as well an explicit expression for the {\it canonical colax-monoidal structure on $L^\mon$}, that is the one adjoint to the natural lax-monoidal structure on $R^\mon$. The latter task requires some amount of work, and occupies most part of Section 4.
Some of results of very general nature we need in Section 4 are formulated in proven separately in Appendix B.
Section 4 and Appendix B together form a technical core of the paper.

The main results of Section 4 are formulated in Theorem \ref{another} and Theorem \ref{keytheorem}.
As well, we prove here more general Theorem \ref{keytheorembis}, which describes the colax-monoidal structure on $L^\mon$ in greater generality.

In Section 5, the different previous results of the paper become connected and interacted, which leads us to a proof of the Main Theorem \ref{theor1_intro} (which comes up under the name Theorem \ref{theor_gr_alg} in the body of the paper).

In Appendix A we collect, to ease the reader's reference, some ``diagrammatic'' definitions, used in the paper.
We recall here the definitions of a lax-monoidal and of a colax-monoidal structures on a functor, and formulate the bialgebra axiom.

In Appendix B we develop some techniques, used in proofs of Section 4.
Roughly speaking, we study here  a property of a pair of functors $L\colon\mathscr{C}\rightleftarrows \mathscr{D}\colon R$ which weakens the property that ``$L$ and $R$ are adjoint functors''. We define {\it weak right adjoint} pairs and {\it weak left adjoint} pairs of functors, followed by {\it very weak left} and {\it very weak right} adjoint pairs. Our goal is to study how these weak concepts interact with the lax-monoidal structures on $R$ and the colax-monoidal structures on $L$, when the categories $\mathscr{C}$ and $\mathscr{D}$ are supposed to be monoidal. (For the case of a honest adjoint functors, there is a 1-1 correspondence between the lax structures on $R$ and the colax structures on $L$, see Lemma \ref{kurica}).
The results proven there are very general in nature, and we decided to
organize them in a single Appendix would be better than to spread them out between the proofs of Section 4.

\subsection{\sc Notations}\label{not}
Throughout the paper, $k$ denotes a field of any characteristic. An ``algebra'' always means an ``associative algebra with unit''.

All differentials have degree +1, as is common in the algebraic literature.

Let $\Delta$ be the category whose objects are $[0]$, $[1]$, $[2]$, $[3]$, and so on, where $[n]$ denotes the completely ordered sets with $n+1$ elements $0<1<2<\dots<n$. A morphism $f\colon [m]\to[n]$ is any map obeying $f(i)\le f(j)$ when $i\le j$. A {\it simplicial} object in a category $\mathscr{C}$ is a functor $\Delta^\opp\to\mathscr{C}$, and a {\it cosimplicial} object in $\mathscr{C}$ is a functor $\Delta\to \mathscr{C}$. We denote by $\mathscr{C}^\Delta$ the category of simplicial objects in $\mathscr{C}$ and by $\mathscr{C}^{\Delta^\opp}$ the category of cosimplicial objects in $\mathscr{C}$. This notation is indeed confusing, but seemingly it is traditional now.

All categories we consider in this paper are small for some universe. We do not meet here any set-theoretical troubles related with the localization of categories, and we always skip the adjective ``small'' in the formulations of our results (except Section \ref{applications}).

\subsection*{}
\subsubsection*{\sc Acknowledgments}
I am grateful to  Sasha Beilinson, Volodya Hinich, Ieke Moerdijk, Stefan Schwede, and Vadik Vologodsky for discussions on the topics related to this paper.
I am greatly indebted to Martin Schlichenmaier for his kindness and support during my 5-year appointment at the University of Luxembourg, which made possible my further development as a mathematician.
The work was done during research stay at the Max-Planck Institut f\"{u}r Mathematik, Bonn. I am thankful to the MPIM for hospitality, for financial support, and for very creative working atmosphere.

\section{\sc The case of simplicial algebras}\label{section1}
The category $\mathscr{A}lg_k^\Delta$ of simplicial algebras over field $k$ is monoidal, with degree-wise $\otimes_k$ as the monoidal structure, and it admits a closed model structure. We recall this closed model structure in Section \ref{section12} below. We refer to this closed model structure in the following result.
\begin{theorem}[Main Theorem for simplicial algebras]\label{theor_simpl_alg}{\itshape
Let $k$ be a field of any characteristic.
There is a functor $\mathfrak{F}\colon \mathscr{A}lg_k^{\Delta}\to \mathscr{A}lg_k^{\Delta}$ and a morphism of functors $w\colon \mathfrak{F}\to Id$ with the following properties:
\begin{itemize}
\item[1.] $\mathfrak{F}(A)$ is cofibrant, and $w\colon \mathfrak{F}(A)\to A$ is a weak equivalence, for any $A\in \mathscr{A}lg_k^\Delta$,
\item[2.] there is a colax-monoidal structure on the functor $\mathfrak{F}$, such that all colax-maps $\beta_{A,B}\colon \mathfrak{F}(A\otimes B)\to \mathfrak{F}(A)\otimes \mathfrak{F}(B)$ are weak equivalences of simplicial algebras, and such that the diagram
    $$
    \xymatrix{
    \mathfrak{F}(A\otimes B)\ar[rr]^{\beta_{A,B}}\ar[rd]&&\mathfrak{F}(A)\otimes \mathfrak{F}(B)\ar[dl]\\
    &A\otimes B
    }
    $$
    is commutative,
\item[3.] the morphism $w(k_\mb)\colon \mathfrak{F}(k_\mb)\to k_\mb$ coincides with $\alpha\colon \mathfrak{F}(k_\mb)\to k_\mb$, where $\alpha$ is a part of the colax-monoidal structure {\rm (see Definition \ref{colax_intro})}, and $k_\mb=1_{\mathscr{A}lg_k^\Delta}$ is the simplicial algebra equal to the one-dimensional $k$-algebra $k$ in each degree.
\end{itemize}
}
\end{theorem}
\subsection{\sc The construction}\label{section11}
The idea is very easy. Let $A_\mb$ be a simplicial algebra. There is the forgetful functor $\mathscr{A}lg_k^\Delta\to\mathscr{V}ect_k^\Delta$ to simplicial vector spaces, having a left adjoint functor of ``free objects''. This is the functor $A_\mb\rightsquigarrow (TA)_\mb$, with
\begin{equation}
(TA)_k=T(A_k)
\end{equation}
where $T(A_k)$ is the free (tensor) algebra. We consider the cotriple, associated with the pair of adjoint functors ($L$ is the left adjoint to $R$)
$$
L\colon \mathscr{V}ect_k^\Delta\rightleftarrows\mathscr{A}lg_k^\Delta\colon R
$$
(see [W], Section 8.6). Explicitly,
\begin{equation}
T=L\circ R
\end{equation}
This implies that there are maps of functors $\epsilon\colon T\to \id$ and $\delta\colon T\to T^2$ obeying the cotriple axioms.
These axioms guarantee that the following collection of algebras $(FA)_k$, $k\ge 0$, has a natural structure of a simplicial algebra $(FA)_\mb$:
$$
(FA)_k=T^{\circ(k+1)}A_k
$$
(there is the $(k+1)$-st iterated tensor power in the r.h.s.),
such that the natural map $FA_\mb\to A_\mb$, $T^{\circ(k+1)}A_k\xrightarrow{\epsilon^{k+1}}A_k$, is a weak equivalence of simplicial algebras.

Explicitly, having such a functor $T$ with maps of functors $\epsilon\colon T\to\id$ and $\delta\colon T\to T^2$, the simplicial structure on $(FA)_\mb$ is defined as follows.

When $A_k=A$ for any $k$, ($A_\mb$ is a constant simplicial algebra), the formulas for simplicial algebra structure on $(FA)_\mb$ are:
\begin{equation}\label{cotriple}
\begin{aligned}
\ &d_i=T^{\circ i}\epsilon T^{\circ (n-i)}\colon  T^{\circ (n+1)}A\to T^{\circ n}A\\
&s_i=T^{\circ i}\delta T^{\circ (n-i)}\colon T^{\circ (n+1)}A\to T^{\circ (n+2)}A
\end{aligned}
\end{equation}
The cotriple axioms then guarantee the simplicial identities (see [W], Section 8.6.4 for detail).

In general case, when $A_\mb$ is not constant, the simplicial algebra $(FA)_\mb$ is defined as the diagonal of the bisimplicial set
\begin{equation}\label{eqf}
(FA)_\mb=\mathrm{diag}((FA_\mb)_\mb)
\end{equation}

For two simplicial algebras $A_\mb$ and $B_\mb$, there is a canonical embedding
$$
\beta_{A,B}\colon F(A\otimes B)_\mb\to (FA)_\mb\otimes (FB)_\mb
$$
defined {\it on the level of algebras} by iterations of the map
\begin{equation}\label{cf}
\begin{aligned}
\ \alpha\colon T(A\otimes B)&\to T(A)\otimes T(B)\\
(a_1\otimes b_1)\otimes \dots \otimes (a_k\otimes b_k)& \xrightarrow{\alpha} (a_1\otimes \dots\otimes a_k)\otimes (b_1\otimes \dots\otimes b_k)
\end{aligned}
\end{equation}
The component $(\beta_{A,B})_k\colon (T^{\circ(k+1)}A_k)\otimes(T^{\circ(k+1)}B_k)\to T^{\circ (k+1)}(A_k\otimes B_k)$ is defined as the iterated power $\alpha^{\circ (k+1)}$.
\begin{lemma}{\itshape
The collection of maps $\{\beta_\ell\}$, $\ell\ge 0$, defines a map of {\rm{simplicial algebras}}
$$
\beta\colon (F(A\otimes B))_\mb\to (FA)_\mb\otimes (FB)_\mb
$$
}
\end{lemma}
\begin{proof}
Denote the product(s) in $A_k$ by $\star$, the product in $T(V)$ by $\otimes$, and the product in $T(T(V))$ by $\bigotimes$. Then the formulas for $\epsilon\colon T\to \id$ and $\delta\colon T\to T^2$ are as follows:
\begin{equation}
\begin{aligned}
\ &\epsilon (a_1\otimes\dots \otimes a_k)=a_1\star\dots\star a_k\\
&\delta (a_1\otimes\dots\otimes a_k)=a_1\bigotimes\dots\bigotimes a_k
\end{aligned}
\end{equation}
The statement of Lemma now follows directly from formulas \eqref{cotriple}, expressing the simplicial faces and degeneracies maps in $\epsilon$ and $\delta$.
\end{proof}
Our goal in this Subsection is to prove that $\mathfrak{F}(A_\mb)=(FA)_\mb$ solves Theorem \ref{theor_simpl_alg}.
We need to prove
\begin{lemma}{\itshape
\begin{itemize}
\item[1.] For any simplicial algebra $A_\mb$, the simplicial algebra $FA_\mb$ is cofibrant in the closed model structure on $\mathscr{A}lg_k^\Delta$,
\item[2.] the map $\beta_{A,B}\colon (F(A\otimes B))_\mb\to (FA)_\mb\otimes (FB)_\mb$ is a weak equivalence.
\end{itemize}
}
\end{lemma}

Before proving the above Lemma, we need to remind some results concerning the closed model structure on the category $\mathscr{A}lg_k^\Delta$, which goes back to Quillen [Q], Section 4.3.

\subsection{\sc The closed model category of simplicial algebras}\label{section12}
Firstly recall the model structure on the category $\mathscr{V}ect_k^\Delta$ of simplicial vector spaces over field $k$.
\begin{equation}
\parbox{5,6in}{
\begin{itemize}
\item[(i)] A map $f\colon X\to Y$ in $\mathscr{V}ect_k^\Delta$ is a weak equivalence if it induces an isomorphism on homotopy groups $\pi_\mb (X)\to\pi_\mb (Y)$,
\item[(ii)] a map $f\colon X\to Y$ is a fibration if it induces a surjection  $\pi_\mb X\to \pi_0(X)\times_{\pi_0(Y)}\pi_\mb(Y)$.
\item[(iii)] A map $f\colon X\to Y$ in $\mathscr{V}ect_k^\Delta$ is a cofibration if it has a form
$$
X_n\to Y_n=X_n\oplus V_n
$$
for some collection of vector spaces $\{V_0,V_1,V_2,\dots\}$, such that each simplicial {\it degeneracy map} $s_i\colon [n+1]\to [n]$ maps $V_n$ to $V_{n+1}$, $n\ge 0$.
\end{itemize}
}
\end{equation}
The model category $\mathscr{V}ect_k^\Delta$ described above is {\it cofibrantly generated} (see [GS], Section 3, for a beautiful short survey of cofibrantly generated model categories). Recall that it means, in particular, that there are given sets $I$ of generating cofibrations, and $J$ of generating acyclic cofibrations, subject to the following two properties:
\begin{itemize}
\item[1.] the source of any morphism in $I$ obeys {\it the Quillen's small object argument} to the category of all cofibrations; the source of any morphism in $J$ obeys the small object argument to the category of all acyclic cofibrations;
\item[2.] a morphism is a fibration if and only if it satisfies the left lifting property with respect to any morphism in $J$; a morphism is an acyclic fibration if and only if it satisfies the left lifting property with respect to any morphism in $I$.
\end{itemize}
Concerning {\it the Quillen's small object argument}, see [GS], Section 3.1, or [Hir], Section 10.5, for thorough treatment.
The meaning of these two conditions is that they make possible to prove the last axiom (CM5) of a closed model category axioms, which is in a sense the hardest one (see loc.cit.).

See [GS], Examples 3.4, for explicit description of the sets $I$ and $J$ in the category $\mathscr{V}ect_k^\Delta$.

There is a pair of adjoint functors
\begin{equation}\label{eq1}
L\colon \mathscr{V}ect_k^\Delta\rightleftarrows \mathscr{A}lg_k^\Delta\colon R
\end{equation}
As the left-hand side category is a cofibrantly generated model category, the model structure can be ``transferred'' to the right-hand-side
category, and this model category is again cofibrantly generated. This transfer principle is explained in [GS], Theorem 3.6, and [Hir], Theorem 11.3.2. As is explained in [GS], Sections 3,4, the assumptions of Theorem 3.6 are satisfied in \eqref{eq1}.

In the situation when assumptions of Theorem 3.6 of [GS] are fulfilled, the sets $L(I)$ and $L(J)$ are generating cofibrations (resp., generating acyclic cofibrations) for the category in the right-hand side.

The obtained closed model structure on $\mathscr{A}lg_k^\Delta$ has the following explicit description, see [GS], Section 4.3.
\begin{equation}\label{sacm}
\parbox{5,6in}{
\begin{itemize}
\item[(i)]a map $f\colon X\to Y$ is a weak equivalence if $\pi_*f\colon \pi_*X\to \pi_*Y$ is an isomorphism,
\item[(ii)]a map $f\colon X\to Y$ is a fibration if the induced map $X\to\pi_0X\times_{\pi_0Y}Y$ is a surjection.

\item[(iii)] a map $f\colon X\to Y$ is a cofibration in $\mathscr{A}lg_k^\Delta$ if it is a retract of the following {\it free map}:
$$
X_n\to Y_n=X_n\sqcup T(V_n)
$$
as algebras, for some collection $\{V_0,V_1,V_2,\dots\}$ of vector spaces, such that all degeneracy maps $s_i\colon [n+1]\to [n]$ maps $V_n$ to $V_{n+1}$, $n\ge 0$.
\end{itemize}
}
\end{equation}
See [Q], Section 4.3 and [GS], Proposition 4.21 for a proof.

\subsection{\sc Proof of Theorem \ref{theor_simpl_alg}}\label{section13}
Firstly we prove
\begin{lemma}\label{lemmacof}{\itshape
For any simplicial algebra $A_\mb$, the simplicial algebra $(FA)_\mb$ is cofibrant, and the projection $p\colon (FA)_\mb\to A_\mb$ is an acyclic fibration.}
\end{lemma}
\begin{proof}
We need to find vector spaces $V_i$ such that $(FA)_n=T(V_n)$ and such that all degeneracies maps $s_i\colon [n+1]\to[n]$ define maps of algebras $T(V_n)\to T(V_{n+1})$ induced by some maps of generators $V_n\to V_{n+1}$. We set $V_n=T^{\circ n}(V_n)$, it is clear that this choice satisfies the both conditions. The statement that the map $(FA)_\mb\to A_\mb$ is both a weak equivalence and a fibration, is clear.
\end{proof}

Next follows
\begin{lemma}{\itshape
For any two simplicial algebras $A_\mb,B_\mb$, the map $\beta_{A,B}\colon F(A\otimes B)_\mb\to (FA)_\mb\otimes (FB)_\mb$ is a weak equivalence.
}
\end{lemma}
\begin{proof}
It is a straightforward and simple check that the diagram
\begin{equation}\label{web}
\xymatrix{
F(A\otimes B)_\mb\ar[rr]^{\beta_{A,B}}\ar[rd]_{p_{A\otimes B}}&&(FA)_\mb\otimes (FB)_\mb\ar[dl]^{p_A\otimes p_B}\\
&A_\mb\otimes B_\mb
}
\end{equation}
is commutative. The map $p_{A\otimes B}$ is a weak equivalence by Lemma \ref{lemmacof}, the product $p_A\otimes p_B$ is a weak equivalence by Lemma \ref{lemmacof} again. Then, the commutativity of the  diagram \eqref{web} implies, by 2-out-of-3 axiom of closed model category, that the third arrow $\beta_{A,B}$ is also a weak equivalence.
\end{proof}

\section{\sc Monoids and the Bialgebra axiom}\label{section2}
\subsection{\sc Lax$\leftrightarrow$colax duality}
Let $\mathscr{C}$ and $\mathscr{D}$ be two strict monoidal categories, and
let $F\colon \mathscr{C}\to\mathscr{D}$ be a functor with two properties:
\begin{itemize}
\item[1)] $F$ is an equivalence of the underlying categories,
\item[2)] $F$ is strict monoidal, that is $F(X\otimes Y)=F(X)\otimes F(Y)$ for any two $X,Y\in\mathscr{C}$.
\end{itemize}
Then one can choose a quasi-inverse to $F$ functor $G_1\colon \mathscr{D}\to \mathscr{C}$ such that $(F,G_1)$ is an {\it adjoint} equivalence.
As well, we can choose a quasi-inverse functor to $F$ functor $G_2\colon \mathscr{D}\to\mathscr{C}$ which is {\it strict monoidal}.
However, one can not choose in general a quasi-inverse $G$ enjoying the both properties at once.
This is an explanation how the {\it lax-monoidal functors} and the {\it colax-monoidal functors} come up into the contemporary mathematics.
The reader is referred to Appendix A for the definitions of the (co)lax-monoidal functors.

The set-up of the following standard lemma is very common.
\begin{lemma}\label{kurica}
Let $\mathscr{C}$ and $\mathscr{D}$ be two strict monoidal categories, and let
$$
L\colon\mathscr{C}\rightleftarrows\mathscr{D}\colon R
$$
be a pair of adjoint functors, with $L$ the left adjoint. Then there is a 1-1 correspondence between the colax-monoidal structures on $L$ and the lax-monoidal structures on $R$, given by \eqref{lg}, \eqref{cg} below. Evermore, this 1-1 correspondence is involutive.
\end{lemma}
\begin{proof}
Let $\epsilon\colon LR\to\Id_\mathscr{D}$ and $\eta\colon\Id_\mathscr{C}\to RL$ are the adjunction maps.

For a colax-monoidal structure $c$ on $L$, written as $c_{X,Y}\colon L(X\otimes Y)\to L(X)\otimes L(Y)$, define a lax-monoidal structure $\ell$ on the functor $A$ as
\begin{equation}\label{lg}
R(X)\otimes R(Y)\xrightarrow{\eta}RL(R(X)\otimes R(Y))\xrightarrow{c_*}R(LR(X)\otimes LR(Y))\xrightarrow{(\epsilon\otimes\epsilon)_*}R(X\otimes Y)
\end{equation}
Vice versa, suppose a lax-monoidal structure $\ell$ on $R$ is given. Define a colax-monoidal structure $c$ on $L$ as
\begin{equation}\label{cg}
L(X\otimes Y)\xrightarrow{(\eta(X)\otimes\eta(Y))_*}L(RL(X)\otimes RL(Y))\xrightarrow{\ell_*}LR(L(X)\otimes L(Y))\xrightarrow{\epsilon_*}L(X)\otimes L(Y)
\end{equation}
Recall that for the adjoint pairs of functors the two compositions, defined out of $\epsilon$ and $\eta$
\begin{equation}
L\rightarrow L\boxed{RL}=\boxed{LR}L\rightarrow L
\end{equation}
and
\begin{equation}
R\rightarrow \boxed{RL}R=R\boxed{LR}\rightarrow R
\end{equation}
are the identity maps of functors. This identities imply all claims of Lemma by a direct check.
\end{proof}

When $L\colon\mathscr{C}\rightleftarrows\mathscr{D}\colon R$ is an adjoint {\it equivalence}, there are more possibilities for the application of the above construction. For example, suppose there are given a colax-monoidal structure $c_L$ and a lax-monoidal structure $\ell_L$ on the functor $L$.
Then \eqref{lg} and \eqref{cg} define a lax-monoidal structure $\ell_R$ and a colax-monoidal structure $c_R$ on the functor $R$. There is a compatibility relation between $c_L$ and $\ell_L$, which holds if and only if the same relation holds for $c_R$ and $\ell_R$. This relation, called {\it the bialgebra axiom}, is elaborated in Section \ref{section23}.
\subsection{\sc The category of monoids}\label{section21}
Let $\mathscr{M}$ be a symmetric monoidal category, $\Mon \mathscr{M}$ be the category of monoids in $\mathscr{M}$. There is the forgetful functor
$$
f\colon \Mon\mathscr{M}\to \mathscr{M}
$$
Under some conditions, the functor $f$ has a left adjoint functor , ``the free monoid functor''. Recall the following result, from [ML], Chapter VII.3:
\begin{lemma}\label{ML}{\itshape
Let $\mathscr{M}$ be a monoidal category with all finite colimits, such that the functors $-\otimes a$ and $a\otimes -$ (for fixed $a$) commute with finite colimits. Then the functor $\mathscr{M}\to \Mon\mathscr{M}$, $X\mapsto T(X)$, with
\begin{equation}\label{MLM}
T(X)=1_\mathscr{M}\coprod X\coprod X\otimes X\coprod\dots
\end{equation}
is left adjoint to the forgetful functor.
}
\end{lemma}
When the finite colimits exist, and there is the {\it inner $Hom$ functor} (a right adjoint to the monoidal product), the functors $-\otimes a$ and $a\otimes -$ commute with colimits by general categorical arguments.

We say that a monoidal category $\mathscr{M}$ is {\it good} when the assumptions of Lemma \ref{ML} hold.

Let now $\mathscr{M}_1,\mathscr{M}_2$ be two symmetric monoidal categories, and let $F\colon \mathscr{M}_1\to\mathscr{M}_2$ be a functor.
Suppose a lax-monoidal structure $\ell_F$ is given. Then there is a functor $F^\mon=F^\mon(\ell_F)\colon \Mon\mathscr{M}_1\to\Mon\mathscr{M}_2$, depending on $\ell_F$. For a monoid $X$ in $\mathscr{M}_1$, the underlying object of $F^\mon(X)$ is defined as $F(X)$, and the monoid structure is given as
\begin{equation}\label{monoidref}
F(X)\otimes F(X)\xrightarrow{\ell_F}F(X\otimes X)\xrightarrow{m_X} F(X)
\end{equation}
We have immediately:
\begin{lemma}{\itshape
In the above notations, the following two diagram is commutative:
\begin{equation}
\xymatrix{\mathscr{M}_1\ar[r]^{F}& \mathscr{M}_2\\
\Mon\mathscr{M}_1\ar[r]^{F^\mon}\ar[u] & \Mon\mathscr{M}_2\ar[u]}
\end{equation}
Here the vertical upward arrows are the forgetful functors.
}
\end{lemma}
\qed

\subsection{\sc The left adjoint functor on monoids}\label{section22}
Suppose now that the functor $F$ admits a left adjoint functor $L\colon \mathscr{M}_2\to \mathscr{M}_1$. In this case, we want to construct a functor $L^\mon\colon \Mon\mathscr{M}_2\to\Mon\mathscr{M}_1$, left adjoint to the functor $F^\mon$.

The following Lemma (and the construction in its proof) is due to [SchS03], Section 3.3:
\begin{lemma}\label{leftproof}{\itshape
Suppose the monoidal categories $\mathscr{M}_1$ and $\mathscr{M}_2$ are good, and suppose that the functor $L$ left adjoint to $F$ exists.
Then the left adjoint functor $L^\mon\colon \Mon\mathscr{M}_2\to\Mon\mathscr{M}_1$ exists, and it makes the diagram
\begin{equation}\label{leftmon}
\xymatrix{\mathscr{M}_1\ar[d]& \mathscr{M}_2\ar[d]\ar[l]_{L}\\
\Mon\mathscr{M}_1 & \Mon\mathscr{M}_2\ar[l]_{L^\mon}}
\end{equation}
commutative. Here the downward vertical arrows are the free monoid functors.
}
\end{lemma}
\begin{proof}
The second claim is a formal consequence from the existence of $L^\mon$, as the free monoid functors are left adjoint to the forgetful functors.
Define the value $L^\mon(X)$ (for a monoid $X$ in $\mathscr{M}_1$) as the co-equalizer in the category $\Mon(\mathscr{M}_2)$:
\begin{equation}\label{quotient_def}
\xymatrix{
T_{\mathscr{M}_2}(L(T_{\mathscr{M}_1}(X)))\ar@<1ex>[r]^{\ \ \ \alpha} \ar@<-1ex>[r]_{\ \ \ \beta} &T_{\mathscr{M}_2}(LX)
}
\end{equation}
where $T_\mathscr{M}$ denotes the free (tensor) monoid in a monoidal category $\mathscr{M}$, see \eqref{MLM}.
The map $\alpha$ in \eqref{quotient_def} comes from the map $T_{\mathscr{M}_1}(X)\to X$ defined from the monoid structure on $X$, and the map $\beta$ in \eqref{quotient_def} is defined from the following map $L(T_{\mathscr{M}_1}(X))\to T_{\mathscr{M}_2}(LX)$:
\begin{equation}\label{themapbeta}
L(\underbrace{X\otimes X\otimes \dots\otimes X}_{n\text{ factors}})\xrightarrow{{c_L}^{n-1}}\underbrace{L(X)\otimes L(X)\otimes \dots \otimes L(X)}_{n\text{ factors}}
\end{equation}
where $c_L$ is the colax-monoidal structure on $L$ adjoint to the lax-monoidal structure $\ell_F$ on $F$.

Let us prove that the functor $L^\mon$, defined by \eqref{quotient_def}, is left adjoint to the functor $F^\mon$.
We start with the following fact:
\begin{sublemma}\label{leftproofsub}{\itshape
Let $\mathscr{M}$ be a good monoidal category, and let $X\in\Mon\mathscr{M}$ be a monoid in $\mathscr{M}$.
Then $X$ is isomorphic to the co-equalizer of the following diagram in $\Mon\mathscr{M}$:
\begin{equation}\label{coeqtriv}
\xymatrix{
T_\mathscr{M}(T_\mathscr{M}(X))\ar@<1ex>[r]^{\ \ \ a} \ar@<-1ex>[r]_{\ \ \ b}&T_\mathscr{M}(X)
}
\end{equation}
where the map $a$ is defined on generators as the product map in $X$, $m_*\colon T_\mathscr{M}(X)\to X$, and the map $b$ is defined as the product map for the monoid $T_\mathscr{M}(X)$, $m_*\colon T_\mathscr{M}(T_\mathscr{M}(X))\to T_\mathscr{M}(X)$. The map $b$ does not depend on the monoid structure on $X$.
}
\end{sublemma}
It is clear.
\qed

{\it We continue to prove Lemma \ref{leftproof}}.

Let $X$ be a monoid in $\mathscr{M}_1$. Represent $X$ as co-equalizer \eqref{coeqtriv}. Denote by $D(X)$ the corresponding diagram. We have:
\begin{equation}
\begin{aligned}
\ &\Hom_{\Mon(\mathscr{M}_1)}(X,R^\mon (Z))=\Hom_{\Mon(\mathscr{M}_1)}(\colim D(X),R^\mon (Z))=\\
&\lim \Hom_{\Mon(\mathscr{M}_1)}(D(X), R^\mon(Z))=
\lim E(X,Z)=\lim E^\vee(X,Z)
\end{aligned}
\end{equation}
where $E(X,Z)$ and $E^\vee(X,Z)$ are the following diagrams:
\begin{equation}
\xymatrix{
\Hom_{\mathscr{M}_1}(T_{\mathscr{M}_1}(X),R(Z))&\Hom_{\mathscr{M}_1}(X,R(Z))\ar@<1ex>[l]^{\ \ \ \ \ \ B}\ar@<-1ex>[l]_{\ \ \ \ \ A}
}
\end{equation}
and
\begin{equation}
\xymatrix{
\Hom_{\mathscr{M}_2}(L(T_{\mathscr{M}_1}(X)),Z)&\Hom_{\mathscr{M}_2}(L(X),Z)\ar@<1ex>[l]^{\ \ \ \ \ \ B^\vee}\ar@<-1ex>[l]_{\ \ \ \ \ A^\vee}
}
\end{equation}
correspondingly.

Let us compute the maps $A$ and $B$ explicitly. The map $A$ is induced by the product $m_*\colon T_{\mathscr{M}_1}(X)\to X$ in the monoid $X$.
The map $B$ is little more tricky. Define $B(\mu)$ for $\mu\in \Hom_{\mathscr{M}_1}(X,R(Z))$. It is enough to define $B(\mu)_n\colon \Hom_{\mathscr{M}_1}(X^{\otimes n},R(Z))$, for any $n\ge 2$. The latter map is defined as $\mu^{\otimes n}\in \Hom_{\mathscr{M}_1}(X^{\otimes n},R(Z)^{\otimes n})$, followed by the product map $R(Z)^{\otimes n}\to R(Z)$ in the monoid $R(Z)$, defined out of monoid $Z$ by \eqref{monoidref}.

We continue:
\begin{equation}
\lim E^\vee(X,Z)=\lim E^{\vee}_{\Mon}(X,Z)=\Hom_{\Mon(\mathscr{M}_2)}(\colim \eqref{quotient_def},Z)
\end{equation}
with $E^{\vee}_{\Mon}$ the following diagram:
\begin{equation}
\xymatrix{
\Hom_{\Mon({\mathscr{M}_2})}(T_{\mathscr{M}_2}(L(T_{\mathscr{M}_1}(X)),Z))&\Hom_{\Mon({\mathscr{M}_2})}(T_{\mathscr{M}_2}(L(X)),Z)\ar@<1ex>[l]^{\ \ \ \ \ \ B_{\mathrm{mon}}}\ar@<-1ex>[l]_{\ \ \ \ \ A_{\mathrm{mon}}}
}
\end{equation}
The maps $A_{\mathrm{mon}}$ and $B_{\mathrm{mon}}$ can be explicitly described from the description of $A$ and $B$ given above. We are done.
\end{proof}
Now we pass to the situation when the functor $F$ admits, besides the lax-monoidal structure $\ell_F$, a colax-monoidal structure $c_F$, {\it compatible by the bialgebra axiom} (see Section 5.3 below).

\subsection{\sc The Bialgebra axiom}\label{section23}
Fix some notations on adjoint functors.

Let $L\colon\mathscr{A}\rightleftarrows\mathscr{B}\colon R$ be two functors. They are called adjoint to each other, with $L$ the left adjoint and $R$ the right adjoint, when
\begin{equation}\label{adj0}
\Mor_\mathscr{B}(LX,Y)\simeq \Mor_\mathscr{A}(X,RY)
\end{equation}
where ``$\simeq$'' here means ``isomorphic as bifunctors $\mathscr{A}^\opp\times \mathscr{B}\to\Sets$''.

This gives rise to maps of functors $\epsilon\colon LR\to\Id_\mathscr{B}$ and $\eta\colon \Id_{\mathscr{A}}\to RL$ such that the compositions
\begin{equation}\label{adj}
\begin{aligned}
\ &L\xrightarrow{L\circ \eta}LRL\xrightarrow{\epsilon\circ L}L\\
&R\xrightarrow{\eta\circ R}RLR\xrightarrow{R\circ \epsilon}R
\end{aligned}
\end{equation}
are identity maps of the functors.

The inverse is true: given maps of functors $\epsilon\colon LR\to\Id_\mathscr{B}$ and $\eta\colon \Id_{\mathscr{A}}\to RL$, obeying \eqref{adj}, gives rise to the isomorphism of bifunctors, that is, to adjoint equivalence (see [ML], Section IV.1, Theorems 1 and 2).

In particular, the case of {\it adjoint equivalence} is the case when $\epsilon\colon LR\to \Id_\mathscr{B}$ and $\eta\colon\Id_\mathscr{A}\to RL$ are {\it isomorphisms of functors}. In this case, setting $\epsilon_1=\eta^{-1}$ and $\eta_1=\epsilon^{-1}$, we obtain another adjunction, with $L$ the {\it right} adjoint and $R$ the {\it left adjoint}.

Let $\phi\in \Mor_\mathscr{B}(LX,Y)$. The following explicit formula for its adjoint $\psi\in \Mor_\mathscr{A}(X,RY)$ will be useful:
\begin{equation}\label{adj2}
X\xrightarrow{\eta}RLX\xrightarrow{R(\phi)}RY
\end{equation}
and analogously for the way back:
\begin{equation}\label{adj3}
LX\xrightarrow{L(\psi)}LRY\xrightarrow{\epsilon}Y
\end{equation}
(see [ML], Section IV.1).

Let now $\mathscr{C}$ and $\mathscr{D}$ be two symmetric monoidal categories, $F\colon \mathscr{C}\to\mathscr{D}$ a functor. Suppose a lax-monoidal structure $\ell_F$ and a colax-monoidal structure $c_F$ on $F$ are given. The {\it bialgebra axiom} is some compatibility condition on the pair $(c_F,\ell_F)$, see Section \ref{bialg}. Recall the following simple fact from [Sh2], Section 2:

\begin{lemma}\label{lemma12}{\itshape
Let $\mathscr{C}$ and $\mathscr{D}$ be two strict symmetric monoidal categories, and let $F\colon \mathscr{C}\rightleftarrows \mathscr{D}\colon G$ be an adjoint equivalence of the underlying categories. Given a pair $(c_F,\ell_F)$ where $c_F$ is a colax-monoidal structure on $F$, $\ell_F$ is a lax-monoidal structure on $F$, assign to it a pair $(c_G,\ell_G)$ of analogous structures on $G$, by \eqref{lg} and \eqref{cg}. If $\mathscr{C}$ and $\mathscr{D}$ are symmetric monoidal, and if the pair $(c_F,\ell_F)$ satisfies the  bialgebra axiom {\rm (see Section \ref{bialg})}, the pair $(c_G,\ell_G)$ satisfies the  bialgebra axiom as well.
}
\end{lemma}

\begin{proof}
Suppose $(c_F,\ell_F)$ are done. Define $(c_G,\ell_G)$ by \eqref{lg} and \eqref{cg}.
When we write down the bialgebra axiom diagram (see Section \ref{bialg}) for $(c_G,\ell_G)$ we see due to cancelations of $\epsilon$ with $\epsilon^{-1}$ and of $\eta$ with $\eta^{-1}$, that the diagram is commutative as soon as the diagram for $(c_F,\ell_F)$ is.
\end{proof}

\begin{lemma}\label{lemmaxx}{\itshape
Let $\mathscr{C},\mathscr{D}$ be two symmetric monoidal categories, and let $F\colon\mathscr{C}\to\mathscr{D}$ be a functor. Suppose a lax-monoidal structure $\ell_F$ on $F$ is given. Consider the functor $$F^\mon=F^\mon(\ell_F)\colon \Mon\mathscr{C}\to\Mon\mathscr{D}$$
defined in \eqref{monoidref}. Then the map  $F^\mon(X)\otimes F^\mon(Y)\to F^\mon (X\otimes Y)$, defined on the underlying objects as $\ell_F$, is a map of monoids, and, therefore, gives a lax-monoidal structure on $F^\mon$. Let now $c_F$ be a colax-monoidal structure on $F$.
If $(\ell_F,c_F)$ satisfies the bialgebra axiom, the map $F^\mon(X\otimes Y)\to F^\mon(X)\otimes F^\mon(Y)$, defined on the underlying objects as $c_F$, is a map of monoids, and, therefore, gives a colax-monoidal structure on $F^\mon$.
}
\end{lemma}

The both claims are straightforward checks. The second claim was, in fact, our motivation for introduction of the bialgebra axiom in [Sh2].

\section{\sc The Dold-Kan correspondence}\label{section3}
We use the following notations:

$\mathscr{C}(\mathbb{Z})$ is the category of unbounded complexes of abelian groups, $\mathscr{C}(\mathbb{Z})^+$ (resp., $\mathscr{C}(\mathbb{Z})^-$) are the full subcategories of $\mathbb{Z}_{\ge 0}$-graded (resp., $\mathbb{Z}_{\le 0}$-graded) complexes. The category of abelian groups placed in degree 0 (with zero differential) is denoted by $\mathscr{M}od(\mathbb{Z})$, thus, $\mathscr{M}od(\mathbb{Z})=\mathscr{C}(\mathbb{Z})^-\cap \mathscr{C}(\mathbb{Z})^+$.

\subsection{\sc}\label{section31}
The Dold-Kan correspondence is the following theorem:
\begin{theorem}[Dold-Kan correspondence]{\itshape
There is an adjoint equivalence of categories
$$
N\colon \mathscr{M}od(\mathbb{Z})^\Delta\rightleftarrows \mathscr{C}(\mathbb{Z})^-\colon\Gamma
$$
where $N$ is the functor of normalized chain complex (which is isomorphic to the Moore complex).
}
\end{theorem}
We refer to [W], Section 8.4, and [SchS03], Section 2, which both contain excellent treatment of this Theorem.

The both categories $\mathscr{M}od(\mathbb{Z})^\Delta$ and $\mathscr{C}(\mathbb{Z})^-$ are {\it symmetric monoidal} in natural way.
However, neither of functors $N$ and $\Gamma$ is monoidal.

There is a colax-monoidal structure on $N$, called {\it the Alexander-Whitney map} $AW\colon N(A\otimes B)\to NA\otimes NB$ and a lax-monoidal structure on $N$, called {\it the shuffle map} $\nabla\colon N(A)\otimes N(B)\to N(A\otimes B)$.

Recall the explicit formulas for them.

The Alexander-Whitney colax-monoidal map $AW\colon N(A\otimes B)\to N(A)\otimes N(B)$ is defined as
\begin{equation}\label{aw}
AW(a^k\otimes b^k)=\sum_{i+j=k}d_\fin^ia^k\otimes d_0^jb^k
\end{equation}
where $d_0$ and $d_\fin$ are the first and the latest simplicial face maps.

The Eilenberg-MacLane shuffle lax-monoidal map $\nabla\colon N(A)\otimes N(B)\to N(A\otimes B)$ is defined as
\begin{equation}
\nabla(a^k\otimes b^\ell)=\sum_{\substack{(k,\ell)\text{-shuffles }(\alpha,\beta)}}(-1)^{(\alpha,\beta)}S_\beta a^k\otimes S_\alpha b^\ell
\end{equation}
where
$$
S_\alpha=s_{\alpha_k}\dots s_{\alpha_1}
$$
and
$$
S_\beta=s_{\beta_\ell}\dots s_{\beta_1}
$$
Here $s_i$ are simplicial degeneracy maps, $\alpha=\{\alpha_1<\dots <\alpha_k\}$, $\beta=\{\beta_1<\dots<\beta_\ell\}$, $\alpha,\beta\subset [0,1,\dots, k+\ell-1]$, $\alpha\cap\beta=\varnothing$.

Let us summarize their properties in the following Proposition, see [SchS03], Section 2, and references therein, for a proof.
\begin{prop}\label{before}{\itshape
The colax-monoidal Alexander-Whitney and the lax-monoidal shuffle structures on the functor $N$ enjoy the following properties:
\begin{itemize}
\item[1.] the composition
$$
NA\otimes NB\xrightarrow{\nabla}N(A\otimes B)\xrightarrow{AW}NA\otimes NB
$$
is the identity,
\item[2.] the composition
$$
N(A\otimes B)\xrightarrow{AW}NA\otimes NB\xrightarrow{\nabla}N(A\otimes B)
$$
is naturally chain homotopic to the identity,
\item[3.] the shuffle map $\nabla$ is symmetric,
\item[4.] the Alexander-Whitney map $AW$ is symmetric up to a natural chain homotopy.
\end{itemize}
}
\end{prop}
\qed

Recall a Theorem proven independently in [AM], Sect. 5.4, and (later) in [Sh2], Sect. 2:
\begin{theorem}\label{mytheor}{\itshape
The pair $(\nabla, AW)$ of the lax-monoidal shuffle structure and the colax-monoidal Alexander-Whitney structure, defined on the normalized chain complex functor $N\colon\mathscr{M}od(\mathbb{Z})^\Delta\rightleftarrows \mathscr{C}(\mathbb{Z})^-$, obeys the bialgebra axiom.
}
\end{theorem}
\qed

This Theorem, along with Proposition \ref{before} (1.), is essentially used in the proof of Main Theorem \ref{theor_gr_alg} in Section 5.

\subsection{\sc Monoidal properties}\label{section32}
Let $F,G\colon \mathscr{C}\to \mathscr{D}$ be two functors between monoidal categories.

\begin{defn}{\rm
Suppose the functor $F$ is colax-monoidal, with the colax-monoidal structure $c_F$, and $G$ is lax-monoidal, with the lax-monoidal structure $\ell_G$.  A morphism of functors
$\Phi\colon F\to G$ is called {\it colax-monoidal} if for any $X,Y\in \mathscr{C}$, the diagram
\begin{equation}
\xymatrix{
F(X\otimes Y)\ar[d]_{c_F}\ar[r]^{\Phi}&G(X\otimes Y)\\
F(X)\otimes F(Y)\ar[r]^{\Phi\otimes\Phi}&G(X)\otimes G(Y)\ar[u]_{\ell_G}
}
\end{equation}

As well, when $F$ is lax-monoidal with the lax-monoidal structure $\ell_F$, and $G$ is colax-monoidal with the colax-monoidal structure $c_G$, a morphism $\Psi\colon F\to G$ is called {\it lax-monoidal}, if for any $X,Y\in\mathscr{C}$ the diagram
\begin{equation}
\xymatrix{
F(X)\otimes F(Y)\ar[r]^{\Psi\otimes\Psi}\ar[d]_{\ell_F}&G(X)\otimes G(Y)\\
F(X\otimes Y)\ar[r]^{\Psi}&G(X\otimes Y)\ar[u]_{c_G}
}
\end{equation}
}
\end{defn}
Each of the functors $N$ and $\Gamma$ admits both lax-monoidal and colax monoidal structures. Therefore, the compositions $N\circ \Gamma$ and $\Gamma\circ N$ are both lax- and colax-monoidal.

Here are the main monoidal properties concerning the Dold-Kan correspondence, from which (ii) is used essentially in our proof of Main Theorem \ref{theor_gr_alg} below.
\begin{lemma}\label{dkmonoidal}{\itshape
\begin{itemize}
\item[(i)] the adjunction map $\epsilon \colon N\circ \Gamma\to \Id$ is lax-monoidal,
\item[(ii)] the adjunction map $\epsilon \colon N\circ\Gamma\to \Id$ is colax-monoidal.
\end{itemize}
}
\end{lemma}
\proof
The claim (i) is Lemma 2.11 of [SchS03]. The claim (ii) is proven analogously, we present here the proof for completeness. We need to prove the commutativity of the diagram
\begin{equation}\label{hkm}
\xymatrix{
N\Gamma(X\otimes Y)\ar[r]^{\varphi}\ar[rrd] & N(\Gamma(X)\otimes \Gamma(Y))\ar[r]^{AW}&N\Gamma(X)\otimes N\Gamma(Y)\ar[d]^{\epsilon\otimes\epsilon}\\
&&X\otimes Y
}
\end{equation}
Here $\varphi$ is the colax-monoidal structure on $\Gamma$ adjoint to the shuffle lax-monoidal structure on $N$, see \eqref{cg}.
The explicit formula for $\varphi$ is:
\begin{equation}
\Gamma(X\otimes Y)\xrightarrow{\epsilon^{-1}\otimes\epsilon^{-1}}\Gamma(N\Gamma(X)\otimes N\Gamma(Y))\xrightarrow{\nabla}\Gamma N(\Gamma(X)\otimes\Gamma(Y))\xrightarrow{\eta}\Gamma(X)\otimes \Gamma(Y)
\end{equation}
Now the horizontal composition map in \eqref{hkm} is:
\begin{equation}\label{hkm2}
\begin{aligned}
\ &N\Gamma(X\otimes Y)\xrightarrow{\epsilon^{-1}\otimes\epsilon^{-1}}N\Gamma(N\Gamma(X)\otimes N\Gamma(Y))\xrightarrow{\nabla}N\boxed{\Gamma N}(\Gamma(X)\otimes\Gamma(Y))\xrightarrow{\eta^{-1}}N(\Gamma(X)\otimes \Gamma(Y))\\
&\xrightarrow{AW}N\Gamma(X)\otimes N\Gamma(Y)
\end{aligned}
\end{equation}
where the map $\eta^{-1}$ is applied to the ``boxed'' $\Gamma N$ factor.

Now the idea is to use the identity $AW\circ \nabla=\Id$ (Lemma \ref{before} (1.)) to ``cancel'' the second and the fourth arrows in \eqref{hkm2}.
We have:
\begin{lemma}\label{help}{\itshape
The following two compositions are equal:
\begin{equation}\label{comp1}
N\boxed{\Gamma N}(\Gamma(X)\otimes\Gamma(Y))\xrightarrow{\eta^{-1}}N(\Gamma(X)\otimes \Gamma(Y))
\xrightarrow{AW}N\Gamma(X)\otimes N\Gamma(Y)
\end{equation}
and
\begin{equation}\label{comp2}
\boxed{N \Gamma}N(\Gamma(X)\otimes\Gamma(Y))\xrightarrow{AW}\boxed{N\Gamma}(N\Gamma(X)\otimes N\Gamma(Y))\xrightarrow{\epsilon}N\Gamma(X)\otimes N\Gamma(Y)
\end{equation}
where  in the first (corresp., second) equation the map $\eta^{-1}$ (corresp., $\epsilon$) is applied to the boxed factors.
}
\end{lemma}
Clearly Lemma \ref{dkmonoidal} (ii) follows from Lemma \ref{help} and Lemma \ref{before} (1.).
\proof
For an adjoint {\it equivalence} $L\colon \mathscr{C}\rightleftarrows\mathscr{D}\colon R$ with the adjunction {\it isomorphisms}
$\epsilon\colon LR\to\Id$ and $\eta\colon\Id\to RL$, the two arrows $\boxed{LR}L\xrightarrow{\epsilon}L$ and $L\boxed{RL}\xrightarrow{\eta^{-1}}L$ coincide.

\qed

\begin{remark}{\rm
The adjunction map $\eta\colon\Id\to\Gamma\circ N$ is both lax-monoidal and colax-monoidal {\it only up to a homotopy}, see Remark 2.14 in [SchS03].
The reason is that the another order composition $\nabla\circ AW$ is equal to identity only up to a homotopy.
}
\end{remark}

\section{\sc The functor $L^\mon$ for $L=N$}
Here we compute explicitly the functor $L^\mon$ (defined in general as a co-equalizer \eqref{quotient_def}), for the functor $L=N$, the normalized chain complex functor in the Dold-Kan correspondence
$$
N\colon \mathscr{M}od(\mathbb{Z})^\Delta\rightleftarrows \mathscr{C}^-(\mathbb{Z})\colon\Gamma
$$
The functor $L^\mon$ here is defined out of the Alexander-Whitney colax-monoidal structure $AW$ on the functor $N$, see \eqref{aw}, \eqref{quotient_def}.
Theorems \ref{another},\ref{keytheorem} proven here are essentially used in Section \ref{sectionmain}, in course of the proof of the Main Theorem \ref{theor_gr_alg} for $\mathbb{Z}_{\le 0}$-graded dg algebras. In fact, they are used at the very end of the proof of Theorem \ref{theor_gr_alg}, in the proof of Proposition \ref{propref}(iii), see Section \ref{finallink}.

We use the notations: $L=N$, $R=\Gamma$, where $N,\Gamma$ are the functors in the Dold-Kan correspondence.
\subsection{\sc }
Consider the Alexander-Whitney colax-monoidal structure $AW\colon L(X\otimes Y)\to L(X)\otimes L(Y)$.
It follows from Proposition \ref{before}(1.) that the map $AW$ is {\it surjective} for any $X,Y$.
Thus, the map of monoids $\beta\colon T_\dg(L(T_\Delta X))\to T_\dg(LX)$, defined on the generators $L(T_\Delta X)$ as the iterated Alexander-Whitney map (see \eqref{themapbeta}, with $c_L=AW$) is {\it surjective}.

Denote
\begin{equation}\label{kx}
\mathscr{K}(X)=\Ker\bigl(\beta\colon T_\dg(L(T_\Delta X))\to T_\dg(LX)\bigr)
\end{equation}

Now let $X$ be a monoid in $\mathscr{M}od(\mathbb{Z})^\Delta$.
Consider the iterated product
\begin{equation}\label{mn}
m_n\colon \underbrace{X\otimes X\otimes\dots\otimes X}_{n\text{ factors}}\to X
\end{equation}
(with $m=m_2$).

The map of monoids $\alpha\colon T_\dg(L(T_\Delta X))\to T_\dg(LX)$ is defined on the generators $L(X^{\otimes n})\subset L(T_\Delta X)$ as $(m_n)_*\colon L(X^{\otimes n})\to L(X)$.

Denote by
\begin{equation}\label{ix}
\mathscr{I}(X)=\alpha(\mathscr{K}(X))
\end{equation}

Then $\mathscr{I}(X)\subset T_\dg(LX)$. Consider now the monoid structure on $LX$, defined by \eqref{monoidref} with the shuffle lax-monoidal structure $\nabla$ on $L$. This monoid structure defines the iterated product map
\begin{equation}\label{mnabla}
m_\nabla\colon T_\dg(LX)\to LX
\end{equation}

Finally, denote
\begin{equation}\label{jx}
\mathscr{J}(X)=m_\nabla(\mathscr{I}(X))\subset LX
\end{equation}

In fact, both $\mathscr{I}(X)\subset T_\dg(LX)$ and $\mathscr{J}(X)\subset LX$ are monoid-ideals.

We give the following explicit form for the functor $L^\mon$, for $L=N$:
\begin{theorem}\label{another}{\itshape
The functor $L^\mon\colon \Mon\mathscr{M}od(\mathbb{Z})^\Delta\to\Mon\mathscr{C}(\mathbb{Z})^-$, defined by \eqref{quotient_def}, is isomorphic to the functor $\tilde{L}^\mon$, defined for a monoid $X$ in $\mathscr{M}od(\mathbb{Z})^\Delta$ as
\begin{equation}\label{ltilde}
\tilde{L}^\mon(X)=L(X)/\mathscr{J}(X)
\end{equation}
where in the right-hand side there is a quotient-monoid of the monoid $L(X)$ (defined out of the lax-monoidal structure $\nabla$ on $L$) by the monoid-ideal $\mathscr{J}(X)$.
}
\end{theorem}

\begin{proof}
Let $X$ be a monoid in $\mathscr{M}od(\mathbb{Z})^\Delta$. We want to compute the co-equalizer \eqref{quotient_def} as the quotient-dg-algebra of $T_\dg(L(X))$ by some dg {\it ideal}
$I$,
\begin{equation}
L^\mon(X)=T_\dg(L(X))/I
\end{equation}
The ideal $I$ is spanned, by definition of co-equalizer, by elements
\begin{equation}\label{ideal}
I=\langle \alpha(t)-\beta(t),\  t\in T_\dg(L(T_\Delta X))\rangle
\end{equation}
The problem is how to compute the ideal $I\subset T_\dg(LX)$ and the quotient $T_\dg(LX)$ explicitly.

Consider $A\in (LX)^{\otimes n}$, $n\ge 1$. The map $AW^{\circ (n-1)}\colon L(X^{\otimes n})\to (LX)^{\otimes n}$ is surjective, by Proposition \ref{before}(1.). Moreower, this Proposition gives a section of the projection $AW^{\circ (n-1)}$, given as
\begin{equation}
s(A)=\nabla^{\circ (n-1)}(A)
\end{equation}
where $\nabla$ is the shuffle lax-monoidal structure on the functor $L$.

Consider $t=s(A)$ in \eqref{ideal}, which gives, for any $A$, a particular element in the ideal $I$.
Compute this element. One has $\beta(s(A))=A$ by construction of $s(A)$. Next, $\alpha(s(A))$ is an element in $LX$. This element is nothing but the iterated product $m_\nabla$ in the monoid $L(X)$, defined by \eqref{monoidref} out of the lax-monoidal structure $\nabla$ on $L$.

Thus, for any $n\ge 2$ and for any $A\in (LX)^{\otimes n}$, one has
\begin{equation}
A-m_\nabla^{\circ (n-1)}(A)\in I
\end{equation}

As the map $\beta$ is surjective, and for any target element we found some its pre-image, one has:
\begin{equation}
I=\langle I_0,\ \alpha(\mathscr{K}(X))\rangle
\end{equation}
where $I_0$ is the vector space spanned by all $A-m_\nabla^{\circ (n-1)}(A)$ for all $A$ and all $n\ge 2$.
\end{proof}
In our notations, $\mathscr{I}(X)=\alpha(\mathscr{K}(X))$, see \eqref{kx}, \eqref{ix}.

We have proved that
$$
L^\mon(X)=T_\dg(LX)/\langle I_0,\ \mathscr{I}(X)\rangle
$$
The latter quotient-monoid is isomorphic to $LX/\mathscr{J}(X)$, see \eqref{jx}. We are done.
\qed
\subsubsection{\sc Example}
Consider the value of the functor $L^\mon$ on a {\it free} monoid $X=T_\Delta(V)$, generated by a simplicial abelian group $V$.
Due to the commutative diagram \eqref{leftmon},
\begin{equation}\label{leftfree}
L^\mon(T_\Delta(V))=T_\dg(L(V))
\end{equation}

On the other hand, our result in Theorem \ref{another} gives an isomorphism $L^\mon=\tilde{L}^\mon$, with formula for $\tilde{L}^\mon$, given in \eqref{ltilde}. It is interesting to compare $\tilde{L}^\mon(T_\Delta(V))$, given in \eqref{ltilde}, with formula \eqref{leftfree}.

We get some identity of the form
\begin{equation}\label{identity}
L(T_\Delta(V))/\mathscr{J}(T_\Delta V)=T_\dg(L(V))
\end{equation}

It seems the latter identity is rather non-trivial combinatorial fact, if one tries to prove it directly.

\subsection{\sc }
Here we specify Theorem \ref{another}, as follows (see Theorem \ref{keytheorem} below).

Let $L\colon\mathscr{C}\rightleftarrows\mathscr{D}\colon R$ be an adjunction between monoidal categories, and let $\varphi$ be a lax-monoidal structure on $R$. Then the functor $L^\mon\colon \mathbf{Mon}(\mathscr{C})\to\mathbf{Mon}(\mathscr{D})$, defined in \eqref{quotient_def}, is left adjoint to $R^\mon$, defined as in \eqref{monoidref}. Then, the functor $R^\mon$ is equal to $R$ on the underlying objects, and the lax-monoidal structure $\varphi$ on $R$ defines a lax-monoidal structure $\varphi^\mon$ on $R^\mon$. Consequently, the left adjoint functor $L^\mon$ comes with the colax-monoidal structure, adjoint to $\varphi^\mon$. We refer to this colax-monoidal structure on $L^\mon$ as {\it canonical}.

Here we describe explicitly this colax-monoidal structure on the functor $\tilde{L}^\mon$, $L^\mon\simeq \tilde{L}^\mon$.

The shuffle lax-monoidal structure $\nabla$ makes $L(X)$ a monoid, for any monoid $X$, by \eqref{monoidref}. We always assume this monoid structure on $LX$, and denote the functor $L\colon \Mon\mathscr{M}od(\mathbb{Z})^\Delta\to \Mon\mathscr{C}(\mathbb{Z})^-$ by $L_\nabla$. The pair $(AW,\nabla)$ of lax-monoidal and colax-monoidal structures on the functor $L$ obeys the {\it bialgebra axiom}, see Theorem \ref{mytheor}. Then Lemma \ref{lemmaxx} says that the Alexander-Whitney map
$AW\colon L_\nabla(X\otimes Y)\to L_\nabla(X)\otimes L_\nabla(Y)$, for monoids $X$ and $Y$, is {\it a map of monoids}, and therefore defines a colax-monoidal structure on the functor $L_\nabla\colon\mathbf{Mon}(\mathscr{M}od(\mathbb{Z})^\Delta)\to\mathbf{Mon}(\mathscr{C}(\mathbb{Z})^-)$.

We claim that this colax-monoidal structure on $L_\nabla$ descents to a colax-monoidal structure on the functor $\tilde{L}^\mon$, and this colax-monoidal structure is the one which $\tilde{L}^\mon$ enjoys out of isomorphism of functors $L^\mon\simeq \tilde{L}^\mon$, given in Theorem \ref{another}, and out of the canonical colax-monoidal structure on $L^\mon$.
\comment
Indeed, the monoid $\tilde{L}^\mon(X)$ is the quotient-monoid $L_\nabla(X)/\mathscr{J}(X)$.

Consider, for monoids $X,Y$, the composition
\begin{equation}\label{eqcomposition}
L_\nabla(X\otimes Y)\xrightarrow{AW}L_\nabla(X)\otimes L_\nabla(Y)\rightarrow (L_\nabla(X)/\langle\mathscr{I}_2(X),\dots\rangle)\otimes
(L_\mon(Y)/\langle \mathscr{I}_2(Y),\dots\rangle)
\end{equation}
\begin{lemma}\label{descent}{\itshape
The ideal of $L_\mon(X\otimes Y)$, generated by $\mathscr{I}_2(X\otimes Y),\mathscr{I}_3(X\otimes Y),\dots$, goes to 0 by the composition \eqref{eqcomposition}.
}
\end{lemma}
\begin{proof}
We prove two sub-lemmas.
\begin{sublemma}\label{sublemma1}{\itshape
Let $X,Y$ be two objects in $\mathscr{M}od(\mathbb{Z})^\Delta$, $n\ge 1$.
Then the following diagram is commutative:
\begin{equation}\label{newdiagram1}
{\scriptsize\xymatrix{
{L(X\otimes Y)\otimes\dots \otimes L(X\otimes Y)}\ar[r]^{AW^{\otimes n}}&{L(X)\otimes L(Y)\otimes\dots\otimes L(X)\otimes L(Y)}\ar[r]^{\sigma_n}&{(L(X)\otimes\dots\otimes L(X))\otimes(L(Y)\otimes\dots\otimes L(Y))}\\
{L((X\otimes Y)\otimes\dots \otimes (X\otimes Y))}\ar[r]^{\sigma_n}\ar[u]^{AW^{\circ (n-1)}}&L((X\otimes \dots\otimes X)\otimes(Y\otimes\dots\otimes Y))\ar[r]^{AW}&  L(X\otimes\dots\otimes X)\otimes L(Y\otimes\dots\otimes Y)\ar[u]_{AW^{\circ (n-1)}\otimes AW^{\circ (n-1)}}
}}
\end{equation}
where there are $n$ of $X$'s and $n$ of $Y$'s in each entry of the diagram.
}
\end{sublemma}
\begin{proof}
We should be especially careful here, because the Alexander-Whitney colax-monoidal structure $AW$ is not symmetric (it is symmetric only up to a homotopy, see Proposition \ref{before}, (4.)). We need one more sub-lemma.
\begin{sublemma}\label{sublemma1bis}{\itshape
Let $X,Y,Z,W\in \mathscr{M}od(\mathbb{Z})^\Delta$. Then the following diagram is commutative:
\begin{equation}\label{newdiagram3}
{\scriptsize\xymatrix{
L(X\otimes Y)\otimes L(Z\otimes W)\ar[rr]^{AW\otimes AW}&&L(X)\otimes L(Y)\otimes L(Z)\otimes L(W)\ar[rr]^{\sigma_{L(Y),L(Z)}}&&L(X)\otimes L(Z)\otimes L(Y)\otimes L(W)\\
L(X\otimes Y\otimes Z\otimes W)\ar[u]^{AW}\ar[rr]^{\sigma_{Y,Z}}&&L(X\otimes Z\otimes Y\otimes W)\ar[rr]^{AW}&&L(X\otimes Z)\otimes L(Y\otimes W)\ar[u]_{AW\otimes AW}
}}
\end{equation}
}
\end{sublemma}
\begin{proof}
We use the definition of the Alexander-Whitney map, given in \eqref{aw}:
$$
AW(a^k\otimes b^k)=\sum_{i+j=k}d_\fin^ia^k\otimes d_0^jb^k
$$
and the identity
\begin{equation}
d_0d_\fin=d_\fin d_0
\end{equation}
Then the both pathes in diagram \eqref{newdiagram3}, applied to $x^k\otimes y^k\otimes z^k\otimes w^k$ (all of degree $k$) gives
\begin{equation}
\sum_{a+b+c+d=2k}\sum_{\substack{b_1+b_2=b\\c_1+c_2=c}}d_\fin^a(x^k)\otimes d_0^{c_1}d_\fin^{c_2}(z^k)\otimes d_0^{b_1}d_\fin^{b_2}(y^k)\otimes d_0^d(w^k)
\end{equation}
\end{proof}
Now the statement of Sub-lemma \ref{sublemma1} follows immediately from  Sub-lemma \ref{sublemma1bis}.
\end{proof}
\begin{sublemma}\label{sublemma2}{\itshape
Let $X,Y$ be monoids in $\mathscr{M}od(\mathbb{Z})^\Delta$, and $n\ge 2$. The following diagram is commutative:
\begin{equation}\label{newdiagram2}
{\scriptsize\xymatrix{
{L((X\otimes Y)\otimes\dots \otimes (X\otimes Y))}\ar[r]^{\sigma_n}\ar[d]_{m_n(X\otimes Y)}&L((X\otimes \dots\otimes X)\otimes(Y\otimes\dots\otimes Y))\ar[r]^{AW}&  L(X\otimes\dots\otimes X)\otimes L(Y\otimes\dots\otimes Y)\ar[d]^{m_n(X)\otimes m_n(Y)}\\
L(X\otimes Y)\ar[rr]^{AW}&&L(X)\otimes L(Y)
}}
\end{equation}
where there are $n$ of $X$'s and $n$ of $Y$'s in each entry of the upper line.
}
\end{sublemma}
\begin{proof}
Consider for simplicity the case $n=2$; the general case is completely analogous.

First of all, for monoids $X$ and $Y$ in a symmetric monoidal category $\mathscr{M}$, there is an isomorphism
\begin{equation}
\xi\colon (X\otimes Y)\otimes (X\otimes Y)\xrightarrow{\sigma_{23}}(X\otimes X)\otimes (Y\otimes Y)
\end{equation}
given by transposition of the second and the third factors, such that the diagram
\begin{equation}
\xymatrix{
(X\otimes Y)\otimes (X\otimes Y)\ar[rr]^{\xi}\ar[dr]_{m_{X\otimes Y}}&&(X\otimes X)\otimes (Y\otimes Y)\ar[dl]^{m_X\otimes m_Y}\\
&X\otimes Y
}
\end{equation}
where $m_A$ is the product in the monoid $A$.

Therefore, we need to prove that the diagram
\begin{equation}\label{wearedone}
\xymatrix{
L((X\otimes X)\otimes (Y\otimes Y))\ar[d]_{(m_X\otimes m_Y)_*}\ar[r]^{c_L}&L(X\otimes X)\otimes L(Y\otimes Y)\ar[d]^{(m_X)_*\otimes(m_Y)_*}\\
L(X\otimes Y)\ar[r]^{c_L}&L(X)\otimes L(Y)
}
\end{equation}
commutes. But the commutativity of \eqref{wearedone} just expresses that $c_L\colon L(A\otimes B)\to L(A)\otimes L(B)$ is a map of bifunctors.
\end{proof}
{\it We continue to prove Lemma \ref{descent}}.

It is sufficient to show that the {\it subcomplexes} $\mathscr{I}_2,\mathscr{I}_3,\dots$ of $L(X\otimes Y)$ are mapped to 0 by the composition \eqref{eqcomposition}. Let $\kappa\in \mathscr{I}_n\subset L(X\otimes Y)$. Then $\kappa=m_n(\upsilon)$ for some $\upsilon\in L(\underbrace{(X\otimes Y)\otimes\dots\otimes(X\otimes Y)}_{n\text{ factors}})$ which belongs to the kernel of the left vertical arrow $AW^{\circ (n-1)}$ of diagram \eqref{newdiagram1}. It is enough to prove, due to the commutativity of \eqref{newdiagram2}, that the composition $AW\circ \sigma_n(\upsilon)$ of the lower column maps of \eqref{newdiagram1} (which is the same that the upper column maps of \eqref{newdiagram2}) belongs to the kernel of the right vertical map in \eqref{newdiagram1}, $AW\circ\sigma_n(\upsilon)\in\mathrm{Ker}(AW^{\circ(n-1)}\otimes AW^{\circ (n-1)})$. The latter claim immediately follows from the commutativity of diagram \eqref{newdiagram1}.
\end{proof}
Lemma \ref{descent} implies that the Alexander-Whitney map $AW$ descents to a colax-monoidal structure $\wtilde{AW}$ on the functor $\tilde{L}^\mon$:
\begin{equation}\label{awtilde}
\wtilde{AW}\colon \tilde{L}^\mon(X\otimes Y)\to \tilde{L}^\mon(X)\otimes\tilde{L}^\mon(Y)
\end{equation}

\endcomment

\begin{theorem}\label{keytheorem}{\itshape
\begin{itemize}
\item[(A)] The Alexander-Whitney colax-monoidal structure on the functor $L_\nabla$ descents to a colax-monoidal structure on the functor
$\tilde{L}^\mon\simeq L_\nabla/\mathscr{J}$,
\item[(B)]
Within the isomorphism of functors $L^\mon$ and $\tilde{L}^\mon$, given in Theorem \ref{another}, the colax-monoidal structure on $\tilde{L}^\mon$ from part (A) is the one which is corresponded to the canonical colax-monoidal structure on $L^\mon$.
\end{itemize}
}
\end{theorem}
We prove this Theorem in Section \ref{section45} below, after formulating and proving a more general Theorem \ref{keytheorembis} in Section \ref{section44}.

\subsection{\sc The colax-monoidal structure on $L^\mon$, the general case}\label{section43}
Let $\mathscr{C},\mathscr{D}$ be general {\it symmetric} monoidal categories, and let
$$
L\colon \mathscr{C}\rightleftarrows \mathscr{D}\colon R
$$
be a pair of adjoint functors, with $R$ the right adjoint.

Suppose a lax-monoidal structure $\varphi$ on $R$ is given, so that the functor
$R^\mon\colon \Mon\mathscr{D}\to\Mon\mathscr{C}$ is defined, as in \eqref{monoidref}.
The functor $R^\mon$ comes with a lax-monoidal structure $\varphi^\mon$, equal to $\varphi$ on the underlying objects.

Denote by $c_L$ the colax-monoidal structure on the functor $L$, adjoint to the lax-monoidal structure $\varphi$ on $R$.

Recall formula \eqref{quotient_def} for the functor $L^\mon$ left adjoint to $R^\mon$.
We call {\it canonical} the colax-monoidal structure on $L^\mon$, adjoint to the lax-monoidal structure $\varphi^\mon$ on $R^\mon$.
Here we discuss the following question: how to express the canonical colax-monoidal structure on $L^\mon$ in terms of the co-equalizer definition \eqref{quotient_def}?

Consider the co-equalizer diagram for $L^\mon$:
\begin{equation}\label{quotient_bis}
\xymatrix{
T_{\mathscr{D}}(L(T_{\mathscr{C}}(X)))\ar@<1ex>[r]^{\ \ \ \alpha} \ar@<-1ex>[r]_{\ \ \ \beta} &T_{\mathscr{D}}(LX)
}
\end{equation}
Denote the diagram \eqref{quotient_bis} by $\mathcal{D}(X)$, with $L^\mon(X)=\colim \mathcal{D}(X)$.
As well, denote by $\mathcal{D}_1(X)$ and $\mathcal{D}_2(X)$ the left-hand (corresp., the right-hand) algebras in the diagram $\mathcal{D}(X)$.

Firstly, we construct, under some assumptions, a map
\begin{equation}\label{not1}
\Psi_{X,Y}\colon L^\mon(X\otimes Y)\to \colim(\mathcal{D}(X)\otimes \mathcal{D}(Y))
\end{equation}
where $\mathcal{D}(X)\otimes \mathcal{D}(Y)$ is the diagram
\begin{equation}\label{colimotimes}
\xymatrix{
\mathcal{D}_1(X)\otimes\mathcal{D}_1(Y)\ar@<1ex>[rr]^{\alpha_X\otimes\alpha_Y} \ar@<-1ex>[rr]_{\beta_X\otimes \beta_Y}&&
\mathcal{D}_2(X)\otimes\mathcal{D}_2(Y)
}
\end{equation}
see Lemma \ref{keypsi}.

Next, there is a canonical map
\begin{equation}\label{not2}
i_{X,Y}\colon \colim(\mathcal{D}(X)\otimes\mathcal{D}(Y))\to \colim(\mathcal{D}(X))\otimes\colim(\mathcal{D}(Y))=L^\mon(X)\otimes L^\mon(Y)
\end{equation}
Indeed, by the universal property of colimit, there are maps $i_1(X)\colon \mathcal{D}_1(X)\to\colim\mathcal{D}(X)$ and
$i_2(X)\colon \mathcal{D}_2(X)\to\colim\mathcal{D}(X)$, and analogous maps for $Y$. Then the maps
$$
i_1(X)\otimes i_1(Y)\colon \mathcal{D}_1(X)\otimes\mathcal{D}_1(Y)\to\colim(\mathcal{D}(X))\otimes \colim(\mathcal{D}(Y))
$$
and
$$
i_2(X)\otimes i_2(Y)\colon \mathcal{D}_2(X)\otimes\mathcal{D}_2(Y)\to\colim(\mathcal{D}(X))\otimes \colim(\mathcal{D}(Y))
$$
define the map $i_{X,Y}$ in \eqref{not2}.

The composition
\begin{equation}\label{defthetadef}
\Theta_{X,Y}=i_{X,Y}\circ \Psi_{X,Y}\colon L^\mon(X\otimes Y)\to L^\mon(X)\otimes L^\mon(Y)
\end{equation}
which is defined once $\Psi_{X,Y}$ is, will enjoy the colax-monoidal property, and is isomorphic to the canonical colax-monoidal structure on $L^\mon$.

Turn back to the map $\Psi_{X,Y}$, see \eqref{not1}.

To define a map of monoids $\Psi_{X,Y}\colon L^\mon(X\otimes Y)\to \colim(\mathcal{D}(X)\otimes\mathcal{D}(Y))$, it is enough
to define two maps {\it of monoids}
$$
\Psi^\prime_{X,Y}\colon T_{\mathscr{D}}(L(T_{\mathscr{C}}(X\otimes Y)))\to T_{\mathscr{D}}(L(T_{\mathscr{C}}(X)))\otimes T_{\mathscr{D}}(L(T_{\mathscr{C}}(Y)))
$$
and
$$
\Psi^{\prime\prime}_{X,Y}\colon T_{\mathscr{D}}(L(X\otimes Y))\to T_{\mathscr{D}}(LX)\otimes T_{\mathscr{D}}(LY)
$$
such that the diagrams
\begin{equation}\label{thetadiagram1}
\xymatrix{
T_{\mathscr{D}}(L(T_{\mathscr{C}}(X\otimes Y)))\ar[rr]^{\Psi^{\prime}_{X,Y}}\ar[d]_{\alpha}&& T_{\mathscr{D}}(L(T_{\mathscr{C}}(X)))\otimes T_{\mathscr{D}}(L(T_{\mathscr{C}}(Y)))\ar[d]^{\alpha\otimes \alpha}\\
T_{\mathscr{D}}(L(X\otimes Y))\ar[rr]^{\Psi^{\prime\prime}_{X,Y}}&& T_{\mathscr{D}}(LX)\otimes T_{\mathscr{D}}(LY)
}
\end{equation}
and
\begin{equation}\label{thetadiagram2}
\xymatrix{
T_{\mathscr{D}}(L(T_{\mathscr{C}}(X\otimes Y)))\ar[rr]^{\Psi^{\prime}_{X,Y}}\ar[d]_{\beta}&& T_{\mathscr{D}}(L(T_{\mathscr{C}}(X)))\otimes T_{\mathscr{D}}(L(T_{\mathscr{C}}(Y)))\ar[d]^{\beta\otimes \beta}\\
T_{\mathscr{D}}(L(X\otimes Y))\ar[rr]^{\Psi^{\prime\prime}_{X,Y}}&& T_{\mathscr{D}}(LX)\otimes T_{\mathscr{D}}(LY)
}
\end{equation}
are commutative.

\smallskip\smallskip

Define the maps $\Psi^\prime_{X,Y}$ and $\Psi^{\prime\prime}_{X,Y}$ as the compositions

\begin{equation}\label{theta1}
{\scriptsize\xymatrix{
\Psi^{\prime}_{X,Y}\colon T_{\mathscr{D}}(L(T_{\mathscr{C}}(X\otimes Y)))\ar[r]^{\theta_*}&T_{\mathscr{D}}(L(T_{\mathscr{C}}(X)\otimes T_\mathscr{C}(Y)))\ar[r]^{(c_L)_*}&T_\mathscr{D}(L(T_\mathscr{C}(X))\otimes L(T_\mathscr{C}(Y)))\ar[r]^{\theta_*}&
T_\mathscr{D}(L(T_\mathscr{C}(X)))\otimes T_\mathscr{D}(L(T_\mathscr{C}(Y)))
}}
\end{equation}
and
\begin{equation}\label{theta2}
{\scriptsize\xymatrix{\Psi^{\prime\prime}_{X,Y}\colon T_{\mathscr{D}}(L(X\otimes Y))\ar[r]^{{c_L}_*}&T_\mathscr{D}(LX\otimes LY)\ar[r]^{\theta_*}&
T_\mathscr{D}(LX)\otimes T_\mathscr{D}(LY)
}}
\end{equation}
correspondingly.

\smallskip\smallskip

We note:
\begin{lemma}\label{moa}{\itshape
The both maps $\Psi_{X,Y}^\prime$ and $\Psi_{X,Y}^{\prime\prime}$ are maps of monoids.
}
\end{lemma}
\begin{proof}
As in the both cases the source monoids are free, to have maps of monoids it is enough to define them somehow on the generators $L(T_{\mathscr{C}}(X\otimes Y))$ (correspondingly, on $L(X\otimes Y)$), and then to extend them by multiplicativity. This is precisely how the definitions \eqref{theta1} and \eqref{theta2} are organized.
\end{proof}

Here ${c_L}_*$ is the map induced by the colax-monoidal structure $c_L$ on the functor $L$, and the map $\theta_*$ is induced by the map
$\theta\colon T(X\otimes Y)\to T(X)\otimes T(Y)$, see \eqref{cf}.

The both maps $c_L$ and $\theta$ are colax-monoidal, therefore, the compositions $\Psi^\prime$ and $\Psi^{\prime\prime}$ also are, in appropriate sense.

Before formulating the following Lemma, we give a definition:
\begin{defn}\label{defquasisym}{
Let $\mathscr{C}$ and $\mathscr{D}$ be {\it symmetric} monoidal categories, and let $F\colon\mathscr{C}\to\mathscr{D}$ be a functor.
A colax-monoidal structure $\Theta$ on $F$ is called {\it quasi-symmetric} if for any four objects $X,Y,Z,W$ in $\mathscr{C}$ the diagram
\begin{equation}\label{quasisym}
{\scriptsize\xymatrix{
F(X\otimes Y)\otimes F(Z\otimes W)\ar[r]^{\Theta\otimes \Theta}&F(X)\otimes F(Y)\otimes F(Z)\otimes F(W)\ar[r]^{\sigma_{F(Y),F(Z)}}&F(X)\otimes F(Z)\otimes F(Y)\otimes F(W)\\
F(X\otimes Y\otimes Z\otimes W)\ar[u]^{\Theta}\ar[r]^{\sigma_{Y,Z}}&F(X\otimes Z\otimes Y\otimes W)\ar[r]^{\Theta}&F(X\otimes Z)\otimes F(Y\otimes W)\ar[u]_{\Theta\otimes \Theta}
}}
\end{equation}
is commutative.
}
\end{defn}

\begin{lemma}\label{keypsi}{\itshape
For $\Psi^\prime$ and $\Psi^{\prime\prime}$ as defined in \eqref{theta1}, \eqref{theta2}:
\begin{itemize}
\item[(i)] the diagram \eqref{thetadiagram1} always commutes,
\item[(ii)] the diagram \eqref{thetadiagram2} commutes if the colax-monoidal structure $c_L$ on the functor $L$ is quasi-symmetric, see Definition \ref{quasisym}.
\end{itemize}
}
\end{lemma}
\begin{proof}
As all maps in the diagrams \eqref{thetadiagram1} and \eqref{thetadiagram2} are maps of monoids, with the source a free monoid, it is enough to prove the commutativity of these diagrams on the generators in the source(s).

Now the claim (i) is reduced to the following diagram, for $n\ge 2$:
\begin{sublemma}{\itshape
Let $X,Y$ be monoids in $\mathscr{C}$, and let $n\ge 2$. Then the following diagram is commutative:
\begin{equation}
{\scriptsize\xymatrix{
L((X\otimes Y)\otimes\dots\otimes (X\otimes Y))\ar[d]_{m_{X\otimes Y}}\ar[r]^{\sigma\ \ \ }&L((X\otimes\dots\otimes X)\otimes (Y\otimes\dots\otimes Y))\ar[r]^{c_L}&L(X\otimes\dots \otimes X)\otimes L(Y\otimes\dots\otimes Y)\ar[d]_{m_{X}\otimes m_{Y}}\\
L(X\otimes Y)\ar[rr]^{c_L}&&L(X)\otimes L(Y)
}}
\end{equation}
where there is $n$ of $X$'s and $n$ of $Y$'s in each entry of the diagram.
}
\end{sublemma}
\begin{proof}
The diagram
\begin{equation}
\xymatrix{L((X\otimes Y)\otimes\dots\otimes (X\otimes Y))\ar[rd]_{m_{X\otimes Y}}\ar[rr]^{\sigma}&&L((X\otimes\dots\otimes X)\otimes(Y\otimes\dots\otimes Y))\ar[dl]^{m_X\otimes m_Y}\\
&L(X\otimes Y)
}
\end{equation}
by trivial reasons. Therefore, we need to prove that the diagram
\begin{equation}\label{mofbif}
\xymatrix{
L((X\otimes\dots \otimes X)\otimes (Y\otimes \dots\otimes Y))\ar[r]^{c_L}\ar[d]_{m_X\otimes m_Y}&L(X\otimes\dots\otimes X)\otimes L(Y\otimes\dots\otimes Y)\ar[d]^{m_X\otimes m_Y}\\
L(X\otimes Y)\ar[r]^{c_L}&L(X)\otimes L(Y)
}
\end{equation}
commutes.

But $c_L\colon L(A\otimes B)\to L(A)\otimes L(B)$ is a map of bifunctors, what easily implies the commutativity of \eqref{mofbif}.
\end{proof}
The claim (i) of Lemma is proven. Let us prove (ii).

Written on the generators of free source monoid, the diagram \eqref{thetadiagram2} is reduced to the following diagram:
\begin{equation}\label{lastdiagram}
{\scriptsize\xymatrix{
L((X\otimes Y)\otimes \dots\otimes (X\otimes Y))\ar[d]_{\beta_{X\otimes Y}}\ar[r]^{\sigma\ \ \ }&L((X\otimes\dots\otimes X)\otimes (Y\otimes\dots\otimes Y))\ar[r]^{c_L}&L(X\otimes\dots\otimes X)\otimes L(Y\otimes\dots\otimes Y)\ar[d]^{c_L^{\circ (n-1)}\otimes c_L^{\circ (n-1)}}\\
L(X\otimes Y)\otimes \dots \otimes L(X\otimes Y)\ar[r]^{c_L^{\otimes n}\ \ }&L(X)\otimes L(Y)\otimes \dots\otimes L(X)\otimes L(Y)\ar[r]^{\sigma\ \ \ \ \ }&
(L(X)\otimes\dots\otimes L(X))\otimes (L(Y)\otimes\dots\otimes L(Y))
}}
\end{equation}
Suppose $c_L$ is quasi-symmetric. Then the diagram \eqref{quasisym} is commutative for any $X,Y,Z,W$ in $\mathscr{C}$. This commutativity implies the commutativity of \eqref{lastdiagram}.

Claim (ii) is proven.

Lemma is proven.
\end{proof}
Suppose that $c_L$ is quasi-symmetric. Then
we have defined a map $\Theta_{X,Y}\colon L^\mon(X\otimes Y)\to L^\mon(X)\otimes L^\mon(Y)$ on co-equalizers \eqref{quotient_bis}.

\begin{theorem}\label{keytheorembis}{\itshape
Let $\mathscr{C}$ and $\mathscr{D}$ be two symmetric monoidal categories, and let
$$
L\colon \mathscr{C}\rightleftarrows\mathscr{D}\colon R
$$
be a pair of adjoint functors, with $R$ the right adjoint. Let $\varphi$ be a lax-monoidal structure on $R$, such that the adjoint colax-monoidal structure $c_L$ on $L$ is quasi-symmetric (see Definition \ref{quasisym}). Define $\Psi_{X,Y}$ as in \eqref{theta1}, \eqref{theta2} above. Define $\Theta_{X,Y}=i_{X,Y}\circ \Psi_{X,Y}$. Then
$$
\Theta_{X,Y}\colon L^\mon(X\otimes Y)\to L^\mon(X)\otimes L^\mon(Y)
$$
is a colax-monoidal structure on the functor $L^\mon$, isomorphic to the canonical one.
}
\end{theorem}
We prove Theorem \ref{keytheorembis} in Section \ref{section44} below.

\subsection{\sc Proof of Theorem \ref{keytheorembis}}\label{section44}
The proof is essentially based on the theory of (very) weak adjoint functors, developed in Appendix B.
The reader is advised to read the statements of Appendix B, especially Theorem \ref{keytheorem2bisbis} before proceeding to read this Section.

We keep the notations of Section \ref{section43}. Denote by $I$ the category
\begin{equation}
\xymatrix{
a\ \bullet\ar@<-1ex>_{\beta}[r]\ar@<1ex>^{\alpha}[r]&\bullet\ b
}
\end{equation}
We introduce two diagrams indexed by $I$,
$$
\mathcal{D}\colon I\to\mathscr{F}un(\Mon\mathscr{C},\Mon\mathscr{D})
$$
and
$$
\mathcal{R}\colon I\to\mathscr{F}un(\Mon\mathscr{D},\Mon\mathscr{C})
$$
By definition, $\mathcal{D}$ is the diagram \eqref{quotient_def},
with $\mathcal{D}(a)(X)=\mathcal{D}_1(X)=T_\mathscr{D}(L(T_\mathscr{X}))$,
$\mathcal{D}(b)(X)=\mathcal{D}_2(X)=T_\mathscr{D}(LX)$, and with the maps $\alpha,\beta$ from \eqref{quotient_def}.

The diagram $\mathcal{R}$ is defined as the constant diagram,
$\mathcal{R}(a)(Y)=\mathcal{R}(b)(Y)=R^\mon(Y)$, with the both maps $\alpha$, $\beta$ equal to the constant maps.

The idea of proof of Theorem \ref{keytheorembis} is as follows. We construct a {\it very weak right adjoint pair structures} on $(\mathcal{D}_1,R^\mon)$ and on $(\mathcal{D}_2,R^\mon)$, such that they define a {\it very weak right adjoint pair of diagrams}
\begin{equation}\label{wapd}
\mathcal{D}\colon\mathscr{C}\rightleftarrows \mathscr{D}\colon\mathcal{R}
\end{equation}
(see Definitions \ref{adefn_1}, \ref{adefn_3} in Appendix B).

Next, we show that the pair $(\Psi^\prime, \varphi^\mon)$ of (co)lax-monoidal structures on the functors $(\mathcal{D}(a),R^\mon$, as well as the pair $(\Psi^{\prime\prime},\varphi^\mon)$ of (co)lax-monoidal structures on the functors $(\mathcal{D}(b),R^\mon)$, are {\it weak compatible}, see Definition \ref{adefn_2} (the colax-monoidal structures $\Psi^\prime$ and $\Psi^{\prime\prime}$ are defined \eqref{theta1}, \eqref{theta2}).

Then we are in the set-up of Corollary \ref{keycoroll}, and this Corollary immediately gives Theorem \ref{keytheorembis}.
\begin{remark}\label{sosna}{\rm
It is most likely that neither of pair of functors $(\mathcal{D}(a),R^\mon)$ and $(\mathcal{D}(b),R^\mon)$ has a natural structure of a weak right adjoint pair, see Definition \ref{adef1}. In fact, $\epsilon(a)$ and $\epsilon(b)$ are uniquely defined, see proof of Key-Lemma \ref{keylemmakt}, (1.). Then, most likely there do not exist any natural adjunctions $\eta(a)\colon \Id_{\mathscr{C}}\to R^\mon\mathcal{D}(a)$
and $\eta(b)\colon \Id_\mathscr{C}\to R^\mon\mathcal{D}(b)$, such that the composition \eqref{a2} is the identity, for $\epsilon(a)=\epsilon_1$ and $\epsilon(b)=\epsilon_2$ (see \eqref{epsilon1def} and \eqref{epsilon2def} below). Nevertheless, the ``corresponding'' colax-monoidal structures $\Psi^\prime$ and $\Psi^{\prime\prime}$ exist. This remark explains our decision to work with {\it very weak adjoint functors} of Appendix \ref{appveryweak} than with {\it weak adjoint functors} of Appendix \ref{appweak}.
}
\end{remark}

Our proof of Theorem \ref{keytheorembis} is based on the following
\begin{klemma}\label{keylemmakt}{\itshape
In the notations and assumptions of Theorem \ref{keytheorembis}:
\begin{itemize}
\item[1.] there are very weak right adjoint structures on the pairs of functors $(\mathcal{D}_1,R^\mon)$ and $(\mathcal{D}_2,R^\mon)$ such that they form a very weak right adjoint pair of diagrams \eqref{wapd},
\item[2.] the colax-monoidal structure $\Psi^\prime$ on the functor $\mathcal{D}_1$ and the colax-monoidal structure $\Psi^{\prime\prime}$ on the functor $\mathcal{D}_2$, form a diagram of colax-monoidal structures,
\item[3.] the pair $(\Psi^\prime,\varphi^\mon)$ of (co)lax-monoidal structures on the functors $(\mathcal{D}_1,R^\mon)$, and the pair $(\Psi^{\prime\prime},\varphi^\mon)$ of (co)lax-monoidal structures on the functors $(\mathcal{D}_2,R^\mon)$, are weakly compatible, see Definition \ref{adefn_2},
\item[4.] the colimit colax-monoidal structure $\colim_I(\Psi^\prime,\Psi^{\prime\prime})$ on the functor $L^\mon$ (which is well-defined by (2.)) is precisely the colax-monoidal structure $\Theta$ on $L^\mon$, see \eqref{defthetadef}.
\end{itemize}
}
\end{klemma}
We firstly show how Theorem \ref{keytheorembis} follows from the Key-lemma above, and then prove the Key-Lemma in the rest of this Subsection.

\smallskip\smallskip

{\it Proof of Theorem \ref{keytheorembis}:}

Equip the both functors $\mathcal{R}(a)=\mathcal{R}(b)=R^\mon$ with the lax-monoidal structure $\varphi^\mon$, with the identity maps between them. Then we get a {\it diagram of lax-monoidal structures}. Its limit is a lax-monoidal structure on the functor ${\lim}_I\mathcal{R}=R^\mon$, equal to $\varphi^\mon$. We know, by Key-lemma \ref{keylemmakt}, that the pairs of (co)lax-monoidal structures
$(\Psi^\prime,\varphi^\mon)$, and $(\Psi^{\prime\prime},\varphi^\mon)$ are weakly compatible, see Definition \ref{adefn_2}.
Now, by Theorem \ref{keytheorem2bisbis}, the colax-monoidal structure on the functor $L^\mon=\colim_I\mathcal{D}$ adjoint to the limit lax-monoidal diagram $\varphi^\mon$, and the colimit $\Psi_\colim$ of the colax-monoidal structures on $\mathcal{D}(a)$ and $\mathcal{D}(b)$ adjoint to $\varphi^\mon$, are also weakly compatible. By Key-lemma \ref{keylemmakt}(3.), the colax-monoidal structure $\Psi_\colim$ on $L^\mon$ is equal to the colax-monoidal structure $\Theta$ on $L^\mon$, see \eqref{defthetadef}.
We conclude, that the colax-monoidal structure $\Theta$ on $L^\mon$ is weakly compatible with the lax-monoidal structure $\varphi^\mon$ on $R^\mon$.

Now, the functors $(L^\mon, R^\mon)$ form a (genuine) adjoint pair of functors, with $\epsilon\colon L^\mon R^\mon\to\Id_{\mathscr{D}}$ equal to the limit $\epsilon_{\lim}$ of $\epsilon(a)$ and $\epsilon(b)$, and with some $\eta\colon \Id_{\mathscr{C}}\to R^\mon L^\mon$.
Therefore, Theorem \ref{keytheorem2bisbis} implies that $\Theta$ is isomorphic to the colax-monoidal structure on $L^\mon$, adjoint to the lax-monoidal structure $\varphi^\mon$ on $R^\mon$.

\qed

\smallskip

{\it Proof of Key-lemma \ref{keylemmakt}:}

{\it Proof of (1.):}

Recall our notations $\mathcal{D}_1(X)=T_\mathscr{D}(L(T_\mathscr{C} X))$ and $\mathcal{D}_2(X)=T_\mathscr{D}(LX)$, for $X\in\Mon\mathscr{C}$.
We need to construct morphisms of functors
\begin{equation}
\epsilon_1\colon \mathcal{D}_1\circ R^\mon\to\Id_{\Mon\mathscr{D}}
\end{equation}
and
\begin{equation}
\epsilon_2\colon\mathcal{D}_2\circ R^\mon\to\Id_{\Mon\mathscr{D}}
\end{equation}
such that the diagrams
\begin{equation}\label{comdiagone}
\xymatrix{
\mathcal{D}_1 R^\mon(Y)\ar[r]^{\ \ \ \epsilon_1}\ar[d]_{\alpha_*}&Y\ar[d]^{id}\\
\mathcal{D}_2 R^\mon (Y)\ar[r]^{\ \ \ \epsilon_2}&Y
}
\end{equation}
and
\begin{equation}\label{comdiagtwo}
\xymatrix{
\mathcal{D}_1 R^\mon(Y)\ar[r]^{\ \ \ \epsilon_1}\ar[d]_{\beta_*}&Y\ar[d]^{id}\\
\mathcal{D}_2 R^\mon (Y)\ar[r]^{\ \ \ \epsilon_2}&Y
}
\end{equation}
for any $Y\in\Mon\mathscr{D}$.

Recall the the functor $L^\mon$ is the co-equalizer \eqref{quotient_def}.
Therefore, there is a canonical map
\begin{equation}
p\colon\mathcal{D}_2(X)\to L^\mon(X)
\end{equation}
such that the two compositions $\alpha\circ p$ and $\beta\circ p$ are equal maps $\mathcal{D}_1(Y)\to L^\mon(Y)$.

Next, there is the adjunction map $\epsilon^\mon\colon L^\mon\circ R^\mon\to\Id_{\Mon\mathscr{D}}$.
Define $\epsilon_2$ as the composition
\begin{equation}\label{epsilon2def}
\epsilon_2=p_*\circ \epsilon^\mon\colon \mathcal{D}_2\circ R^\mon\to\Id_{\Mon\mathscr{D}}
\end{equation}
and define
\begin{equation}\label{epsilon1def}
\epsilon_1=\alpha_*\epsilon_2=\beta_*\epsilon_2\colon\mathcal{D}_1\circ R^\mon\to\Id_{\Mon\mathscr{D}}
\end{equation}
Then the diagrams \eqref{comdiagone} and \eqref{comdiagtwo} commute by the construction.

{\it Proof of (2.):} The claim means the commutativity of diagrams \eqref{thetadiagram1} and \eqref{thetadiagram2}.
It was proven in Lemma \ref{keypsi}, with assumption that the colax-monoidal structure $c_L$ on $L$, adjoint to the lax-monoidal structure $\varphi$ on $R$, is quasi-symmetric, see Definition \ref{defquasisym} (the latter is assumed in the statement of Key-Lemma \ref{keylemmakt}).

{\it Proof of (3.):} We need firstly to know an explicit expression for $\epsilon_2\colon\mathcal{D}_2R^\mon\to\Id_{\Mon\mathscr{D}}$.
\begin{sublemma}\label{ksl}{\itshape
In the above notations, the map $\epsilon_2\colon\mathcal{D}_2R^\mon\to\Id_{\Mon\mathscr{D}}$ is equal to the composition
\begin{equation}\label{thekeyanswer}
T_\mathscr{D}(L(R^\mon(Y)))=T_\mathscr{D}(LR(Y))\xrightarrow{\epsilon_*}T_\mathscr{D}(Y)\xrightarrow{m_Y}Y
\end{equation}
where the last map uses the monoid structure on $Y$.
}
\end{sublemma}
\begin{proof}
By our definition \eqref{epsilon2def}, the following composition
\begin{equation}\label{simple}
\begin{aligned}
\ &\Hom_{\Mon\mathscr{D}}(T_\mathscr{D}L((R^\mon Y),Z))\xleftarrow{t}\\
&\xleftarrow{t}\lim\Bigl((\Hom_{\Mon\mathscr{D}}(T_{\mathscr{D}}(L(R^\mon Y)),Z)\rightrightarrows\Hom_{\Mon\mathscr{D}}(T_\mathscr{D}(T_\mathscr{D}(L(R^\mon Y))),Z)\Bigr)\simeq\\
&\simeq\Hom_{\Mon\mathscr{D}}(L^\mon R^\mon (Y), Z)
\end{aligned}
\end{equation}
is induced by the map $\mathcal{D}_2(X)=T_\mathscr{D}(L(X))\to L^\mon(X)$ for $X\in\Mon\mathscr{C}$ (here $t$ is the universal limit map).

Thus, we need to find (a unique) $$\kappa_2\in \Hom_{\Mon\mathscr{D}}(T_{\mathscr{D}}(L(R^\mon Y)),Y) \text{ and } \kappa_1\in\Hom_{\Mon\mathscr{D}}(T_\mathscr{D}(T_\mathscr{D}(L(R^\mon Y))),Y)$$ such that $\alpha^*\kappa_2=\beta^*\kappa_2=\kappa_1$, and such that the corresponding element $\kappa\in \Hom_{\Mon\mathscr{D}}(L^\mon R^\mon Y,Y)$ is equal to $\epsilon^\mon$, or, in other words, is corresponded to the identity in $\Hom_{\Mon\mathscr{C}}(R^\mon Y,R^\mon Y)$.

We follow the proof of Lemma \ref{leftproof}, and find, step by step, corresponding representatives in (co)limits, starting with the identity map in
$\Hom_{\Mon\mathscr{C}}(R^\mon Y,R^\mon Y)$.

By Sub-lemma \ref{leftproofsub}, the monoid $R^\mon Y$ is isomorphic to
\begin{equation}
R^\mon(Y)\simeq\colim \Bigl(T_\mathscr{C}T_\mathscr{C}R^\mon(Y)\rightrightarrows T_\mathscr{C}R^\mon(Y)\Bigr)
\end{equation}
where the two arrows are as in Sub-lemma \ref{leftproofsub}.

We have:
\begin{equation}\label{iso1}
\begin{aligned}
\ &\Hom_{\Mon\mathscr{C}}(R^\mon Y,R^\mon Y)\simeq\\
&\lim\Bigl(\Hom_{\Mon\mathscr{C}}(T_\mathscr{C}R^\mon Y, R^\mon Y)\rightrightarrows\Hom_{\Mon\mathscr{C}}
(T_\mathscr{C}T_\mathscr{C}R^\mon Y, R^\mon Y)\Bigr)
\end{aligned}
\end{equation}
Consider $\omega_2\in \Hom_{\Mon\mathscr{C}}(T_\mathscr{C}R^\mon Y, R^\mon Y)$ defined as the map induced by the product in the monoid $R^\mon Y$, and consider $\omega_1\in \Hom_{\Mon\mathscr{C}}
(T_\mathscr{C}T_\mathscr{C}R^\mon Y, R^\mon Y)$ defined as the composition of the map
in $\Hom_{\Mon\mathscr{C}}(T_\mathscr{C}T_\mathscr{C}R^\mon Y, T_\mathscr{C}R^\mon Y)$ induced by the monoid structure in $T_\mathscr{C}R^\mon Y$, followed by $\omega_2$.

One easily checks that the pair $(\omega_1,\omega_2)$ is compatible in the limit diagram, and is corresponded, by the isomorphism \eqref{iso1}, to the identity map of $R^\mon Y$.

We proceed as in Lemma \ref{leftproof}:
\begin{equation}\label{iso2}
\begin{aligned}
\ &\lim\Bigl(\Hom_{\Mon\mathscr{C}}(T_\mathscr{C}R^\mon Y, R^\mon Y)\rightrightarrows\Hom_{\Mon\mathscr{C}}
(T_\mathscr{C}T_\mathscr{C}R^\mon Y, R^\mon Y)\Bigr)\simeq\\
&\lim\Bigl(\Hom_\mathscr{C}(R^\mon Y,R^\mon Y)\rightrightarrows \Hom_\mathscr{C}(T_\mathscr{C}R^\mon Y, R^\mon Y)\Bigr)
\end{aligned}
\end{equation}
Consider $\chi_2\in \Hom_\mathscr{C}(RY,RY)$ defined as the identity map, and consider $\chi_1\in \Hom_\mathscr{C}(T_\mathscr{C}(RY), RY)$ defined from the product map. One easily sees that $(\chi_1,\chi_2)$ is corresponded to $(\omega_1,\omega_2)$ under the isomorphism \eqref{iso2}.

The next step is:
\begin{equation}\label{iso3}
\begin{aligned}
\ &\lim\Bigl(\Hom_\mathscr{C}(R(Y),R(Y))\rightrightarrows \Hom_\mathscr{C}(T_\mathscr{C}R(Y), R(Y))\Bigr)\simeq\\
&\lim\Bigl(\Hom_\mathscr{D}(LR(Y), Y)\rightrightarrows \Hom_\mathscr{D}(L(T_\mathscr{C}(R Y)),Y)\Bigr)
\end{aligned}
\end{equation}
One sees that $(\sigma_1,\sigma_2)$ corresponded to $(\chi_1,\chi_2)$ under \eqref{iso3} are defined as
\begin{equation}
\sigma_1=\epsilon(L(\chi_1))\text{   and    }\sigma_2=\epsilon(L(\chi_2))
\end{equation}
where $\epsilon\colon LR\to\Id_\mathscr{D}$ is the adjunction.

Now
\begin{equation}\label{iso4}
\begin{aligned}
\ &\lim\Bigl(\Hom_\mathscr{D}(LR(Y), Y)\rightrightarrows \Hom_\mathscr{D}(L(T_\mathscr{C}(R Y)),Y)\Bigr)\simeq\\
&\lim\Bigl(\Hom_{\Mon\mathscr{D}}(T_\mathscr{D}(L(R(Y)),Y))\rightrightarrows \Hom_{\Mon\mathscr{D}}(T_\mathscr{D}(L(T_\mathscr{C}(RY))),Y)\Bigr)\simeq\\
&\Hom_{\Mon\mathscr{D}}(L^\mon(Y),Y)
\end{aligned}
\end{equation}
To define $(\o_1,\o_2)$ corresponded to $(\sigma_1,\sigma_2)$ by \eqref{iso4}, we recall that $Y$ is a monoid, and define $(\o_1,\o_2)$ as the maps from free monoids, induced by the maps $(\sigma_1,\sigma_2)$, and by the monoid structure in $Y$.

We finish to prove the Sub-lemma.

By \eqref{simple}, $\epsilon_2=\o_2$, and for $\o_2$ we have an explicit formula, which gives \eqref{thekeyanswer} immediately.

\end{proof}

{\it We continue to prove the claim (3.):} having the above Sub-lemma as a toolkit, it is not hard to prove the commutativity of the corresponding diagrams, by a direct check. We need to prove the commutativity of the diagrams:
\begin{equation}\label{kld1}
\xymatrix{
\mathcal{D}_1(R^\mon X\otimes R^\mon Y)\ar[r]^{\ \ \ \varphi^\mon}\ar[d]_{\Psi^\prime}&\mathcal{D}_1R^\mon(X\otimes Y)\ar[d]^{\epsilon_1}\\
\mathcal{D}_1R^\mon X\otimes \mathcal{D}_1R^\mon (Y)\ar[r]^{\ \ \ \ \ \ \ \ \epsilon_1\otimes\epsilon_1}&X\otimes Y
}
\end{equation}
and
\begin{equation}\label{kld2}
\xymatrix{
\mathcal{D}_2(R^\mon X\otimes R^\mon Y)\ar[r]^{\ \ \ \varphi^\mon}\ar[d]_{\Psi^{\prime\prime}}&\mathcal{D}_2R^\mon(X\otimes Y)\ar[d]^{\epsilon_2}\\
\mathcal{D}_2R^\mon X\otimes \mathcal{D}_2R^\mon (Y)\ar[r]^{\ \ \ \ \ \ \ \ \epsilon_2\otimes\epsilon_2}&X\otimes Y
}
\end{equation}
To prove the commutativity of \eqref{kld2}, divide it into two diagrams, as follows:
\begin{equation}\label{kld2bis}
\xymatrix{
T_\mathscr{D}L(RX\otimes RY)\ar[r]^{\varphi}\ar[d]_{c_L}   &T_\mathscr{D}(LR(X\otimes Y))\ar[d]^{\epsilon_*}\\
T_\mathscr{D}(LRX\otimes LRY)\ar[d]_{\Theta}\ar[r]^{\epsilon\otimes\epsilon}& T_\mathscr{D}(X\otimes Y)\ar[dd]^{m_{X\otimes Y}}\\
T_{\mathscr{D}}LR(X)\otimes T_\mathscr{D} LR(Y)\ar[d]_{\epsilon_*\otimes\epsilon_*}\\
T_\mathscr{D}X\otimes T_\mathscr{D}Y\ar[r]^{m_X\otimes m_Y}&X\otimes Y
}
\end{equation}
where $c_L$ is the colax-monoidal structure on $L$ adjoint to the lax-monoidal structure $\varphi$ on $R$, and the map $\Theta$ is \eqref{cf}.
The commutativity of the diagram \eqref{kld2bis} is equivalent to the commutativity of the diagram \eqref{kld2}, by Sub-lemma \ref{ksl}.
The commutativity of the both sub-diagrams of \eqref{kld2bis} is clear: the upper one commutes as $c_L$ and $\varphi$ are adjoint (co)lax-monoidal structures, and the lower one commutes by tautological reasons.

Now prove the commutativity of \eqref{kld1}. The diagram \eqref{kld1} is divided into two diagrams, as follows:
\begin{equation}
\xymatrix{
&\mathcal{D}_1(R^\mon X\otimes R^\mon Y)\ar[dl]_{\Psi^\prime}\ar[d]^{\alpha_*=\beta_*}\ar[r]^{\ \ \ \ \varphi^\mon}&\mathcal{D}_1R^\mon(X\otimes Y)\ar[d]^{\alpha_*=\beta_*}\\
\mathcal{D}_1R^\mon(X)\otimes \mathcal{D}_1R^\mon(Y)\ar[dr]_{\alpha_*\otimes\alpha_*}&\mathcal{D}_2(R^\mon X\otimes R^\mon Y)\ar[d]^{\Psi^{\prime\prime}}\ar[r]^{\ \ \ \ \varphi^\mon}&\mathcal{D}_2 R^\mon(X\otimes Y)\ar[d]^{\epsilon_2}\\
&\mathcal{D}_2 R^\mon(X)\otimes \mathcal{D}_2 R^\mon(Y)\ar[r]^{\ \ \ \ \ \ \ \ \ \epsilon_2\otimes\epsilon_2}&X\otimes Y
}
\end{equation}
The commutativity of the left ``triangle-like'' diagram is exactly the claim (2.) of the Key-Lemma.
The lower-right square diagram is the diagram \eqref{kld2}, and its commutativity is just proven above. The commutativity of the upper-right square diagram is a general nonsense.

We are done.

{\it Proof of (4.):} The colax-monoidal structure $\Theta$ on $L^\mon$ is the colimit of $\Psi^\prime$ and $\Psi^{\prime\prime}$ by definition, see Definition \ref{limlax}.

Key-lemma \ref{keylemmakt} is proven.

\qed

\subsection{\sc A proof of Theorem \ref{keytheorem}}\label{section45}
Here we prove Theorem \ref{keytheorem}.

We want to apply Theorem \ref{keytheorembis}. We need to show that the assumptions of this Theorem are fulfilled, in the case $L=N$, the normalized chain complex functor in the Dold-Kan correspondence, and $c_L=AW$, the Alexander-Whitney colax-monoidal structure on $L$.

\begin{lemma}\label{newaw}{\itshape
The Alexander-Whitney colax-monoidal structure $AW\colon L(X\otimes Y)\to L(X)\otimes L(Y)$ is quasi-symmetric, see Definition \ref{defquasisym}.
}
\end{lemma}
\begin{proof}
We need to prove that diagram \eqref{quasisym} commutes, for $F=L=N$, and $\Theta=AW$, and for any $X,Y,Z,W\in\mathscr{M}od(\mathbb{Z})^\Delta$.

To this end, recall the definition of the Alexander-Whitney map, see \eqref{aw}:
$$
AW(a^k\otimes b^k)=\sum_{i+j=k}d_\fin^ia^k\otimes d_0^jb^k
$$
As well, the identity
\begin{equation}
d_0d_\fin=d_\fin d_0
\end{equation}
is valid for any simplicial set.

Then the both pathes in diagram \eqref{quasisym}, applied to $x^k\otimes y^k\otimes z^k\otimes w^k$ (all of degree $k$) gives
\begin{equation}
\sum_{a+b+c+d=2k}\sum_{\substack{b_1+b_2=b\\c_1+c_2=c}}d_\fin^a(x^k)\otimes d_0^{c_1}d_\fin^{c_2}(z^k)\otimes d_0^{b_1}d_\fin^{b_2}(y^k)\otimes d_0^d(w^k)
\end{equation}

\end{proof}
Now we apply Theorem \ref{keytheorembis}.

Firstly we compute, in the particular case of $L=N$, the colimit $\colim(\mathcal{D}(X)\otimes\mathcal{D}(Y))$, see \eqref{colimotimes}.
As in the computation of $\colim \mathcal{D}(X)$ in Section 4.1, we write
\begin{equation}
\colim(\mathcal{D}(X)\otimes\mathcal{D}(Y))=((T_\dg(LX))\otimes T_\dg(LY))/I_{X,Y}
\end{equation}
for some ideal $I_{X,Y}\subset (T_\dg(LX))\otimes(T_\dg(LY))$. The same surjectivity argument reduces the answer to the computation of
$(m_{\nabla,X}\otimes m_{\nabla, Y})\circ (m_X\otimes m_Y) (\Ker(\beta_X\otimes \beta_Y))$.

Recall the following trivial observation: Let $V,W,V^\prime,W^\prime$ be vector spaces over a field $k$, and let $f\colon V\to W$, $g\colon V^\prime\to W^\prime$ be $k$-linear maps. Then
\begin{equation}\label{tf}
\Ker(f\otimes g)=(\Ker f)\otimes W+V\otimes (\Ker g)
\end{equation}

Due to \eqref{tf},
\begin{equation}
\Ker(\beta_X\otimes\beta_Y)=(\Ker\beta_X)\otimes T_\dg(L(T_\Delta Y)) + T_\dg(L(T_\Delta X))\otimes (\Ker \beta_Y)
\end{equation}
Consequently,
\begin{equation}
\colim(\mathcal{D}(X)\otimes\mathcal{D}(Y))=(\colim \mathcal{D}(X))\otimes (\colim \mathcal{D}(Y))=L^\mon(X)\otimes L^\mon(Y)
\end{equation}
and the map $i_{X,Y}$ (see \eqref{not2}) is an isomorphism in the case we consider.

We have the following commutative diagram, with $\Theta_{X,Y}$ given in Theorem \ref{keytheorembis}, and the vertical isomorphisms just described,
\begin{equation}
\xymatrix{
\colim\mathcal{D}(X\otimes Y)\ar[r]^{\Theta_{X,Y}\ \ \ }\ar[d]_{\simeq}&\colim(\mathcal{D}(X)\otimes\mathcal{D}(Y))\ar[d]^{\simeq}\\
L(X\otimes Y)/\mathscr{J}(X\otimes Y)\ar[r]^{\mathfrak{X}\ \ \ \ \ \ }&(L(X)/\mathscr{J}(X))\otimes (L(Y)/\mathscr{J}(Y))
}
\end{equation}
with undefined $\mathfrak{X}$. It follows from Theorem \ref{keytheorembis} that one necessarily has
$$
\mathfrak{X}=c_L=AW
$$
We are done.
\qed
\begin{remark}{\rm
Indirectly, we have proven that the Alexander-Whitney map
$$
AW\colon L(X\otimes Y)\to L(X)\otimes L(Y)
$$
where $X,Y$ are monoids, descents to the quotients
$$
L_\nabla(X\otimes Y)/\mathscr{J}(X\otimes Y)\to (L_\nabla(X)/\mathscr{J}(X))\otimes (L_\nabla(Y)/\mathscr{J}(Y))
$$
In other words, we proved that
\begin{equation}\label{lr}
AW(\mathscr{J}(X\otimes Y))\subset\langle \mathscr{J}(X)\otimes 1,\ 1\otimes \mathscr{J}(Y)\rangle
\end{equation}
We do not see any ``direct'' way to prove \eqref{lr}.
}
\end{remark}

\section{\sc Main theorem for $\mathbb{Z}_{\le 0}$-graded dg algebras}\label{sectionmain}
\subsection{}
Here we prove our main result:
\begin{theorem}\label{theor_gr_alg}{\itshape
Let $k$ be a field of any characteristic.
There is a functor $\mathfrak{R}\colon \mathscr{A}lg_k^{\le 0}\to \mathscr{A}lg_k^{\le 0}$ and a morphism of functors $w\colon \mathfrak{R}\to Id$ with the following properties:
\begin{itemize}
\item[1.] $\mathfrak{R}(A)$ is cofibrant, and $w\colon \mathfrak{R}(A)\to A$ is a quasi-isomorphism, for any $A\in \mathscr{A}lg_k^{\le 0}$,
\item[2.] there is a colax-monoidal structure on the functor $\mathfrak{R}$, such that all colax-maps $\beta_{A,B}\colon \mathfrak{R}(A\otimes B)\to \mathfrak{R}(A)\otimes \mathfrak{R}(B)$ are quasi-isomorphisms of dg algebras, and such that the diagram
    $$
    \xymatrix{
    \mathfrak{R}(A\otimes B)\ar[rr]^{\beta_{A,B}}\ar[rd]&&\mathfrak{R}(A)\otimes \mathfrak{R}(B)\ar[dl]\\
    &A\otimes B
    }
    $$
    is commutative,
\item[3.] the morphism $w(k^\udot)\colon \mathfrak{R}(k^\udot)\to k^\udot$ coincides with $\alpha\colon \mathfrak{R}(k^\udot)\to k^\udot$, where $\alpha$ is a part of the colax-monoidal structure {\rm (see Definition \ref{colax_intro})}, and $k^\udot=1_{\mathscr{A}lg_k^{\le 0}}$ is the dg algebra equal to the one-dimensional $k$-algebra in degree 0, and vanishing in other degrees.
\end{itemize}
}
\end{theorem}
The idea is to use the solution of the analogous problem for simplicial algebras (given in Theorem \ref{theor_simpl_alg} in Section 1), and ``transfer'' it to dg algebras using the Dold-Kan correspondence.

More precisely, consider the Dold-Kan correspondence
$$
N\colon \mathscr{M}od(\mathbb{Z})^\Delta\rightleftarrows \mathscr{C}(\mathbb{Z})^-\colon\Gamma
$$
The functors $N$ and $\Gamma$ form an adjoint equivalence of categories; therefore, we have some freedom which of these two functors to consider as the left (right) adjoint. We consider $\Gamma$ as the right adjoint. From now on, we use the notations: $\Gamma=R$, $N=L$.

The functor $L$ comes with the colax-monoidal (Alexander-Whitney) structure $AW$ and with the lax-monoidal (shuffle) structure $\nabla$. They induce a lax-monoidal structure $\ell_R$ and a colax-monoidal structure $c_R$ on the functor $R$ by the adjunction, as is explained in \eqref{lg} and \eqref{cg}.

Consider the functor $R^\mon\colon \mathscr{A}lg^{\le 0}_k\to \mathscr{A}lg_k^\Delta$, induced  by the functor $R$ and from its lax-monoidal structure $\ell_R$, on the categories of monoids (see
Section \ref{section21}). It admits a left adjoint functor $L^\mon$, defined in Section \ref{section22}.

Recall that for a simplicial algebra $A$ we denote by $\mathfrak{F}(A)$ a solution of Theorem \ref{theor_simpl_alg}.

Define
\begin{equation}
\mathfrak{R}(A)=L^\mon(\mathfrak{F}(R^\mon(A)))
\end{equation}
where $A\in\mathscr{A}lg_k^{\le 0}$.

There is a projection $p_\mathfrak{F}\colon \mathfrak{F}(R^\mon(A))\to R^\mon(A)$. Define the projection
\begin{equation}
p_A\colon \mathfrak{R}(A)\to A
\end{equation}
as the composition of the projection $p_\mathfrak{F}$ with the adjunction map $L^\mon\circ R^\mon\to\Id$.

We claim that this functor $\mathfrak{R}$ gives a solution to Theorem \ref{theor_gr_alg}. We need to prove the following statements:
\begin{prop}\label{propref}{\itshape
\begin{itemize}
\item[(i)] the functor $\mathfrak{R}\colon \mathscr{A}lg_k^{\le 0}\to \mathscr{A}lg_k^{\le 0}$ has a natural colax-monoidal structure $\beta$,
\item[(ii)] $\mathfrak{R}(A)$ is cofibrant, and the projection $p_A\colon\mathfrak{R}(A)\to A$ is a weak equivalence, for any $A\in \mathscr{A}lg_k^{\le 0}$,
\item[(iii)] the diagram
\begin{equation}
 \xymatrix{
    \mathfrak{R}(A\otimes B)\ar[rr]^{\beta_{A,B}}\ar[rd]_{p_{A\otimes B}}&&\mathfrak{R}(A)\otimes \mathfrak{R}(B)\ar[dl]^{p_A\otimes p_B}\\
    &A\otimes B
}
\end{equation}
commutes; consequently, it follows from (ii) and from the 2-out-of-3 axiom that $\beta_{A,B}$ is a weak equivalence for any $A,B\in\mathscr{A}lg_k^{\le 0}$.
\end{itemize}
}
\end{prop}

The three items of this Proposition rely on three different theories: they are the bialgebra axiom for (i), the Schwede-Shipley theory of weak monoidal Quillen pairs for (ii), and the monoidal property of the Dold-Kan correspondence (Lemma \ref{dkmonoidal}) for (iii).
We give the detailed proof in the rest of this Section.

\subsection{\sc Proof of Proposition \ref{propref}, (i)}
The functor $\mathfrak{R}=L^\mon\circ \mathfrak{F}\circ R^\mon$ is a composition of three functors. The functor $L^\mon$ comes with its colax-monoidal structure (adjoint to the lax-monoidal structure on $R^\mon$), and the functor $\mathfrak{F}$ has the colax-monoidal structure \eqref{cf}. It remains to define a colax-monoidal structure on $R^\mon$ (a priori $R^\mon$ has only a lax-monoidal structure).

Recall that the functor $L=N$ has the Alexander-Whitney colax-monoidal structure $AW$ and a lax-monoidal structure $\nabla$, {\it compatible by the bialgebra axiom} (see Theorem \ref{mytheor}). Then it follows from Lemma \ref{lemma12} that the adjoint lax-monoidal and colax-monoidal structures on $R$ obey the bialgebra axiom as well. In general, the lax-monoidal structure on $R$ induces a lax-monoidal structure on $R^\mon\colon \Mon\mathscr{D}\to\Mon\mathscr{C}$ (the same as the lax-monoidal structure on $R$ for the underlying objects), {\it but that is not true for the colax-monoidal structure}. When the both structures are compatible by the bialgebra axiom, Lemma \ref{lemmaxx} says the colax-monoidal structure, defined on the underlying objects as the one on $R$, defines a colax-monoidal structure on $R^\mon$.

Thus, $\mathfrak{R}$ is a composition of three functors, each of which comes with natural colax-monoidal structure. Therefore, $\mathfrak{R}$ is colax-monoidal. We always assume this structure when refer to a colax-monoidal structure on $\mathfrak{R}$.
\qed

\subsection{\sc Quillen pairs and weak monoidal Quillen pairs}
\subsubsection{\sc Quillen pairs and Quillen equivalences}
To prove the statement (ii) of Proposition, we need firstly to recall some definitions on Quileen pairs of functors between two closed model categories, and to recall some results of Schwede-Schipley [SchS03] on weak monoidal Quillen pairs.

In classical homological algebra one can derive left exact or right exact functor. When we work with closed model categories, we try to extend the classical homological algebra to non-abelian (and non-additive) context. A typical examples are the category of topological spaces and the category of dg associative algebras. How we can define the notions of a left (right) exact functor (i.e., of those functors we can derive) in such generality? The answer is given (by Quillen) in the concept of {\it a Quillen pair of functors}.

To motivate the definition below, recall the following simple fact:
\begin{lemma}{\itshape
Suppose $\mathscr{A},\mathscr{B}$ are two abelian categories, and let $$L\colon\mathscr{A}\rightleftarrows\mathscr{B}\colon R$$ be a pair of adjoint functors, with $L$ left adjoint. Then $L$ is right exact and $R$ is left exact.
}
\end{lemma}
Prove as an exercise, or see the proof in [W], Theorem 2.6.1.
\qed

Morally, we can not say in non-abelian categories what is the right (left) exactness, but we know what adjoint functors are. These are (among other assumptions) the functors we can derive. Therefore, they come in pairs.

\begin{defn}{\rm
Let $\mathscr{C},\mathscr{D}$ be two closed model categories, and let
$$
L\colon\mathscr{C}\rightleftarrows\mathscr{D}\colon R
$$
is a pair of adjoint functors, with $L$ the left adjoint. The pair $(L,R)$ is called {\it a Quillen pair of functors} if
\begin{itemize}
\item[(1)] $L$ preserves cofibrations and trivial cofibrations,
\item[(2)] $R$ preserves fibrations and trivial fibrations.
\end{itemize}
}
\end{defn}
It is proven (see, e.g., [Hir], Prop. 8.5.7) that, under these conditions, $L$ takes weak equivalences between cofibrant objects to weak equivalences, and $R$ takes weak equivalences between fibrant object to weak equivalences.

It is proven that a Quillen pair of functors defines an adjoint pair of functors between the homotopy categories,
$$
\mathbb{L}\colon \Ho\mathscr{C}\rightleftarrows \Ho\mathscr{D}\colon\mathbb{R}
$$

The next step is to find conditions on $(L,R)$ under which the pair $(\mathbb{L},\mathbb{R})$ is an adjoint equivalence. This is the case when $(L,R)$ is {\it a Quillen equivalence}.

\begin{defn}\label{qe}{\rm
A Quillen pair
$$
L\colon\mathscr{C}\rightleftarrows\mathscr{D}\colon R
$$
is called {\it a Quiilen equivalence} if for any cofibrant object $X$ in $\mathscr{C}$ and for every fibrant object $Y$ in $\mathscr{D}$, a map
$f\colon X\to RY$ is a weak equivalence if and only if the adjoint map $f^\sharp\colon LX\to Y$ is a weak equivalence.
}
\end{defn}
It is proven (see e.g. [Hir], Theorem 8.5.23) that if $(L,R)$ is a Quillen equivalence, the functors
$$
\mathbb{L}\colon\Ho\mathscr{C}\rightleftarrows\Ho\mathscr{D}\colon \mathbb{R}
$$
form an adjoint {\it equivalence} of categories.

All these results are due to D.Quillen [Q].

\subsubsection{\sc Weak monoidal Quillen pairs}
Here we recall a result on weak monoidal Quillen pairs which is essential for our proof of Proposition \ref{propref}, (ii) below.
Our intention here is not to give a throughout treatment (as it would be just a copy of published papers), but rather to recall very briefly the definitions and results, for convenient reference in the next Subsection.

The categories we consider here are at once closed model and monoidal. There is some reasonable compatibility between these two structures on a category $\mathscr{C}$, which guarantee, in particular, that $\Ho\mathscr{C}$ is a monoidal category. The concept is called {\it a monoidal model category}. We do not reproduce this definition here as we do not use it practically, all our categories in this paper fulfill this definition. The interested reader is referred to [SchS00].

The following definition is due to Schwede-Shipley [SchS03] (Definition 3.6).
\begin{defn}{\rm
Let $\mathscr{C},\mathscr{D}$ be {\it monoidal model categories}, and let
$$
L\colon\mathscr{C}\rightleftarrows\mathscr{D}\colon R
$$
be a Quillen pair of the underlying closed model categories. Suppose there is a lax-monoidal structure $\ell$ on the functor $R$, denote by $\varphi$ the adjoint colax-monoidal structure on $L$.

The triple $(L,R,\ell)$ is called {\it a weak monoidal Quillen pair} if
\begin{itemize}
\item[(i)] for all cofibrant objects $X,Y$ in $\mathscr{C}$, the colax-monoidal map
$$
\varphi_{X,Y}\colon L(X\otimes Y)\to L(X)\otimes L(Y)
$$
is a weak equivalence,
\item[(ii)] for some cofibrant replacement $q\colon \mathbb{I}^c\to\mathbb{I}$ of the unit object in $\mathscr{C}$, the composition
$$
L(\mathbb{I}^c_\mathscr{C})\xrightarrow{L(q)}L(\mathbb{I}_\mathscr{C})\xrightarrow{\mu}\mathbb{I}_\mathscr{D}
$$
is a weak equivalence (where $\mu$ is a part of colax-monoidal structure, see Definition \ref{colax_intro}).
\end{itemize}

A triple $(L,R,\ell)$ is called {\it a weak monoidal Quillen equivalence}, if is a weak monoidal Quillen pair, such that the underlying Quillen pair $(L,R)$ is a Quillen equivalence.
}
\end{defn}

We use essentially the following result from [SchS03] (Theorem 3.12 (3)).

\begin{theorem}\label{ss}{\itshape
Let $(L,R,\ell)$ be a weak monoidal Quillen equivalence, and let
$$
L\colon \mathscr{C}\rightleftarrows \mathscr{D}\colon R
$$
be the underlying Quillen pair. Suppose that the unit objects in $\mathscr{C}$ and $\mathscr{D}$ are cofibrant, and suppose that the forgetful functors $\Mon\mathscr{C}\to\mathscr{C}$ and $\Mon\mathscr{D}\to\mathscr{D}$ create model structures in $\Mon\mathscr{C}$ and $\Mon\mathscr{D}$ (see the explanation just below). Then
\begin{equation}\label{best}
L^\mon\colon\Mon\mathscr{C}\rightleftarrows\Mon\mathscr{D}\colon R^\mon
\end{equation}
is a Quillen equivalence.
}
\end{theorem}
\begin{remark}{\rm
1. See [GS], Section 3, or [Hir], Chapter 11, for detailed explanation of the meaning of ``the forgetful functors generate closed model structures''. This concept refers to the transfer of closed model structures for {\it cofibrantly generated model categories}. See loc.cit. for all these concepts, as well as for a proof that our categories $\mathscr{C}=\mathscr{M}od(\mathbb{Z})^\Delta$ and $\mathscr{D}=\mathscr{C}(\mathbb{Z})^-$ satisfy this assumptions.

2. The Quillen equivalence \eqref{best} is {\it not} a weak monoidal Quillen equivalence. In fact, the natural monoidal structure on $\Mon\mathscr{M}$ for a monoidal model category $\mathscr{M}$, is {\it not} a monoidal model category in general. For instance, the monoidal bifunctor does not commute with the coproducts as a functor of one argument, for fixed another one.
}
\end{remark}

\subsection{\sc Proof of Proposition \ref{propref}, (ii)}
We need to prove that, for any dg algebra $A\in\mathscr{A}lg_k^{\le 0}$, the dg algebra $\mathfrak{R}(A)$ is cofibrant, and the projection
\begin{equation}\label{claimii}
p_A\colon\mathfrak{R}(A)\to A
\end{equation}
is a quasi-isomorphism of dg algebras.

Consider the Dold-Kan correspondence. In [SchS03], Section 3.4, there is given a criterium for when a triple $(L,R,\ell)$ is a weak Quillen equivalence.
It is proven as well, that this criterium works in the following two cases. The first case is the case of the Dold-Kan correspondence, with $\Gamma=R$ the right adjoint, with the lax-monoidal structure being the adjoint to the colax-monoidal structure $AW$ on $N=L$. The second case is the case when $N=R$ is the right adjoint, with the lax-monoidal shuffle structure $\nabla$.
In our applications, we need only the first possibility. Let us summarize.

\begin{lemma}\label{ss0}{\itshape
Consider the Dold-Kan correspondence
$$N\colon\mathscr{M}od(Z)^\Delta\rightleftarrows \mathscr{C}(\mathbb{Z})^-\colon\Gamma$$
Use the notations $L=N$, $R=\Gamma$, and let $\ell$ be the lax-monoidal structure on $R$, adjoint to the colax-monoidal structure $AW$ on $L$.
Then $(L,R,\ell)$ is a weak Quillen equivalence.
}
\end{lemma}
\qed

Now we apply Theorem \ref{ss}.
\begin{coroll}\label{cabove}{\itshape
In the above notations, the adjoint pair of functors
\begin{equation}
L^\mon\colon \Mon(\mathscr{M}od(\mathbb{Z})^\Delta)\rightleftarrows \Mon(\mathscr{C}(\mathbb{Z})^-)\colon R^\mon
\end{equation}
is a weak Quillen equivalence.
}
\end{coroll}
\proof
It follows immediately from Lemma \ref{ss0} and from Theorem \ref{ss}.
\qed

Now we pass to a proof of the claims of Proposition \ref{propref}, (ii).
\proof
Prove that $\mathfrak{R}(A)$ is cofibrant dg algebra for any $A\in\mathscr{A}lg_k^{\le 0}$.

Indeed,
$$
\mathfrak{R}(A)=L^\mon\circ\mathfrak{F}\circ R^\mon(A)
$$
We know that $\mathfrak{F}(-)$ is cofibrant (Lemma \ref{lemmacof}), and that $(L^\mon,R^\mon)$ is a Quillen equivalence, in particular, a Quillen pair (Corollary \ref{cabove}). By Definition of a Quillen pair, $L^\mon$ maps cofibrations to cofibrations. As $L^\mon$ maps the initial object to the initial object, it maps the cofibrant objections to cofibrant objects. Therefore, $\mathfrak{R}(A)$ is cofibrant, as $\mathfrak{F}(R^\mon(A))$ is cofibrant.

Prove that, for any $A\in\mathscr{A}lg_k^{\le 0}$, the projection $p_A\colon\mathfrak{R}(A)\to A$ is a weak equivalence.

The projection
\begin{equation}
p_A\colon L^\mon\circ \mathfrak{F}\circ R^\mon (A)\to A
\end{equation}
is adjoint to
\begin{equation}
p_A^\sharp\colon \mathfrak{F}\circ R^\mon (A)\to R^\mon (A)
\end{equation}
According to Corollary \ref{cabove}, the pair $(L^\mon, R^\mon)$ is a Quillen equivalence. We are going to apply the defining property of Quillen equivalences, see Definition \ref{qe}.
Moreover, $\mathfrak{F}\circ R^\mon (A)$ is cofibrant by Lemma \ref{lemmacof} (as $\mathfrak{F}(-)$ is cofibrant by this Lemma), and $R^\mon(A)$ is fibrant by the description in Section \ref{section12}. Moreover, the map $p_A^\sharp$ is a weak equivalence, again by Lemma \ref{lemmacof}. Therefore, $p_A$ is a weak equivalence as well, by Definition \ref{qe}.

\qed

\subsection{\sc Proof of Proposition \ref{propref}, (iii)}\label{finallink}
The claim of Proposition \ref{propref}, (iii) follows from the two Lemmas below. The first one is straightforward, while the second one heavily relies on results of Section 4.
\begin{lemma}\label{liii1}{\itshape
For any $X,Y$, the following diagram is commutative:
\begin{equation}
\xymatrix{
L^\mon(\mathfrak{F}(R^\mon(X\otimes Y)))\ar[r]\ar[d]& L^\mon(\mathfrak{F}(R^\mon(X)))\otimes L^\mon(\mathfrak{F}(R^\mon(Y)))\ar[d]\\
L^\mon R^\mon(X\otimes Y)\ar[r]&L^\mon R^\mon(X)\otimes L^\mon R^\mon(Y)
}
\end{equation}
}
\end{lemma}

\begin{lemma}\label{liii2}{\itshape
For any $X,Y$, the following diagram is commutative:
\begin{equation}\label{potom}
\xymatrix{
L^\mon R^\mon(X\otimes Y)\ar[rr]\ar[rd]_{p_{X\otimes Y}}&&L^\mon R^\mon(X)\otimes L^\mon R^\mon(Y)\ar[ld]^{p_X\otimes p_Y}\\
&X\otimes Y
}
\end{equation}
}
\end{lemma}
{\it Proof of Lemma \eqref{liii1}:} We have the commutative diagram:
\begin{equation}\label{eqiii1}
\xymatrix{
L^\mon(\mathfrak{F}(R^\mon(X\otimes Y)))\ar[r]^{c_R}\ar[d]&L^\mon(\mathfrak{F}(R^\mon(X)\otimes R^\mon(Y)))\ar[d]\\
L^\mon(R^\mon(X\otimes Y))\ar[r]^{c_R}&L^\mon(R^\mon(X)\otimes R^\mon(Y))
}
\end{equation}
by tautological reasons. Now consider the diagram
\begin{equation}\label{eqiii2}
\xymatrix{
\mathfrak{F}(R^\mon(X)\otimes R^\mon(Y))\ar[rr]\ar[rd]&&\mathfrak{F}(R^\mon(X))\otimes \mathfrak{F}(R^\mon(Y))\ar[ld]\\
&R^\mon(X)\otimes R^\mon(Y)
}
\end{equation}
from Lemma \ref{web}. The composition of the diagram \eqref{eqiii1} with the application of $L^\mon$ to the diagram \eqref{eqiii2} gives the diagram in Lemma \ref{liii1}.\qed

\smallskip
\smallskip

{\it Proof of Lemma \eqref{liii2}:}
Theorems \ref{another},\ref{keytheorem} describe the functor $L^\mon$ explicitly, with its colax-monoidal structure.
According to these results,
$$
L^\mon(X)\simeq\tilde{L}^\mon(X)=L_\nabla(X)/\mathscr{J}(X)
$$
and the colax-monoidal structure, adjoint to the lax-monoidal structure on $R^\mon$, is induced by the Alexander-Whitney colax-monoidal structure
$$
AW\colon L(X\otimes Y)\to L(X)\otimes L(Y)
$$
by passing to the quotient-monoids,
$$
\wtilde{AW}\colon L_\nabla(X\otimes Y)/\mathscr{J}(X\otimes Y)\to (L_\nabla(X)/\mathscr{J}(X))\otimes (L_\nabla(Y)/\mathscr{J}(Y))
$$
\comment
Recall that the functor $L^\mon\colon \Mon\mathscr{C}\to \Mon\mathscr{D}$ is constructed as the left adjoint to the functor $\Mon\mathscr{C}\leftarrow\Mon\mathscr{D}\colon R^\mon$. The functor $R^\mon$ coincides with the functor $R$ on the underlying objects, and the monoid structure is defined from a lax-monoidal structure $\ell_R$ on $R$ (see \eqref{monoidref}). Then $\ell_R$ defines a lax-monoidal structure on $R^\mon$ as well, and the colax-monoidal structure $c_{L^\mon}$ is defined from $\ell_R$ by the adjunction (see \eqref{cg}).

Here we give a more explicit description of this colax-monoidal structure on $L^\mon$.

Recall (see [SchS03], Section 3.3) that the value $L^\mon(X)$ (for a monoid $X$ in $\mathscr{C}$) is defined as the co-equalizer
\begin{equation}\label{quotient}
\xymatrix{
T_{\mathscr{D}}(L(T_{\mathscr{C}}(X)))\ar@<1ex>[r]^{\ \ \ \alpha} \ar@<-1ex>[r]_{\ \ \ \beta} &T_{\mathscr{D}}(LX)
}
\end{equation}
where $T_\mathscr{M}$ denotes the free (tensor) monoid in a monoidal category $\mathscr{M}$, see \eqref{MLM}. In particular, it is some quotient of the free monoid $T_{\mathscr{D}}(LX)$.

From now on, we restrict ourselves with the case of the Dold-Kan correspondence, $\mathscr{C}=\mathscr{M}od(\mathbb{Z})^\Delta$, $\mathscr{D}=\mathscr{C}(\mathbb{Z})^-$, $L=N$, $R=\Gamma$.

Consider the shuffle lax-monoidal structure $\nabla$ on $L$. This structure defines a functor $\tilde{L}^\mon\colon \Mon\mathscr{C}\to \Mon\mathscr{D}$, as in \eqref{monoidref}. We claim that the functors $L^\mon$ and $\tilde{L}^\mon$ coincide.
\begin{klemma}\label{kl}{\itshape
For the Dold-Kan correspondence and the above notations, the two functors $$L^\mon,\tilde{L}^\mon\colon \Mon\mathscr{C}\to\Mon\mathscr{D}$$ coincide.
Moreover, the colax-monoidal structure $c_{L^\mon}$ on $L^\mon$ (a priori defined by adjunction from the lax-monoidal structure on $R^\mon$) coincides on the underlying objects  with the Alexander-Whitney colax-monoidal structure $AW$ on $L$ (the latter defines a colax-monoidal structure on monoids
as the pair $(\nabla,AW)$ obeys the bialgebra axiom, see Lemma \ref{lemmaxx}).
}
\end{klemma}
\proof
Consider the definition \eqref{quotient} of the functor $L^\mon$, in the special case of the Dold-Kan correspondence. It gives that $L^\mon(X)$ is the quotient of the free dg algebra $T(LX)$ by the {\it two-sided ideal}, generated by
\begin{equation}\label{ideal}
AW(L(x\otimes y))-L(x\star y),\ \  x,y\in X
\end{equation}
where $\star$ is the product in $X$.

Now the crucial idea, special for the case of Dold-Kan correspondence, is that {\it the subspace in $L(X)\otimes L(X)$ generated by the elements $AW(L(x\otimes  y))$ coincides with the entire space $L(X)\otimes L(X)$}.

Indeed, we know from Proposition \ref{before}, (1.), that the composition
$$
L(X)\otimes L(X)\xrightarrow{\nabla}L(X\otimes X)\xrightarrow{AW}L(X)\otimes L(X)
$$
is the identity map of $L(X)\otimes L(X)$. Then, if we have any element $\omega\in L(X)\otimes L(X)$, we can write it as
$$
\omega=AW(\nabla(\omega))
$$

This claim simplifies the ideal \eqref{ideal}. Namely, we have now that this ideal is generated by
\begin{equation}
L(a)\otimes L(b)-m_\star(\nabla(L(a)\otimes L(b)))
\end{equation}
where the notation $m_\star$ in the second summand assumes the composition $$L(X)\otimes L(X)\xrightarrow{\nabla}L(X\otimes X)\xrightarrow{L\circ\star}L(X)$$
The latter formula is precisely the equation \eqref{monoidref}, which defines the monoid structure on $F(X)$ from a lax-monoidal structure on $F$ and from a monoid structure on $X$.

All other claims of Key-Lemma follow from this argument as well.
\qed
\endcomment
The diagram \eqref{potom} may be now rewritten as
\begin{equation}\label{potom2}
\xymatrix{
\tilde{L}^\mon (R(X\otimes Y))\ar[rr]\ar[dr]_{p_{X\otimes Y}}&&\tilde{L}^\mon(R(X))\otimes \tilde{L}^\mon (R(Y))\ar[dl]^{p_X\otimes p_Y}\\
&X\otimes Y
}
\end{equation}
where the horizontal map comes from the colax-monoidal structure on $R$ (adjoint to the lax-monoidal structure $\nabla$ on $L$), followed by the colax-monoidal structure $\wtilde{AW}$ on $\tilde{L}^\mon$. The projections comes from the adjunction maps.

Now the commutativity of diagram \eqref{potom2} follows immediately from the statement of Lemma \ref{dkmonoidal} (ii), by passing to quotient-monoids. We are done.

\qed

Theorem \ref{theor_gr_alg} is proven.
\qed

\appendix

\appsection{\sc Diagrams}\label{diagrams}

\appsubsection{\sc Colax-monoidal structure on a functor}\label{colax}
\begin{defn}[{\rm Colax-monoidal functor}]\label{colax_intro}{\rm
Let $\mathscr{M}_1$ and $\mathscr{M}_2$ be two strict associative monoidal categories. A functor $F\colon\mathscr{M}_1\to \mathscr{M}_2$ is called {\it colax-monoidal} if there is a map of bifunctors $\beta_{X,Y}\colon F(X\otimes Y)\to F(X)\otimes F(Y)$ and a morphism $\alpha\colon F(1_{\mathscr{M}_1})\to 1_{\mathscr{M}_2}$ such that:

(1): for any three objects $X,Y,Z\in\Ob(\mathscr{M}_1)$, the diagram
\begin{equation}\label{colaxdiagram}
\xymatrix{
&F(X\otimes Y)\otimes F(Z)\ar[rd]^{\beta_{X,Y}\otimes\id_{F(Z)}}\\
F(X\otimes Y\otimes Z)\ar[ru]^{\beta_{X\otimes Y,Z}}\ar[rd]_{\ \ \ \ \ \ \ \beta_{X,Y\otimes Z}}&& F(X)\otimes F(Y)\otimes F(Z)\\
&F(X)\otimes F(Y\otimes Z)\ar[ru]_{\id_{F(X)}\otimes\beta_{Y,Z}}
}
\end{equation}
is commutative. The functors $\beta_{X,Y}$ are called the {\it colax-monoidal maps}.

(2): for any $X\in\Ob\mathscr{M}_1$ the following two diagrams are commutative
\begin{equation}
\begin{aligned}
\ &\xymatrix{
F(1_{\mathscr{M}_1}\otimes X)\ar[d]\ar[r]^{\beta_{1,X}}& F(1_{\mathscr{M}_1})\otimes F(X)\ar[d]^{\alpha\otimes\id}\\
F(X)&      1_{\mathscr{M}_2}\otimes F(X)\ar[l]
} &
\xymatrix{
F(X\otimes 1_{\mathscr{M}_1})\ar[d]\ar[r]^{\beta_{X,1}}& F(X)\otimes F(1_{\mathscr{M}_1})\ar[d]^{\id\otimes \alpha}\\
F(X)&     F(X)\otimes 1_{\mathscr{M}_2}\ar[l]
}
\end{aligned}
\end{equation}
}
\end{defn}

\appsubsection{\sc Lax-monoidal structure on a functor}\label{lax}
\begin{defn}[{\rm Lax-monoidal functor}]\label{lax_intro}{\rm
Let $\mathscr{M}_1$ and $\mathscr{M}_2$ be two strict associative monoidal categories. A functor $G\colon\mathscr{M}_1\to \mathscr{M}_2$ is called {\it lax-monoidal} if there is a map of bifunctors $\gamma_{X,Y}\colon G(X)\otimes G(Y)\to G(X\otimes Y)$ and a morphism $\kappa\colon 1_{\mathscr{M}_2}\to G(1_{\mathscr{M}_1})$ such that:

(1): for any three objects $X,Y,Z\in\Ob(\mathscr{M}_1)$, the diagram
\begin{equation}\label{colaxdiagram}
\xymatrix{
&G(X\otimes Y)\otimes G(Z)\ar[dl]_{\gamma_{X\otimes Y,Z}}\\
G(X\otimes Y\otimes Z)&& G(X)\otimes G(Y)\otimes G(Z)\ar[ul]_{\gamma_{X, Y}\otimes\id_{G(Z)}}\ar[dl]^{\id_{G(X)}\otimes\gamma_{Y,Z}}\\
&G(X)\otimes G(Y\otimes Z)\ar[ul]^{\gamma_{X,Y\otimes Z}}
}
\end{equation}
is commutative. The functors $\gamma_{X,Y}$ are called the {\it lax-monoidal maps}.

(2): for any $X\in\Ob\mathscr{M}_1$ the following two diagrams are commutative
\begin{equation}
\begin{aligned}
\ &\xymatrix{
F(1_{\mathscr{M}_1}\otimes X)\ar[d]& F(1_{\mathscr{M}_1})\otimes F(X)\ar[l]_{\gamma_{1,X}}\\
F(X)&      1_{\mathscr{M}_2}\otimes F(X)\ar[u]_{\kappa\otimes\id}\ar[l]
} &
\xymatrix{
F(X\otimes 1_{\mathscr{M}_1})\ar[d]& F(X)\otimes F(1_{\mathscr{M}_1})\ar[l]_{\gamma_{X,1}}\\
F(X)&     F(X)\otimes 1_{\mathscr{M}_2}\ar[u]_{\id\otimes \kappa}\ar[l]
}
\end{aligned}
\end{equation}
}
\end{defn}

\appsubsection{\sc Bialgebra axiom}\label{bialg}
This axiom, expressing a compatibility between the lax-monoidal and colax-monoidal structures on a functor between  {\it symmetric} monoidal categories, seems to be new.
\begin{defn}[Bialgebra axiom]{\rm
Suppose there are given both colax-monoidal and lax-monoidal structures on a functor $F\colon \mathscr{C}\to\mathscr{D}$, where $\mathscr{C}$ and $\mathscr{D}$ are strict {\it symmetric} monoidal categories.
Denote these structures by $c_F(X,Y)\colon F(X\otimes Y)\to F(X)\otimes F(Y)$, and $l_F\colon F(X)\otimes F(Y)\to F(X\otimes Y)$.
We say that the pair $(l_F,c_F)$ satisfies the bialgebra axiom, if for any for objects $X,Y,Z,W\in\Ob\mathscr{C}$, the following two morphisms
$F(X\otimes Y)\otimes F(Z\otimes W)\to F(X\otimes Z)\otimes F(Y\otimes W)$
coincide:
\begin{equation}\label{bialgebra1}
\begin{aligned}
\ &F(X\otimes Y)\otimes F(Z\otimes W)\xrightarrow{l_F}F(X\otimes Y\otimes Z\otimes W)\xrightarrow{F(\id\otimes \sigma\otimes \id)}\\
&F(X\otimes Z\otimes Y\otimes W)\xrightarrow{c_F}F(X\otimes Z)\otimes F(Y\otimes W)
\end{aligned}
\end{equation}
and
\begin{equation}\label{bialgebra2}
\begin{aligned}
\ &F(X\otimes Y)\otimes F(Z\otimes W)\xrightarrow{c_F\otimes c_F}F(X)\otimes F(Y)\otimes F(Z)\otimes F(W)\xrightarrow{\id\otimes \sigma \otimes \id}\\
& F(X)\otimes F(Z)\otimes F(Y)\otimes F(W)\xrightarrow{l_F\otimes l_F}F(X\otimes Z)\otimes F(Y\otimes W)
\end{aligned}
\end{equation}
where $\sigma$ denotes the symmetry morphisms in $\mathscr{C}$ and in $\mathscr{D}$.

Thus, the commutative diagram, expressing the bialgebra axiom, is
\begin{equation}
\xymatrix{
F(X\otimes Y)\otimes F(Z\otimes W)\ar@/^2pc/[rr]^{\eqref{bialgebra1}}\ar@/_2pc/[rr]_{\eqref{bialgebra2}}&&F(X\otimes Z)\otimes F(Y\otimes W)
}
\end{equation}
}
\end{defn}

\appsection{\sc Weak adjoint functors}\label{appendixb}
Here we develop some techniques used in Section 4 of the paper. The main topic here is roughly the following. Recall from \eqref{lg}, \eqref{cg} that a pair of adjoint functors $L\colon\mathscr{C}\rightleftarrows\mathscr{D}\colon R$ provides $1\leftrightarrow 1$ correspondence between the colax-monoidal structures on $L$ and the lax-monoidal structures on $R$. In Section 4 we need to know whether and to which extend this duality remains true, when we work with {\it weak adjoint functors}, see Definition \ref{adef1} below. Then we study the situation (also came up in Section 4) when $L$ and $R$ are honest adjoint functors but they are represented as (co)limits of pairs of functors $L_i\colon\mathscr{C}\rightleftarrows\mathscr{D}\colon R_i$, for $i$ running trough some category $I$, such that for each $i$ the functors $(L_i,R_i)$ are weak adjoint. The main results of this Appendix are Theorems \ref{keytheorem2bis} and \ref{keytheorem2bisbis}.
\appsubsection{\sc }\label{sectionb1}
Suppose
$$
L\colon\mathscr{C}\rightleftarrows\mathscr{D}\colon R
$$
is a pair of adjoint functors, with $L$ the left adjoint.
Recall that this means that for any $X$ in $\mathscr{C}$ and $Y$ in $\mathscr{D}$, there is an isomorphism of sets
\begin{equation}
\Hom_\mathscr{D}(LX,Y)\simeq \Hom_\mathscr{C}(X, RY)
\end{equation}
such that, for all $X$ and $Y$, the two corresponding bifunctors $\mathscr{C}^\circ\times\mathscr{D}\to\mathscr{S}ets$ are isomorphic.

This gives rise to a morphisms of functors $\epsilon LR\to\Id_{\mathscr{D}}$ and $\eta\colon \Id_{\mathscr{D}}\to RL$, such that the compositions
\begin{equation}\label{a1}
L\xrightarrow{\eta_*}LRL\xrightarrow{\epsilon_*}L
\end{equation}
and
\begin{equation}\label{a2}
R\xrightarrow{\eta_*}RLR\xrightarrow{\epsilon_*}R
\end{equation}
are the identity morphisms of the functor $L$ (correspondingly, of $R$), see [ML IV.1, Theorem 1].

It is well-known that the inverse is also true:

If there are two functors $L\colon\mathscr{C}\to\mathscr{D}$ and $R\colon\mathscr{D}\to\mathscr{C}$, with morphisms of functors $\epsilon\colon LR\to\Id_\mathscr{D}$ and $\eta\colon \Id_{\mathscr{C}}\to RL$, such that the compositions \eqref{a1} and \eqref{a2} are identity maps of functors, the functors $L$ and $R$ are adjoint ,see [ML IV.1, Theorem 2].

In this Appendix B, we refer to the classical adjoint functors described above as to {\it genuine} adjoint functors, and consider various definitions of weak adjoint functors. We define {\it weak right adjoint functors} and {\it weak left adjoint functors} in Definition \ref{adef1} just below, and define {\it very weak adjoint functors} in Section \ref{appveryweak}.

We give the following definition:
\begin{defn}\label{adef1}{\rm
Let $\mathscr{C}$ and $\mathscr{D}$ be two categories, and let $L\colon \mathscr{C}\to\mathscr{D}$, $R\colon\mathscr{D}\to\mathscr{C}$ be two functors. Suppose there are given two morphisms of functors $\epsilon\colon LR\to\Id_\mathscr{D}$ and $\eta\colon \Id_{\mathscr{C}}\to RL$,
such that the composition \eqref{a2} is the identity morphism of the functor $R$, but the composition \eqref{a1} may fail to be the identity morphism of $L$.
Then the data $(L,R,\epsilon,\eta)$ is called a {\it weak right adjunction} between $L$ and $R$.
Analogously, when \eqref{a1} is the identity morphism of $L$, but \eqref{a2} may fail to be the identity morphism of $R$, the data $(L,R,\epsilon,\eta)$ is called a {\it weak left adjunction} between $L$ and $R$.
}
\end{defn}
The following Lemma gives an equivalent way to describe weak adjoint pairs. We will use it in the proof of Lemma \ref{lemmaa} below.
\begin{lemma}\label{twoweak}{\itshape
Let
$$
L\colon\mathscr{C}\rightleftarrows\mathscr{D}\colon R
$$
be a pair of functors, and suppose the two following morphisms of bifunctors $\mathscr{C}^\opp\times\mathscr{D}\to\mathscr{S}ets$
\begin{equation}\label{podrugomu}
\begin{aligned}
\ &\Theta\colon\qquad \Hom_\mathscr{D}(LX,Y)\to \Hom_\mathscr{C}(X,RY)\\
&\Upsilon\colon\qquad \Hom_\mathscr{C}(X,RY)\to\Hom_\mathscr{D}(LX,Y)
\end{aligned}
\end{equation}
are given. They produce the maps $\epsilon\colon LR\to\Id_\mathscr{D}$ and $\eta\colon\Id_\mathscr{C}\to RL$ in a standard way, and vice versa. Then:
\begin{itemize}
\item[(i)] the two diagrams
\begin{equation}\label{dikobraz}
\xymatrix{\Hom_\mathscr{C}(X,Y)\ar[d]_{L_*}\ar[r]^{\eta\circ -}&\Hom_\mathscr{C}(X,RLY)\\
\Hom_\mathscr{D}(LX,LY)\ar[ur]_{\Theta}\\
}
\xymatrix{
\Hom_\mathscr{D}(Z,W)\ar[r]^{-\circ\epsilon}\ar[d]_{R_*}&\Hom_\mathscr{D}(LRZ,W)\\
\Hom_\mathscr{C}(RZ,RW)\ar[ur]_{\Upsilon}\\
}
\end{equation}
commute,
\comment
\item[(ii)] the two diagrams
\begin{equation}\label{dikobrazbis}
\xymatrix{
\Hom_\mathscr{C}(RLX,RLY)\ar[r]^{-\circ\eta}&\Hom_\mathscr{C}(X,RLY)\\
\Hom_\mathscr{D}(LX,LY)\ar[u]^{R_*}\ar[ur]_{\Theta}
}
\xymatrix{
\Hom_\mathscr{D}(LRZ,LRW)\ar[r]^{\epsilon\circ -}&\Hom_\mathscr{D}(LRZ,W)\\
\Hom_\mathscr{C}(RZ,RW)\ar[u]^{L_*}\ar[ur]_{\Upsilon}
}
\end{equation}
commute,
\endcomment
\item[(ii)] the pair $(L,R)$ is weak right adjoint iff the composition
\begin{equation}\label{tylen1}
\Hom_\mathscr{C}(X,RY)\xrightarrow{\Upsilon}\Hom_\mathscr{D}(LX,Y)\xrightarrow{\Theta}\Hom_\mathscr{C}(X,RY)
\end{equation}
is the identity map;
\item[(iii)] the pair $(L,R)$ is weak left adjoint iff the composition
\begin{equation}\label{tylen2}
\Hom_\mathscr{D}(LX,Y)\xrightarrow{\Theta}\Hom_\mathscr{C}(X,RY)\xrightarrow{\Upsilon}\Hom_\mathscr{D}(LX,Y)
\end{equation}
is the identity map.
\end{itemize}
}
\end{lemma}

{\it Proof of (i).}
Let us prove the commutativity of the left-hand-side diagram in \eqref{dikobraz}.
Let $X,Y\in \mathscr{C}$, and let $f\in\Hom_\mathscr{C}(X,Y)$.

Consider the diagram
\begin{equation}
\xymatrix{
\Hom_\mathscr{C}(X,X)\ar[r]^{L_*}\ar[d]_{f\circ -}&\Hom_\mathscr{D}(LX,LX)\ar[r]^{\Theta}\ar[d]_{L(f)\circ -}&\Hom_\mathscr{C}(X,RLX)\ar[d]^{RL(f)\circ -}\\
\Hom_\mathscr{C}(X,Y)\ar[r]^{L_*}&\Hom_\mathscr{D}(LX,LY)\ar[r]^{\Theta}&\Hom_\mathscr{C}(X,RLY)
}
\end{equation}
We compute both paths on the identity morphism $\id_X\in\Hom_\mathscr{C}(X,X)$, which gives the left-hand-side diagram in \eqref{dikobraz}.
The proof of commutativity of the right-hand-side diagram is analogous.

{\it Proof of (ii).}

Prove implication $\eqref{tylen1}\Rightarrow\eqref{a2}$.

Consider the map $\Theta\colon \Hom_\mathscr{D}(LX,Y)\to \Hom_\mathscr{C}(X,RY)$. Let $f\in\Hom_\mathscr{D}(LX,Y)$. We have the following diagram:
\begin{equation}
\xymatrix{
&\Hom_\mathscr{D}(LX,LX)\ar[d]_{f_*}\ar[r]^{\Theta}& \Hom_\mathscr{C}(X,RLX)\ar[d]^{f_*}\\
\Hom_\mathscr{C}(X,RY)\ar[r]^{\Upsilon}&\Hom_\mathscr{D}(LX,Y)\ar[r]^{\Theta}&\Hom_\mathscr{C}(X,RY)
}
\end{equation}
The square commutes as $\Theta$ is a map of bifunctors, and the composition of the two maps in the bottom line is identity, by the assumption.

Now substitute $X=RY$, and let $f\colon LRY\to Y$ be the map $\epsilon$.
We have
\begin{equation}
\xymatrix{
&\Hom_\mathscr{D}(LRY,LRY)\ar[d]^{\epsilon_*}\ar[r]^{\Theta}&\Hom_\mathscr{C}(RY,RLR(Y))\ar[d]^{\mathfrak{X}}\\
\Hom_\mathscr{C}(RY,RY)\ar[ur]^{L_*}\ar[r]^{\Upsilon}&\Hom_\mathscr{D}(LRY,Y)\ar[r]^{\Theta}&\Hom_\mathscr{C}(RY,RY)
}
\end{equation}
We know that the right-hand-side square sub-diagram commutes.
We start with the identity morphism $\id_{RY}\in\Hom_\mathscr{C}(RY,RY)$. Then the left-hand side triangle commutes (we need to know this claim only with the identity of $RY$ in the left corner, in this case it is tautological).
Thus, the overall diagram commutes. Then the map $\mathfrak{X}$ is equal to the identity of $RY$, if we start with the identity of $RY$, as the composition $\Theta\circ \Upsilon=\Id$ by assumption.

Now define $\Theta$ and $\Upsilon$, out of $\eta$ and $\epsilon$, by
\begin{equation}\label{awaytodefine}
\xymatrix{
\Hom_\mathscr{C}(X,RY)\ar[r]^{\Upsilon}\ar[d]_{L_*}&\Hom_\mathscr{D}(LX,Y)\\
\Hom_\mathscr{D}(LX,LRY)\ar[ur]_{\epsilon\circ -}
}
\xymatrix{
\Hom_\mathscr{D}(LX,Y)\ar[r]^{\Theta}\ar[d]_{R_*}&\Hom_\mathscr{C}(X,RY)\\
\Hom_\mathscr{C}(RLX,RY)\ar[ur]_{-\circ \eta}
}
\end{equation}
For the implication $\eqref{a2}\Rightarrow\eqref{tylen1}$, consider the diagram
\begin{equation}
\xymatrix{
\Hom_\mathscr{C}(X,RY)\ar[r]^{\Upsilon}\ar[d]_{L_*}&\Hom_\mathscr{D}(LX,Y)\ar[r]^{\Theta}\ar[d]_{R_*}&\Hom_\mathscr{C}(X,RY)\\
\Hom_\mathscr{D}(LX,LRY)\ar[d]_{R_*}\ar[ur]_{\epsilon\circ -}&\Hom_\mathscr{C}(RLX,RY)\ar[ur]_{-\circ \eta}\\
\Hom_\mathscr{C}(RLX,R\boxed{LR}Y)\ar[ur]_{\epsilon(\square)\circ -}
}
\end{equation}
The diagram is commutative. Thus, we need to compute the composition
\begin{equation}\label{lastochka}
\Hom_\mathscr{C}(X,RY)\xrightarrow{R_*L_*}\Hom_\mathscr{C}(RLX,RLRY)\xrightarrow{-\circ\eta}\Hom_\mathscr{C}(X,R\boxed{LR}Y)\xrightarrow{\epsilon\circ -}\Hom_\mathscr{C}(X,RY)
\end{equation}
(we switched $\epsilon$ and $\eta$ which does dot affect the answer).

We have:
\begin{sublemma}{\itshape
Let
$$
L\colon\mathscr{C}\rightleftarrows\mathscr{D}\colon R
$$
be a pair of functors, and let $\Psi\colon \Id_\mathscr{C}\to RL$ is a map of functors.
Let $Z,W\in\mathscr{C}$. Then the composition
\begin{equation}
\Hom_\mathscr{C}(Z,W)\xrightarrow{(RL)_*}\Hom_\mathscr{C}(RLZ,RLW)\xrightarrow{-\circ \Psi(Z)}\Hom_\mathscr{C}(Z,RLW)
\end{equation}
is equal to the composition
\begin{equation}
\Hom_\mathscr{C}(Z,W)\xrightarrow{\Psi(W)\circ -}\Hom_\mathscr{C}(Z,RLW)
\end{equation}
}
\end{sublemma}
For, consider the diagram, for a morphism $f\colon Z\to W$:
\begin{equation}
\xymatrix{
\Hom_\mathscr{C}(Z,Z)\ar[r]^{(RL)_*\quad}\ar[d]_{f_*\circ -}&\Hom_\mathscr{C}(RLZ,RLZ)\ar[r]^{-\circ\Psi(Z)}\ar[d]_{f_*\circ -}&\Hom_\mathscr{C}(Z,RLZ)\ar[d]^{f_*\circ -}\\
\Hom_\mathscr{C}(Z,W)\ar[r]^{(RL)_*\quad}&\Hom_\mathscr{C}(RLZ,RLW)\ar[r]^{-\circ\Psi(Z)}&\Hom_\mathscr{C}(Z,RLW)
}
\end{equation}
(with the identity in the upper-left corner) and
\begin{equation}
\xymatrix{
Z\ar[r]^{\Psi(Z)}\ar[d]_{f}&RLZ\ar[d]^{f_*}\\
W\ar[r]^{\Psi(W)}&RLW
}
\end{equation}
\qed

By the above Sub-lemma, we rewrite composition \eqref{lastochka} as
\begin{equation}\label{lastochkabis}
\Hom_\mathscr{C}(X,RY)\xrightarrow{\eta(RY)\circ -}\Hom_\mathscr{C}(X,R\boxed{LR}Y)\xrightarrow{\epsilon(\square)}\Hom_\mathscr{C}(X,RY)
\end{equation}
Now \eqref{lastochkabis} is the identity map by assumption \eqref{a2}.

{\it Proof of (iii)} is analogous to the proof of (ii), and is left to the reader.
\qed

\bigskip

We turn to the discussion of interplay between the lax- and the colax-monoidal structures, in the weak adjoint setting. We have
\begin{lemma}{\itshape
There are the following statements:
\begin{itemize}
\item[(i)] Let $L\colon\mathscr{C}\rightleftarrows\mathscr{D}\colon R$ be a weak right adjunction, for some $\epsilon$ and $\eta$. Let $\ell$ be a lax-monoidal structure on $R$. Then the composition
\begin{equation}\label{lax2colax}
L(X\otimes Y)\xrightarrow{(\eta\otimes\eta)_*}L(RLX\otimes RLY)\xrightarrow{\ell_*}LR(LX\otimes LY)\xrightarrow{\epsilon_*}LX\otimes LY
\end{equation}
is a colax-monoidal structure on L;
\item[(ii)] Analogously, if $(L,R,\epsilon,\eta)$ is a left weak adjunction, and $c$ is a colax-monoidal structure on $L$, the composition
\begin{equation}\label{colax2lax}
RX\otimes RY\xrightarrow{\eta_*}RL(RX\otimes RY)\xrightarrow{c_*}R(LR(X)\otimes LR(Y))\xrightarrow{(\epsilon\otimes\epsilon)_*}R(X\otimes Y)
\end{equation}
is a lax-monoidal structure on $L$.
\end{itemize}
}
\end{lemma}
\begin{proof}
We prove (i), the proof of (ii) is analogous.

We need to prove the commutativity of the diagram
\begin{equation}\label{zebra0}
\xymatrix{
L(X\otimes Y\otimes Z)\ar[r]^{c_{X\otimes Y,Z}}\ar[d]_{c_{X,Y\otimes Z}}&L(X\otimes Y)\otimes L(Z)\ar[d]^{c_{X,Y}\otimes\id_Z}\\
L(X)\otimes L(Y\otimes Z)\ar[r]_{\id_X\otimes c_{Y,Z}}&L(X)\otimes L(Y)\otimes L(Z)
}
\end{equation}
where $c$ is the colax-monoidal structure on $L$, defined out of the lax-monoidal structure $\ell$ on $R$ by \eqref{lax2colax}.
\begin{sublemma}{\itshape
There are the following statements:
\begin{itemize}
\item[A.] The composition
\begin{equation}\label{zebra1}
L(X\otimes Y\otimes Z)\xrightarrow{c_{X\otimes Y,Z}} L(X\otimes Y)\otimes L(Z)\xrightarrow{c_{X,Y}\otimes\id_Z} L(X)\otimes L(Y)\otimes L(Z)
\end{equation}
is equal to the composition
\begin{equation}\label{zebra2}
\begin{aligned}
\ &L(X\otimes Y\otimes Z)\xrightarrow{{\eta_*}^{\otimes 3}}L(RLX\otimes RLY\otimes RLZ)\xrightarrow{\ell_{LX,LY}\otimes\id_{RLZ}}L(R(LX\otimes L(Y))\otimes RLZ)\\&\xrightarrow{\ell_{LX\otimes LY,LZ}}
LR(LX\otimes LY\otimes LZ)\xrightarrow{\epsilon}LX\otimes LY\otimes LZ
\end{aligned}
\end{equation}
\item[B.] the composition
\begin{equation}\label{zebra3}
L(X\otimes Y\otimes Z)\xrightarrow{c_{X,Y\otimes Z}}L(X)\otimes L(Y\otimes Z)\xrightarrow{\id_{LX}\otimes c_{Y,Z}}L(X)\otimes L(Y)\otimes L(Z)
\end{equation}
is equal to the composition
\begin{equation}\label{zebra4}
\begin{aligned}
\ &L(X\otimes Y\otimes Z)\xrightarrow{\eta_*^{\otimes 3}}L(RLX\otimes RLY\otimes RLZ)\xrightarrow{\id_{RLX}\otimes \ell_{LY,LZ}}
L(RLX\otimes R(LY\otimes LZ))\\
&\xrightarrow{\ell_{LX,LY\otimes LZ}}LR(LX\otimes LY\otimes LZ)\xrightarrow{\epsilon}LX\otimes LY\otimes LZ
\end{aligned}
\end{equation}
\end{itemize}
}
\end{sublemma}
The claim (i) of Lemma follows easily from the Sub-lemma above. Indeed, \eqref{zebra2}=\eqref{zebra4}, as $\ell$ is a lax-monoidal functor, then Sub-lemma implies \eqref{zebra1}=\eqref{zebra3}. The latter is exactly the commutativity of the diagram \eqref{zebra0}.

It remains to prove the Sub-lemma.

{\it Proof of Sub-lemma:} We prove the claim (A.), the proof of claim (B.) is analogous. In the total diagram below the upper-right composition is equal to \eqref{zebra1}, and the lower-left composition is equal to \eqref{zebra2}.
\begin{equation}\label{zebra5}
\xymatrix{
L(X\otimes Y\otimes Z)\ar[r]\ar[d]&L(RL(X\otimes Y)\otimes RLZ)\ar[d]\\
L(RLX\otimes RLY\otimes RLZ)\ar[r]\ar[d]_{\ell_*}&L(RL(RLX\otimes RLY)\otimes RLZ)\ar[d]^{\ell_*}\\
L(R(LX\otimes LY)\otimes RLZ)\ar[r]_{\eta_*}\ar[d]_{\id}&L(RLR(LX\otimes LY)\otimes RLZ)\ar[d]^{\epsilon}\\
L(R(LX\otimes LY)\otimes RLZ)\ar[r]_{\id}\ar[d]_{\epsilon\circ\ell_*}&L(R(LX\otimes LY)\otimes RLZ)\ar[d]^{\epsilon\circ\ell_*}\\
LX\otimes LY\otimes LZ\ar[r]_{\id}&LX\otimes LY\otimes LZ
}
\end{equation}
Therefore, it is sufficient to prove that the overall diagram \eqref{zebra5} commutes. The first, the second, and the fourth from above sub-diagrams
commute by general categorical nonsense, while the third from the above diagram commutes by \eqref{a2}. We are done.
\qed

\end{proof}

\comment

For Theorem \ref{keytheorem2bis} below, we need the following reformulation of the previous Lemma on the language of maps $(\Upsilon, \Theta)$ instead of maps $(\epsilon,\eta)$, see Lemma \eqref{twoweak}. The goal is reformulate \eqref{a1}, \eqref{a2} in a way using only the maps $\Upsilon$ and $\Theta$.
\begin{lemma}\label{ezhik}
Let
$$
L\colon\ \mathscr{C}\rightleftarrows\mathscr{D}\ \colon R
$$
be weak right adjoint pair of functors, let
$$
\epsilon\colon LR\to\Id_\mathscr{D}\text{    and    }\eta\colon\Id_\mathscr{C}\to RL
$$
be the adjunction maps, such that \eqref{a2} holds. Let
$$
\Upsilon\colon \Hom_\mathscr{C}(X,RY)\to\Hom_\mathscr{D}(LX,Y)\text{   and   }\Theta\colon\Hom_\mathscr{D}(LX,Y)\to \Hom_\mathscr{C}(X,RY)
$$
be the maps associated with $(\epsilon,\eta)$ by Lemma \ref{twoweak}, with
$$
\Theta\circ\Upsilon=\Id
$$
and
\begin{equation}
\epsilon(X)=\Upsilon(\boxed{RX}\xrightarrow{\id}R\boxed{X})=\Upsilon(\id)
\end{equation} and
\begin{equation}
\eta(X)=\Theta(L\boxed{X}\xrightarrow{id}\boxed{LX})=\Theta(\id)
\end{equation}
Let $\ell$ be a lax-monoidal structure on $R$.
Then the colax-monoidal structure on $L$, defined in \eqref{lax2colax}, can be equivalently defined as
\begin{equation}\label{lax2colaxbis}
\Upsilon(\boxed{X\otimes Y}\xrightarrow{\Theta(\id)\otimes\Theta(\id)}RLX\otimes RLY\xrightarrow{\ell}R(\boxed{LX\otimes LY}))
\end{equation}
The analogous claim is true for the left weak adjoint pairs, and for lax$\leftrightarrow$colax switched.
\end{lemma}
The proof is a direct application of the above discussion, and is left to the reader.

\qed
\endcomment

\appsubsection{\sc Intrinsic characterization of the lax-colax duality}\label{appintrinsic}
\begin{defn}{\rm
Let $\mathscr{C}$ and $\mathscr{D}$ be monoidal categories, and let
$$
L\colon\ \mathscr{C}\rightleftarrows\mathscr{D}\ \colon R
$$
be a {\it genuine} adjunction.
\begin{itemize}
\item[1)]
We say that a colax-monoidal structure $c$ on $L$ and a lax-monoidal structure $\ell$ on $R$ are {\it compatible} if the following diagram commutes:
\begin{equation}\label{krokodil}
\xymatrix{
L(RX\otimes RY)\ar[r]^{\ell_*}\ar[d]_{c_*}&LR(X\otimes Y)\ar[d]^{\epsilon_*}\\
LR(X)\otimes LR(Y)\ar[r]^{\quad\quad\epsilon_*\otimes\epsilon_*}&X\otimes Y
}
\end{equation}
\item[2)] Suppose a lax-monoidal structure $\ell$ on $R$ is given. We call the colax-monoidal structure $c$ on $L$, defined by \eqref{lax2colax}, the {\it canonical} colax-monoidal structure, induced by $\ell$. As well, suppose a colax-monoidal structure $c$ on $L$ is given. We call a lax-monoidal structure $\ell$ on $R$, defined by \eqref{colax2lax}, the {\it canonical} lax-monoidal structure, induced by $c$.
\end{itemize}
}
\end{defn}
\begin{lemma}\label{utka}{\itshape
Let $\mathscr{C}$ and $\mathscr{D}$ be monoidal categories, and let
$$
L\colon\ \mathscr{C}\rightleftarrows\mathscr{D}\ \colon R
$$
be a {\it genuine} adjunction.
\begin{itemize}
\item[(i)] given a lax-monoidal structure $\ell$ on $R$, it is compatible with the canonical colax-monoidal structure on $L$, induced by $\ell$. Analogously, given a colax-monoidal structure $c$ on $L$, it is compatible with the canonical lax-monoidal structure on $R$, induced by $c$;
\item[(ii)] given a compatible pair $(c,\ell)$ of (co)lax-monoidal structures, each of them is uniquely defined by another. By (i), it implies that the only compatible pairs of (co)lax-monoidal structures are the canonically induced ones.
\end{itemize}
}
\end{lemma}
It is simple and fairly standard, but the result follows from our more subtle Lemmas \ref{strangelemma}, \ref{strangelemma1} as well.

\qed

Our goal in this Subsection is to study what happens with Lemma \ref{utka} in weak adjoint context. These results are used below in Appendix B, in proofs of Theorem \ref{keytheorem2bis} and of Theorem \ref{keytheorem2bisbis}.

We have:
\begin{lemma}\label{strangelemma}{\itshape
Let
$$
L\colon\ \mathscr{C}\rightleftarrows\mathscr{D}\ \colon R
$$
be a weak right adjunction.
\begin{itemize}
\item[(1)] Suppose a lax-monoidal structure $\ell$ on $R$ is given. Then the canonical colax-monoidal structure $c$ on $L$, induced by $\ell$ as in \eqref{lax2colax}, is compatible with $\ell$ (in the sense that diagram \eqref{krokodil} commutes);
\item[(2)] suppose a colax-monoidal structure $c$ on $L$ and a lax-monoidal structure $\ell$ on $R$ are given, such that the diagram \eqref{krokodil} commutes. Then $\ell$ can be uniquely reconstructed by $c$.
\end{itemize}
}
\end{lemma}
\begin{remark}{\rm
We would emphasize a rather unexpected claim in (ii): not a colax-monoidal structure $c$ on $L$ making \eqref{krokodil} commutative, one particular possibility for it is given in (i),
is uniquely defined by $\ell$, but, vice versa, $\ell$ is uniquely reconstructed by $c$.
}
\end{remark}
Before proving Lemma \ref{strangelemma}, let us formula a pattern of it for left weak adjunctions.

Consider the diagram
\begin{equation}\label{krokodil1}
\xymatrix{
X\otimes Y\ar[r]^{\eta\otimes\eta\quad}\ar[d]_{\eta}&RL(X)\otimes RL(Y)\ar[d]^{\ell_*}\\
RL(X\otimes Y)\ar[r]^{c_*\ \ \ }&R(L(X)\otimes L(Y))
}
\end{equation}
Here the maps are defined from a lax-monoidal structure $\ell$ on $R$, a colax-monoidal structure $c$ on $L$, and from the adjunction $\eta\colon\Id_\mathscr{C}\to RL$.
\begin{remark}{\rm
For a pair of genuine adjoint functors, the diagram \eqref{krokodil1} commutes iff \eqref{krokodil} does. For weak adjunctions, it is not anymore the case, and we have two different diagrams.}
\end{remark}
The ``dual'' to Lemma \ref{strangelemma} is
\begin{lemma}\label{strangelemma1}{\itshape
Let
$$
L\colon\ \mathscr{C}\rightleftarrows\mathscr{D}\ \colon R
$$
be a weak left adjunction.
\begin{itemize}
\item[(1)] Suppose a colax-monoidal structure $c$ on $L$ is given. Then the canonical lax-monoidal structure $\ell$ on $R$, induced by $c$ as in \eqref{colax2lax}, is compatible with $\ell$ in the sense that diagram \eqref{krokodil1} commutes;
\item[(2)] suppose a colax-monoidal structure $c$ on $L$ and a lax-monoidal structure $\ell$ on $R$ are given, such that the diagram \eqref{krokodil1} commutes. Then $c$ can be uniquely reconstructed by $\ell$.
\end{itemize}
}
\end{lemma}
{\it Proof of Lemma \ref{strangelemma}:}

(i):

Suppose $\ell$ is given, define $c$ by \eqref{lax2colax}.
Then the lower-left path of diagram \eqref{krokodil} is:
\begin{equation}
L(RX\otimes RY)\xrightarrow{\eta_*\otimes\eta_*}L(RLRX\otimes RLRY)\xrightarrow{\ell_*}LR(LRX\otimes LRY)\xrightarrow{\epsilon(\epsilon^{\otimes 2})} X\otimes Y
\end{equation}
The composition
$$
L(RLRX\otimes RLRY)\xrightarrow{\ell_*}LR(LRX\otimes LRY)\to X\otimes Y
$$
is the same, by the functoriality, that the composition
$$
L(R\boxed{LR}X\otimes R\boxed{LR}Y)\xrightarrow{\epsilon_*\otimes\epsilon_*}L(RX\otimes RY)\xrightarrow{\ell_*}LR(X\otimes Y)\xrightarrow{\epsilon} X\otimes Y
$$
(in the first arrow, $\eta$'s act on the boxed factors).

The latter composition is equal to the upper-right path of diagram \eqref{krokodil}, because the composition
$$
R\xrightarrow{\eta} RLR\xrightarrow{\epsilon}R
$$
is equal to the identity map of $R$, by assumption.

\smallskip

(ii):

Build up a top level to the  $R(\eqref{krokodil})$, as follows:
\begin{equation}\label{krokodil2}
\xymatrix{
RX\otimes RY\ar[r]^{\ell}\ar[d]_{\eta}&R(X\otimes Y)\ar[d]^{\eta}\\
RL(RX\otimes RY)\ar[r]^{\ell_*}\ar[d]_{c_*}&RLR(X\otimes Y)\ar[d]^{\epsilon_*}\\
R(LR(X)\otimes LR(Y))\ar[r]^{\quad\quad\epsilon_*\otimes\epsilon_*}&R(X\otimes Y)
}
\end{equation}
It follows from our assumption of right weak adjunction that the upper-right path of the overall diagram \eqref{krokodil2} gives the lax-monoidal structure $\ell\colon RX\otimes RY\to R(X\otimes Y)$. On the other hand, the lower-left path expresses it in $c$, $\epsilon$, and $\eta$.

\qed

The proof of Lemma \ref{strangelemma1} is analogous and is left to the reader.

\appsubsection{\sc (Co)limits of weak adjoint pairs}\label{appweak}
We start with the definition of a morphism between weak adjoint pairs.
\begin{defn}\label{adef2}{\rm
Let $\mathscr{C}$ and $\mathscr{D}$ be two categories, and suppose
\begin{equation}\label{morphismone}
L\colon\mathscr{C}\rightleftarrows\mathscr{D}\colon R
\end{equation}
and
\begin{equation}\label{morphismtwo}
L^\prime\colon\mathscr{C}\rightleftarrows\mathscr{D}\colon R^\prime
\end{equation}
are two {\it weak right adjunctions} (resp., weak left adjunctions), see Definition \ref{adef1}.

{\it A morphism} $\Psi$  from \eqref{morphismone} to \eqref{morphismtwo} is given by the following data:
\begin{itemize}
\item[(i)] two maps of functors $\Psi_L\colon L^\prime\to L$ (in backward direction) and $\Psi_R\colon R\to R^\prime$,
which induce maps of bifunctors $$\Psi_{L*}\colon \Hom_\mathscr{D}(LX,Y)\to\Hom_\mathscr{D}(L^\prime X,Y)$$ and
$$
\Psi_{R*}\colon \Hom_\mathscr{C}(X,RY)\to\Hom_\mathscr{C}(X,R^\prime Y)
$$
\item[(ii)] the diagram
\begin{equation}
\xymatrix{
\Hom_\mathscr{D}(LX,Y)\ar[r]^{\Theta}\ar[d]_{\Psi_{L*}}&\Hom_\mathscr{C}(X,RY)\ar[d]^{\Psi_{R*}}\\
\Hom_\mathscr{D}(L^\prime X,Y)\ar[r]^{\Theta^\prime}&\Hom_\mathscr{C}(X,R^\prime Y)
}
\end{equation}
and the diagram
\begin{equation}
\xymatrix{
\Hom_\mathscr{C}(X,RY)\ar[r]^{{\Upsilon}}\ar[d]_{\Psi_{R*}}&\Hom_\mathscr{D}(LX,Y)\ar[d]^{\Psi_{L*}}\\
\Hom_\mathscr{C}(X,R^\prime Y)\ar[r]^{\Upsilon^\prime}&\Hom_\mathscr{D}(L^\prime X,Y)
}
\end{equation}
are commutative,
where $(\Theta,\Upsilon)$ and $(\Theta^\prime,\Upsilon^\prime)$ are corresponded to the weak right adjunctions \eqref{morphismone} (correspondingly, \eqref{morphismtwo}) by \eqref{awaytodefine}.
\end{itemize}
}
\end{defn}
\bigskip

We now turn to the discussion of (co)limits of weak adjoint functors and of (co)lax-monoidal structures.

Let $I$ and $\mathscr{C}$ be categories; we refer to a functors from $I$ to $\mathscr{C}$ as {\it diagrams}.
Recall that a {\it limit} $\lim\mathcal{D}$ of a diagram $\mathcal{D}\colon I\to\mathscr{C}$ is defined, if it exists, by one of the two equivalent ways:
\begin{equation}\label{botinok1}
\Hom_\mathscr{C}(X,\lim\mathcal{D})={\lim}_I\Hom(X,\mathscr{D}(i))=\Hom_{I\to\mathscr{C}}(\Delta X, \mathcal{D})
\end{equation}
where $\Delta(X)$ is the functor $\Delta(X)\colon I\to\mathscr{C}$ defined on objects as $\Delta(X)(i)=X$, and defined as the identity morphism of $X$ on all morphisms in $I$.

Analogously, a {\it colimit} $\colim\mathcal{D}$ of the diagram $\mathcal{D}\colon I\to\mathscr{C}$ is defined, if it exists, by
\begin{equation}\label{botinok2}
\Hom_\mathscr{C}(\colim \mathcal{D}, Y)={\lim}_I\Hom(\mathcal{D}(i),Y)=\Hom_{I\to\mathscr{C}}(\mathcal{D},\Delta Y)
\end{equation}

For any two categories $\mathscr{P}$ and $\mathscr{Q}$, there is a canonical equivalence
\begin{equation}\label{3opp}
\mathscr{F}un(\mathscr{P},\mathscr{Q})\simeq \mathscr{F}un(\mathscr{P}^\opp,\mathscr{Q}^\opp)^\opp
\end{equation}

Thus, a functor $\mathcal{D}\colon{I}\to\mathscr{C}$, can be regarded as well as a functor $\mathcal{D}^\opp\colon I^\opp\to\mathscr{C}^\opp$.
Then
\begin{equation}\label{limcolim}
{\lim}_{I}\mathcal{D}=(\colim_{I^\opp}\mathcal{D}^\opp)^\opp
\end{equation}

Let $\alpha\colon J\to I$ be a functor. Then with any functor $\mathcal{D}\colon I\to \mathscr{C}$ one assigns the composition
$\mathcal{D}\circ\alpha\colon J\to\mathscr{C}$.

Recall for the reference below
\begin{lemma}
In the above notations, the functor $\alpha\colon J\to I$ defines canonical morphisms
\begin{equation}\label{j2ilim}
\alpha^*\colon \lim \mathcal{D}\to\lim(\mathcal{D}\circ\alpha)
\end{equation}
and
\begin{equation}\label{j2icolim}
\alpha_*\colon\colim(\mathcal{D}\circ \alpha)\to\colim\mathcal{D}
\end{equation}
functorial in diagram $\mathcal{D}\colon I\to\mathscr{C}$.
\end{lemma}
\begin{proof}
For the case of limits, there are two adjunctions
\begin{equation}
\Hom_\mathscr{C}(X,\lim\mathcal{D})\simeq \Hom_{I\to\mathscr{C}}(\Delta_I X,\mathcal{D})
\end{equation}
and
\begin{equation}
\Hom_\mathscr{C}(X,\lim(\mathcal{D}\circ \alpha))\simeq \Hom_{J\to\mathscr{C}}(\Delta_J X,\mathcal{D}\circ\alpha)
\end{equation}
The map $\alpha\colon J\to I$ defines a map of the right-hand-sides
\begin{equation}
\Hom_{I\to\mathscr{C}}(\Delta_I X,\mathcal{D})\to\Hom_{J\to\mathscr{C}}(\Delta_J X,\mathcal{D}\circ\alpha)
\end{equation}
which defines a map of the left-hand sides (as a map of bifunctors):
\begin{equation}\label{yl}
\Hom_\mathscr{C}(X,\lim\mathcal{D})\to\Hom_\mathscr{C}(X,\lim (\mathcal{D}\circ\alpha))
\end{equation}
Then \eqref{yl} defines a functor (unique up to a unique isomorphism)
$$
\lim\mathcal{D}\to\lim(\mathcal{D}\circ\alpha)
$$
by the Yoneda lemma.

The case of colimits is analogous (and dual).
\end{proof}

See [ML, Ch.V] and [KaS, Ch.2] for more detail on (co)limits.

Recall a well-known fact on the (co)limits in the categories of functors.
\begin{prop}\label{spravka}{\itshape
Let $\mathscr{C}$ and $\mathscr{D}$ be two categories. Consider the category $\mathscr{F}un(\mathscr{C},\mathscr{D})$. Let $I$ be an indexing category, and let $\mathcal{D}\colon I\to \mathscr{F}un(\mathscr{C},\mathscr{D})$ be a diagram.
\begin{itemize}
\item[(i)] When $\mathscr{D}$ is complete, the limit $\lim\mathcal{D}$ exists, and can be computed point-wise:
\begin{equation}\label{strekoza1}
(\lim\mathscr{D})(X)=\lim(\mathcal{D}(X))
\end{equation}
Analogously, when $\mathscr{D}$ is cocomplete, the colimit $\colim\mathscr{D}$ exists, and can be computed point-wise
\begin{equation}\label{strekoza2}
(\colim\mathcal{D})(X)=\colim(\mathcal{D}(X))
\end{equation}
\item[(ii)] Suppose $\mathscr{D}$ is cocomplete and $\mathscr{C}$ is complete, and let
$\mathcal{D}_L\colon I^\opp\to\mathscr{F}un(\mathscr{C},\mathscr{D})$ and $\mathcal{D}_R\colon I\to \mathscr{F}un(\mathscr{D},\mathscr{C})$ be two diagrams. Suppose that for each $i\in I$, the pair of functors
$$
\mathcal{D}_L(i)\colon\mathscr{C}\rightleftarrows\mathscr{D}\colon\mathcal{D}_R(i)
$$
is genuine adjoint. Then the functors
$$
\colim\mathcal{D}_L\colon \mathscr{C}\rightleftarrows \mathscr{D}\colon\lim\mathcal{D}_R
$$
exist, and are genuine adjoint.
\end{itemize}
}
\end{prop}
Recall a proof, for completeness.

Denote by $\nlim$ and $\ncolim$ the functors, defined by \eqref{strekoza1} and \eqref{strekoza2}, correspondingly.
One directly shows that
\begin{equation}\label{strekoza3}
\Hom_{\mathscr{F}un(\mathscr{C},\mathscr{D})}(F, \nlim\mathcal{D})=
{\lim}_{i\in I}\Hom_{\mathscr{F}un(\mathscr{C},\mathscr{D})}(F, \mathcal{D}(i))
\end{equation}
and
\begin{equation}\label{strekoza4}
\Hom_{\mathscr{F}un(\mathscr{C},\mathscr{D})}(\ncolim\mathcal{D}, G)=
{\lim}_{i\in I}\Hom_{\mathscr{F}un(\mathscr{C},\mathscr{D})}(\mathcal{D}(i), G)
\end{equation}
Now the claim (i) follows from \eqref{botinok1}, \eqref{botinok2}, and the Yoneda lemma.

The claim (ii) is a direct consequence from the claim (i).

\qed

\begin{remark}{\rm
The (co)limits in the categories of functors {\it may exist} as well when the target category is {\it not} (co)complete. See [Kel], Section 3.3.
}
\end{remark}

Henceforth, we suppose that all necessary (co)limits in the target categories exist, and the (co)limits in the categories of functors can be computed point-wise.

\begin{defn}\label{adef3}{\rm
Let $I$ be an indexing category, and let
$\mathcal{D}_L\colon I\to \mathscr{F}un(\mathscr{C},\mathscr{D})^\opp$ and $\mathcal{D}_R\colon I\to \mathscr{F}un(\mathscr{D},\mathscr{C})$ be two diagrams. Suppose that for each fixed $i\in I$ the maps
\begin{equation}
\mathcal{D}_L(i)\colon \mathscr{C}\rightleftarrows\mathscr{D}\colon \mathcal{D}_R(i)
\end{equation}
is a weak right adjunction, for some
\begin{equation}
\epsilon(i)\colon \mathcal{D}_L(i)\circ \mathcal{D}_R(i)\to \Id_{\mathscr{D}},\ \ \eta(i)\colon\Id_\mathscr{C}\to\mathcal{D}_R(i)\circ\mathcal{D}_L(i)
\end{equation}
and for each morphism $f\colon i\to j$ in $I$, the maps of functors
\begin{equation}
(\mathcal{D}_L(j)\xrightarrow{f^*}\mathcal{D}_L(i),\ \ \mathcal{D}_R(i)\xrightarrow{f_*}\mathcal{D}_R(j))
\end{equation}
form {\it a morphism between weak right adjunctions}, see Definition \ref{adef2}. Then the diagrams $\mathcal{D}_L$ and $\mathcal{D}_R$ are called
{\it weak right adjoint diagrams}.
}
\end{defn}

We have:
\begin{lemma}\label{lemmaa}{\itshape
Suppose $\mathscr{C},\mathscr{D}$ are two categories, with $\mathscr{C}$ complete and $\mathscr{D}$ cocomplete. Let
$$
\mathcal{D}_L\colon \mathscr{C}\rightleftarrows \mathscr{D}\colon\mathcal{D}_R
$$
be weak right adjoint diagrams, see Definition \ref{adef3}. Then the functors
\begin{equation}
\colim\mathcal{D}_L\colon\mathscr{C}\rightleftarrows \mathscr{D}\colon \lim\mathcal{D}_R
\end{equation}
has a natural structure of a weak right adjunction.
}
\end{lemma}
\begin{proof}
We use the description of a pair of (right) weak adjoint functors in terms of maps $\Theta$ and $\Upsilon$ (see \eqref{podrugomu}) satisfying \eqref{tylen1}.

We can rephrase the statement of Lemma, saying that there are diagrams of morphisms of bifunctors
\begin{equation}
\Theta_\mathcal{D}(i)\colon\quad \Hom_\mathscr{D}(\mathcal{D}_L(i)(X),Y)\to\Hom_\mathscr{C}(X,\mathcal{D}_R(i)(Y))
\end{equation}
and
\begin{equation}
\Upsilon_\mathcal{D}(i)\colon\quad\Hom_\mathscr{C}(X,\mathcal{D}_R(i)(Y))\to\Hom_\mathscr{D}(\mathcal{D}_L(i)(X),Y)
\end{equation}
such that
\begin{equation}
\Theta_\mathcal{D}(i)\circ \Upsilon_\mathcal{D}(i)=\Id\text{\qquad \rm{for any $i\in I$}}
\end{equation}
Now we can define
\begin{equation}
\lim\Theta_\mathcal{D}:={\lim}_I\Theta_\mathcal{D}(i)\colon\quad  \lim\Hom_\mathscr{D}(\mathcal{D}_L(i)(X),Y)\to\lim\Hom_\mathscr{C}(X,\mathcal{D}_R(i)(Y))
\end{equation}
and
\begin{equation}
\lim\Upsilon_\mathcal{D}:={\lim}_I\Upsilon_\mathcal{D}(i)\colon\quad \lim\Hom_\mathscr{C}(X,\mathcal{D}_R(i)(Y))\to\lim\Hom_\mathscr{D}(\mathcal{D}_L(i)(X),Y)
\end{equation}
By Proposition \ref{spravka}, the above limit morphisms are in fact
\begin{equation}
\lim\Theta_\mathcal{D}:\quad \Hom_\mathscr{D}(\colim\mathcal{D}_L(X),Y)\to\Hom_\mathscr{C}(X,\lim\mathcal{D}_R(Y))
\end{equation}
and
\begin{equation}
\lim\Upsilon_\mathcal{D}:\quad \Hom_\mathscr{C}(X,\lim\mathcal{D}_R(Y))\to\Hom_\mathscr{D}(\colim\mathcal{D}_L(X),Y)
\end{equation}
Clear the morphisms $\lim\Theta_\mathcal{D}$ and $\lim\Upsilon_\mathcal{D}$ are maps of bifunctors.
Therefore, they define, by Lemma \ref{twoweak}, {\it some canonical} maps of functors
$$
\lim\epsilon\colon \quad \colim\mathcal{D}_L\circ \lim\mathcal{D}_R\to\Id_\mathcal{D}
$$
and
$$
\lim\eta\colon\quad \lim\mathcal{D}_R\circ\colim\mathcal{D}_L\to\Id_\mathscr{C}
$$
It remains to prove that $(\lim\epsilon,\lim\eta)$ define a weak right adjunction on the functors
$$
\colim\mathcal{D}_L\colon\quad \mathscr{C}\rightleftarrows\mathscr{D}\quad\colon\lim\mathcal{D}_R
$$
By Lemma \ref{twoweak}, it is enough to prove that
\begin{equation}\label{zaichik}
\lim\Theta_\mathcal{D}\ \circ\ \lim\Upsilon_\mathcal{D}\ =\Id
\end{equation}
As for each $i\in I$ we have a weak right adjunction, one has
\begin{equation}\label{zaichik0}
\Theta_\mathcal{D}(i)\ \circ\ \Upsilon_\mathcal{D}(i)\ =\Id
\end{equation}
Then \eqref{zaichik} follows from \eqref{zaichik0} by passing to limits.
\end{proof}

Let $\mathcal{D}_L\colon\mathscr{C}\rightleftarrows \mathscr{D}\colon\mathcal{D}_R$ be weak right adjoint diagrams, indexed by a category $I$, see Definition \ref{adef3}. Suppose that each functor $\mathcal{D}_R(i)$ is endowed with a lax-monoidal structure $\ell(i)$, such that the diagram
\begin{equation}
\xymatrix{
\mathcal{D}_R(i)(X)\otimes \mathcal{D}_R(i)(Y)\ar[r]^{\qquad\ell(i)}\ar[d]_{f_*}&\mathcal{D}_R(i)(X\otimes Y)\ar[d]^{f_*}\\
\mathcal{D}_R(j)(X)\otimes \mathcal{D}_R(j)(Y)\ar[r]^{\qquad\ell(j)}&\mathcal{D}_R(j)(X\otimes Y)
}
\end{equation}
is commutative for any morphism $f\colon i\to j$ in $I$, and for any $X,Y$ in $\mathscr{D}$.
In this case we say that we are given {\it a diagram of lax-monoidal structures}.

Analogously, using $\mathcal{D}_L(i),\ i\in I$, and a colax-monoidal structure $c(i)$ on each $\mathcal{D}_L(i)$, we define {\it a diagram of colax-monoidal structures}.

Suppose we have weak right adjoint diagrams $\mathcal{D}_L\colon\mathscr{C}\rightleftarrows \mathscr{D}\colon\mathcal{D}_R$, and a diagram of lax-monoidal structures on $\mathcal{D}_R$. Then, for each $i\in I$, we define a colax-monoidal structure $c^\ell(i)$ on $\mathcal{D}_L(i)$, adjoint to the lax-monoidal structure $\ell(i)$ on $\mathcal{D}_R(i)$.

\begin{lemma}{\itshape
In the above set-up, the colax-monoidal structures $c^\ell(i)$, $i\in I$, on $\mathcal{D}_L(i)$, form a diagram of colax-monoidal structures.
}
\end{lemma}
It is clear.
\qed

Suppose there is a diagram $\mathcal{D}\colon I\to\mathscr{F}un(\mathscr{D},\mathscr{C})$, and a diagram of lax-monoidal structures $\ell(i)$ on $\mathcal{D}(i)$.

We define the {\it limit lax-monoidal structure} ${\lim}_{i\in I}\ell(i)$ on the functor ${\lim}_{i\in I}\mathcal{D}(i)$, as follows.

For any $i\in I$, and any $X,Y\in \mathscr{C}$, there are universal maps
$$
(\lim \mathcal{D})(X)\to \mathcal{D}(i)(X)\quad\text{and}\quad (\lim\mathcal{D})(Y)\to\mathcal{D}(j)(Y)
$$
which define
\begin{equation}
(\lim\mathcal{D})(X)\otimes(\lim\mathcal{D})(Y)\to \mathcal{D}(i)(X)\otimes\mathcal{D}(j)(Y)
\end{equation}
which is compatible with the morphisms in $I\times I$. Then the universal property gives
\begin{equation}
\Xi_1(X,Y)\colon\ (\lim\mathcal{D})(X)\otimes(\lim\mathcal{D}(Y))\to {\lim}_{I\times I}(\mathcal{D}(i)(X)\otimes\mathcal{D}(j)(Y))
\end{equation}
Next, there is a map
\begin{equation}
\Xi_2(X,Y)\colon\ {\lim}_{I\times I}(\mathcal{D}(i)(X)\otimes\mathcal{D}(j)(Y))\to {\lim}_{I}(\mathcal{D}(i)(X)\otimes\mathcal{D}(i)(Y))
\end{equation}
The latter map is $\alpha^*$ (see \eqref{j2ilim}) for the diaginal functor $\alpha\colon I\to I\times I$.

Finally, we apply the lax-monoidal structure $\ell(i)$ for each $\mathcal{D}(i)$, and take the limit:
\begin{equation}
\Xi_3(X,Y)\colon\ {\lim}_{I}(\mathcal{D}(i)(X)\otimes\mathcal{D}(i)(Y))\xrightarrow{\ell(i)}{\lim}_I\mathcal{D}(i)(X\otimes Y)=\lim\mathcal{D}(X\otimes Y)
\end{equation}

Define the limit lax-monoidal structure $\lim\ell$ on $\lim\mathcal{D}$ as the composition
\begin{equation}\label{limlax}
\lim\ell=\Xi_3\circ\Xi_2\circ \Xi_1
\end{equation}

Analogously one defines the colimit of colax-monoidal structures.

One of our two most principal claims in this Appendix (along with Theorem \ref{keytheorem2bisbis}) is the following:
\begin{apptheorem}\label{keytheorem2bis}{\itshape
Suppose $\mathscr{C}$, $\mathscr{D}$ are categories, with $\mathscr{C}$ complete and $\mathscr{D}$ cocomplete. Let $I$ be an indexing category. Let $$\mathcal{D}_L\colon I\to\mathscr{F}un(\mathscr{C},\mathscr{D})^\opp,\ \
\mathcal{D}_R\colon I\to\mathscr{F}un(\mathcal{D},\mathcal{C})$$ are diagrams, such that
$$
\mathcal{D}_L\colon\mathcal{C}\rightleftarrows \mathcal{D}\colon\mathcal{D}_R
$$
is a weak right adjunction of diagrams, see Definition \ref{adef3}, {\bf with $\mathcal{D}_R(i)$ the constant functor (not depending on $i\in I$)}, which we denote by $\mathcal{D}_R$.  Suppose there is given a lax-monoidal structure $\ell$ on $\mathcal{D}_R$. Consider the natural weak right adjunction
\begin{equation}\label{limit}
\colim\mathcal{D}_L\colon\  \mathscr{C}\rightleftarrows\mathscr{D}\ \colon \lim\mathcal{D}_R=\mathcal{D}_R
\end{equation}
constructed in Lemma \ref{lemmaa}. {\bf Suppose, furthermore, that \eqref{limit} is a genuine adjunction.}

Then
the colimit colax-monoidal structure
$\colim_{i\in I}c^{\ell}(i)$ on the functor $\colim\mathcal{D}_L$, where each $c^\ell(i)$ is the canonical colax-monoidal structure induced by $\ell$, is isomorphic to the canonical colax-monoidal structure on $\colim\mathcal{D}_L$, induced by the limit lax-monoidal structure ${\lim}_{i\in I}\ell=\ell$ on the functor $\lim\mathcal{D}_R=\mathcal{D}_R$.
}
\end{apptheorem}
\begin{proof}
The proof is based on the following lemma.
\begin{lemma}\label{soroka}{\itshape
Let $\mathscr{C}$, $\mathscr{D}$ be monoidal categories, with $\mathscr{C}$ complete and $\mathscr{D}$ cocomplete, and let
$\mathcal{D}_L\colon I\to\mathscr{F}un(\mathscr{C},\mathscr{D})^\opp$ and $\mathcal{D}_R\colon I\to\mathscr{F}un(\mathscr{D},\mathscr{C})$ be diagrams, with $\mathcal{D}_R(i)=\mathcal{D}_R$ the constant functor. Suppose that for each $i$ the functors
$$
\mathcal{D}_L(i)\colon\ \mathscr{C}\rightleftarrows\mathscr{D}\ \colon\mathcal{D}_R(i)=\mathcal{D}_R
$$
are weak right adjoint, such that one has a diagram of weak right adjunctions (see Definition \ref{adef3}).
Suppose, furthermore, that there are given a lax-monoidal structure $\ell$ on $\mathcal{D}_R$, and a diagram of colax-monoidal structures $c(i)$ on $\mathcal{D}_L(i)$, such that for each $i\in I$ the diagram \eqref{krokodil} commutes for $(c(i),\ell)$.

Then the pair $(\colim c(i),\ell)$ of the colimit colax-monoidal structure $\colim c(i)$ on the functor $\colim\mathcal{D}_L(i)$ and of the lax-monoidal structure $\lim\ell=\ell$ on the functor $\lim\mathcal{D}_R=\mathcal{D}_R$, the diagram \eqref{krokodil} commutes as well.
}
\end{lemma}
\begin{proof}
First of all, we use \eqref{limcolim} and replace in \eqref{krokodil} all entries of $\colim_I\mathcal{D}_L$ by ${\lim}_{I^\opp}\mathcal{D}^\opp$.
According to that, we replace the notation for the colimit colax-monoidal structure $\colim_I c^\ell(i)$ by ${\lim}_{I^\opp}c^\ell(i)$.
\comment
Then the diagram \eqref{krokodil} for the pair of functors $\colim\mathcal{D}_L$ and $\lim\mathcal{D}_R$ into three subdiagrams, as follows:
\begin{equation}{\footnotesize
\xymatrix{
{\lim}_{I^\opp}\mathcal{D}_L^\opp({\lim}_I\mathcal{D}_RX\otimes {\lim}_I\mathcal{D}_RY)\ar[d]_{{\lim}_{I^\opp} c^\ell(i)}\ar[r]&{\lim}_{(i,j,k)\in I^\opp\times I\times I}
\mathcal{D}_L^\opp(i)(\mathcal{D}_R(j)X\otimes \mathcal{D}_R(k)Y)\ar[d]^{{\lim}_{I^\opp}(c^\ell(i))_*}\\
{\lim}_{I^\opp}\mathcal{D}_L^\opp({\lim}_I\mathcal{D}_R(X))\otimes {\lim}_{I^\opp}\mathcal{D}_L^\opp({\lim}_I\mathcal{D}_R(Y))\ar[r]&
{\lim}_{(i,j,k)\in I^\opp\times I\times I}\Bigl(\mathcal{D}_L^\opp(i)(\mathcal{D}_R(j)(X))\otimes\mathcal{D}_L^\opp(i)(\mathcal{D}_R(k)(Y))\Bigr)
}
}
\end{equation}
followed by
\begin{equation}{\footnotesize
\xymatrix{
{\lim}_{(i,j,k)\in I^\opp\times I\times I}
\mathcal{D}_L^\opp(i)(\mathcal{D}_R(j)X\otimes \mathcal{D}_R(k)Y)\ar[d]^{{\lim}_{I^\opp}(c^\ell(i))_*}\ar[r]\ar[d]&
{\lim}_{(i,j)\in I^\opp\times I}\mathcal{D}_L^\opp(i)(\mathcal{D}_R(j)(X)\otimes\mathcal{D}_R(j)(Y))\ar[d]\\
{\lim}_{(i,j,k)\in I^\opp\times I\times I}\Bigl(\mathcal{D}_L^\opp(i)(\mathcal{D}_R(j)(X))\otimes\mathcal{D}_L^\opp(i)(\mathcal{D}_R(k)(Y))\Bigr)\ar[r]&
{\lim}_{(i,j)\in I^\opp\times I}\Bigl(\mathcal{D}_L^\opp(i)(\mathcal{D}_R(j)(X))\otimes\mathcal{D}_L^\opp(i)(\mathcal{D}_R(j)(Y))\Bigr)
}
}
\end{equation}
followed by (wee replace everywhere ${\lim}_{I^\opp\times I}\mathscr{X}$ by $\colim_{I\times I^\opp}\mathscr{X}^\opp$ (we omit the overall $\opp$ in the target for simplicity):
\begin{equation}{\scriptsize
\xymatrix{
{\colim}_{(i,j)\in I\times I^\opp}\mathcal{D}_L(i)(\mathcal{D}_R^\opp(j)(X)\otimes\mathcal{D}_R^\opp(j)(Y))\ar[r]\ar[d]&{\colim}_{(i,j)\in I\times I^\opp}
\mathcal{D}_L(i)(\mathcal{D}_R^\opp(j)(X\otimes Y))\\
{\colim}_{(i,j)\in I\times I^\opp}\Bigl(\mathcal{D}_L(i)(\mathcal{D}_R^\opp(j)(X))\otimes\mathcal{D}_L(i)(\mathcal{D}_R^\opp(j)(Y))\Bigr)\ar[r]&
{\colim}_{(i,j)\in I\times I^\opp}\Bigl(\mathcal{D}_L(i)(\mathcal{D}_R^\opp(j)(X)))\otimes{\colim}_{(i,j)\in I\times I^\opp}\Bigl(\mathcal{D}_L(i)(\mathcal{D}_R^\opp(j)(Y))\Bigr)
}
}
\end{equation}
and, finally
\begin{equation}
\end{equation}
When the three above diagrams are composed in a row from the left to the right, one gets the diagram \eqref{krokodil}.
Therefore, it is sufficient to prove that each of three diagrams \eqref{krevetka1}-\eqref{krevetka3} commutes.
\endcomment
First of all, we use \eqref{limcolim} and replace in \eqref{krokodil} all entries of $\colim_I\mathcal{D}_L$ by ${\lim}_{I^\opp}\mathcal{D}^\opp$.
According to that, we replace the notation for the colimit colax-monoidal structure $\colim_I c^\ell(i)$ by ${\lim}_{I^\opp}c^\ell(i)$.

Within our proof, we divide the diagram \eqref{krokodil} into 2 commutative diagrams, \eqref{konfetka1} and \eqref{konfetka2} below, which implies that the overall diagram \eqref{krokodil} commutes as well. 
\begin{equation}\label{konfetka1}
{\tiny
\xymatrix{
({\lim}_{I^\opp}\mathcal{D}_L)(\mathcal{D}_R(X)\otimes\mathcal{D}_R(Y))\ar[d]_{\ell_*}\ar[rr]^{\simeq, \text{ by }\eqref{botinok1}}&&
{\lim}_{I^\opp}\Bigl(\mathcal{D}_L(i)(\mathcal{D}_R(X)\otimes\mathcal{D}_R(Y))\Bigr)\ar[d]^{(\ell)_*}\\
({\lim}_{I^\opp}\mathcal{D}_L)(\mathcal{D}_R(X))\otimes({\lim}_{I^\opp}\mathcal{D}_L)(\mathcal{D}_R(Y))\ar[r]&
{\lim}_{I^\opp\times I^\opp}\Bigl(\mathcal{D}_L(i)\mathcal{D}_R(X)\otimes\mathcal{D}_L(j)\mathcal{D}_R(Y)\Bigr)\ar[r]^{\Delta}&
{\lim}_{I^\opp}\Bigl(\mathcal{D}_L(i)\mathcal{D}_R(X)\otimes\mathcal{D}_L(i)\mathcal{D}_R(Y)\Bigr)
}
}
\end{equation}

\begin{equation}\label{konfetka2}
{\tiny
\xymatrix{
{\lim}_{I^\opp}\Bigl(\mathcal{D}_L(i)(\mathcal{D}_R(X)\otimes\mathcal{D}_R(Y))\Bigr)\ar[r]^{\ell_*}\ar[d]_{(c^\ell(i))_*}&
{\lim}_{I^\opp}\Bigl(\mathcal{D}_L(i)(\mathcal{D}_R(X\otimes Y)))\Bigr)\ar[r]&X\otimes Y\ar[d]^{\id}\\
{\lim}_{I^\opp}\Bigl(\mathcal{D}_L(i)\mathcal{D}_R(X)\otimes\mathcal{D}_L(i)\mathcal{D}_R(Y)\Bigr)\ar[rr]&&
X\otimes Y
}
}
\end{equation}
The commutativity of \eqref{konfetka1} is an exercise on the definition \eqref{limlax} of the limit lax structure (in its colax version), and is left to the reader.

The commutativity of \eqref{konfetka2} follows from the assumption that for each $i$ the diagram \eqref{krokodil} commutes, as \eqref{konfetka2} is the the limit of the commutative diagrams \eqref{krokodil}, and therefore is commutative by its own.

Lemma is proven.

\end{proof}
We continue to prove Theorem \ref{keytheorem2bis}.

By assumption, the pair of functors
\begin{equation}\label{treska}
\colim_I \mathcal{D}_L\colon\mathscr{C}\rightleftarrows\mathscr{D}\colon\mathcal{D}_R
\end{equation}
is {\it genuine} adjoint.

Furthermore, we have {\it two} colax-monoidal structures on $\colim\mathcal{D}_L$, namely the canonical one $c_1$, adjoint to the lax-monoidal structure $\lim \ell=\ell$ on $\lim \mathcal{D}_R=\mathcal{D}_R$, and the second one $c_2$ is $\colim_I c^\ell(i)$, the colimit of the canonical lax structures for each $i\in I$. The both pairs $(c_1,\ell)$ and $(c_2,\ell)$ make the diagram \eqref{krokodil} commutative, for adjoint pair \eqref{treska}. Indeed, the pair $(c_1,\ell)$ makes \eqref{krokodil} commutative by Lemma \ref{utka}, while the commutativity for $(c_2,\ell)$ follows from Lemma \ref{soroka}.

Apply Lemma \ref{utka} once again (for {\it genuine} adjoint pair $(\colim\mathcal{D}_L,\mathcal{D}_R)$), we conclude that the colax-monoidal structures $c_1$ and $c_2$ on the functor $\colim \mathcal{D}_L$ are isomorphic. We are done.

\end{proof}

\appsubsection{\sc Very weak adjoint pairs}\label{appveryweak}
For our main application to Theorem \ref{keytheorembis}, the assumptions of Theorem \ref{keytheorem2bis} are too strong.
Indeed, in the set-up of Theorem \ref{keytheorembis} one has {\it not} a weak right adjoint pair structure on the pair $(\mathcal{D}_L(i),\mathcal{D}_R)$ (where $i\in I$). In fact, the only what we have are maps of functors $\epsilon(i)\colon \mathcal{D}_L(i)\circ\mathcal{D}_R\to \Id_\mathscr{D}$, but the maps
$\eta(i)\colon \Id_\mathscr{C}\to\mathcal{D}_R\circ \mathcal{D}(i)$ may not exist. See Remark \\ref{sosna} above.

It turns out, fortunately, that we did {\it not} use the existence of the map $\eta(i)$ anywhere in the proof Theorem \ref{keytheorem2bis}.
The goal of this (last) Subsection is to formulate and to prove the corresponding Theorem (see Theorem \ref{keytheorem2bisbis} below) in its proper generality.

\begin{defn}\label{adefn_1}{\rm
Let $L\colon \mathscr{C}\to\mathscr{D}$ and $R\colon \mathscr{D}\to\mathscr{C}$ be functors. We call $(L,R)$ {\it a very weak right adjunction} if there is a morphism of functors
$$
\epsilon\colon LR\to \Id_{\mathscr{D}}
$$
The latter is equivalent to the existence of a map of bifunctors $\Upsilon\colon \Hom_\mathscr{D}(X,RY)\to\Hom_\mathscr{C}(LX,Y)$. 
}
\end{defn}
\begin{defn}\label{adefn_2}
{\rm
Let $L,R$ be a very weak right adjunction, and let a lax-monoidal structure $\varphi$ on $R$ and a colax-monoidal structure $c$ on $L$ are given.
We say that $(c,\varphi)$ are {\it weak right compatible}, if the diagram
\eqref{krokodil} commutes
for any $X,Y$ in $\mathscr{D}$.
}
\end{defn}
It is crucial to notice that the diagram \eqref{krokodil} is written only within $\epsilon\colon LR\to\Id_\mathscr{C}$, not $\eta\colon\Id_\mathscr{C}\to RL$.

A difference between the weak right adjunctions and very weak right adjunctions is that in the very weak case one can not define the {\it canonical} colax-monoidal structure on $L$ out of a lax-monoidal structure on $R$, as \eqref{lax2colax} is not a colax-monoidal structure anymore.

As a substitution, we have the ``intrinsic'' characterization of a pair of a colax-monoidal structure on $L$ and a lax-monoidal structure on $R$ by \eqref{krokodil}, which is the only one remains to use.

We now pass to (co)limits, our goal is to formulate and to prove a result, analogous to Theorem \ref{keytheorem2bis}, in our {\it very weak} generality.
Firstly we suitably adjust Definition \ref{adef3}.

\begin{defn}\label{adefn_3}{\rm
Let $I$ be an indexing category, and let
$\mathcal{D}_L\colon I\to \mathscr{F}un(\mathscr{C},\mathscr{D})^\opp$ and $\mathcal{D}_R\colon I\to \mathscr{F}un(\mathscr{D},\mathscr{C})$ be two diagrams. Suppose that for each fixed $i\in I$ the maps
\begin{equation}
\mathcal{D}_L(i)\colon \mathscr{C}\rightleftarrows\mathscr{D}\colon \mathscr{D}_R(i)
\end{equation}
form a {\it very weak right adjunction}, for some
\begin{equation}
\epsilon(i)\colon \mathcal{D}_L(i)\circ \mathcal{D}_R(i)\to \Id_{\mathscr{D}}
\end{equation}
or, equivalently, for some 
\begin{equation}
\Upsilon(i)\colon \Hom_\mathscr{C}(X,\mathcal{D}_R(i)Y)\to\Hom_\mathscr{D}(\mathcal{D}_L(i)X,Y)
\end{equation}
such that for each morphism $f\colon i\to j$ in $I$, 
the diagram 
\begin{equation}
\xymatrix{
\Hom_\mathscr{C}(X,\mathcal{D}_R(i)Y)\ar[r]^{\Upsilon(i)}\ar[d]_{f_*}&\Hom_\mathscr{D}(\mathcal{D}_L(i)X,Y)\ar[d]^{f_*}\\
\Hom_\mathscr{C}(X,\mathcal{D}_R(j)Y)\ar[r]^{\Upsilon(j)}&\Hom_\mathscr{D}(\mathcal{D}_L(j)X,Y)
}
\end{equation}
commutes, for any $X$ in $\mathscr{C}$ and $Y$ in $\mathscr{D}$.
.
Then the diagrams $\mathcal{D}_L$ and $\mathcal{D}_R$ are called
{\it very weak right adjoint diagrams}.
}
\end{defn}
Suppose very weak right adjoint diagrams
$$
\mathcal{D}_L(i)\colon\mathscr{C}\rightleftarrows\mathscr{D}\colon \mathcal{D}_R(i)
$$
indexed by a category $I$, are given (see Definition \ref{adefn_3}).

Consider the (co)limit pair of functors
$$
\colim \mathcal{D}_L\colon\mathscr{C}\rightleftarrows \mathscr{D}\colon\lim\mathcal{R}
$$
We have
\begin{lemma}{\itshape
Let $\mathcal{D}_L\colon\mathscr{C}\rightleftarrows\mathscr{D}\colon\mathcal{D}_R$ be very weak right adjoint diagrams, see Definition \ref{adefn_3}. Then the (co)limit functors
$$
\colim \mathcal{D}_L\colon\mathscr{C}\rightleftarrows \mathscr{D}\colon\lim\mathcal{R}
$$
form a very weak right adjoint pair of functors.
}
\end{lemma}
The proof simply repeats the proof in weak adjoint case (see \ref{lemmaa} and its proof).

We define a {\it diagram of lax-monoidal structures} and a {\it diagram of colax-monoidal structures}, and their (co)limits, precisely as in
Section \ref{appweak} above.

\begin{apptheorem}\label{keytheorem2bisbis}{\itshape
Suppose $\mathscr{C}$, $\mathscr{D}$ are categories, with $\mathscr{C}$ complete and $\mathscr{D}$ cocomplete. Let $I$ be an indexing category. Let $$\mathcal{D}_L\colon I\to\mathscr{F}un(\mathscr{C},\mathscr{D})^\opp,\ \
\mathcal{D}_R\colon I\to\mathscr{F}un(\mathcal{D},\mathcal{C})$$ are diagrams, such that
$$
\mathcal{D}_L\colon\mathcal{C}\rightleftarrows \mathcal{D}\colon\mathcal{D}_R
$$
is a very weak right adjunction of diagrams, see Definition \ref{adef3}, {\bf with $\mathcal{D}_R(i)$ the constant functor} (not depending on $i\in I$), which we simply denote by $\mathcal{D}_R$.  Suppose there are given a lax-monoidal structure $\ell$ on $\mathcal{D}_R$, and colax-monoidal structures $c(i)$ on $\mathcal{D}_L(i)$, such that (1) for any $i\in I$, the diagram \eqref{krokodil} for $(c(i),\ell)$ on the functors $(\mathcal{D}_L(i),\mathcal{D}_R$ commutes, and (2) the colax-monoidal structures $c(i)$ form a diagram of colax-monoidal structures.
{\bf Suppose that the natural weak right adjunction}
\begin{equation}\label{limit}
\colim\mathcal{D}_L\colon\  \mathscr{C}\rightleftarrows\mathscr{D}\ \colon \lim\mathcal{D}_R=\mathcal{D}_R
\end{equation}
{\bf is a genuine adjunction.}

Then
the colimit colax-monoidal structure
$\colim_{i\in I}c(i)$ on the functor $\colim\mathcal{D}_L$, is isomorphic to the canonical colax-monoidal structure on $\colim\mathcal{D}_L$, induced by the limit lax-monoidal structure ${\lim}_{i\in I}\ell=\ell$ on the functor $\lim\mathcal{D}_R=\mathcal{D}_R$.
}
\end{apptheorem}
\begin{proof}
The proof repeats, with minor variations, the proof of Theorem \ref{keytheorem2bis}.
The main point is that Lemma \ref{soroka} still holds in the very weak context, with literally identical proof.
Then we use Lemma \ref{soroka}, but not for the canonical colax-monoidal structures $c^\ell(i)$ (which no longer make sense in the very weak context), but for the given in the data colax-monoidal structures $c(i)$ on $\mathcal{D}_L(i)$, and this this the only difference.
\end{proof}

\bigskip

\noindent {\sc Max-Planck Institut f\"{u}r Mathematik, Vivatsgasse 7, 53111 Bonn,\\
GERMANY}

\bigskip

\begin{tabbing}
\noindent {\em E-mail address\/}:\quad\= {\tt borya$\_$port@yahoo.com}\\
\noindent {\em Web\/}:\> {\tt http://borya.eu}
\end{tabbing}

\end{document}